\documentclass[]{compositio}

\usepackage[bottom]{footmisc}
\usepackage{xr-hyper}
\usepackage[dvipsnames,table,xcdraw]{xcolor}

\usepackage[colorlinks=true]{hyperref}
	\hypersetup{urlcolor=RubineRed,linkcolor=RoyalBlue,citecolor=ForestGreen}

\usepackage{cooltooltips}
\usepackage{graphicx}

	\newtheorem{theorem}{Theorem}[subsection]
	\newtheorem{lemma}[theorem]{Lemma}
	\newtheorem{proposition}[theorem]{Proposition}
	\newtheorem{corollary}[theorem]{Corollary}
	\newtheorem*{theorem*}{Theorem}
	\newtheorem*{lemma*}{Lemma}
	\newtheorem*{proposition*}{Proposition}
	\newtheorem*{corollary*}{Corollary}
		\theoremstyle{definition}
		
		\newtheorem{definition}[theorem]{Definition}
		\newtheorem{example}[theorem]{Example}

		\newtheorem{remark}[theorem]{Remark}

\usepackage[bottom]{footmisc}
\usepackage{tikz}
\usepackage{tikz-cd}
\usepackage{mathrsfs}
\usepackage{supertabular}
\usepackage[shortlabels]{enumitem}
\usetikzlibrary{lindenmayersystems}
\newcommand{\Gal}{Gal}
\newcommand{\GL}{GL}

\newcommand{\surjto}{\twoheadrightarrow}

\newcommand{\Spec}{Spec} 
\newcommand{\ext}{ext} 
\newcommand{\Hom}{Hom} 
\newcommand{\End}{End} 

\newcommand{\Aff}{\textit{\emph{\text{Aff}}}}

\newcommand{\B}{\mathbf{B}}

\renewcommand{\a}{\mathfrak{a}}
\newcommand{\cA}{\mathcal{A}}
\newcommand{\cF}{\mathcal{F}}

\newcommand{\M}{\mathcal{M}}

\newcommand{\cB}{\mathcal{B}}

\newcommand{\hra}{\hookrightarrow}

\newcommand{\cS}{\mathcal{S}}

\renewcommand{\P}{\mathbf{P}}

\newcommand{\cT}{\mathcal{T}}

\newcommand{\Ad}{\text{Ad}}

\renewcommand{\O}{\mathcal{O}}

\DeclareFontEncoding{OT2}{}{} 

\renewcommand{\H}{\mathbf{H}}

\newcommand{\red}{red}

\newcommand{\aff}{\textit{\emph{\text{aff}}}}

\newcommand{\Hbf}{\mathbf{H}}

\newcommand{\Nbf}{\mathbf{N}}

\newcommand{\Pbf}{\mathbf{P}}
\newcommand{\Zbf}{\mathbf{Z}}
\newcommand{\Mbf}{\mathbf{M}}
\newcommand{\Sbf}{\mathbf{S}}

\newcommand{\der}{der}

\newcommand*\iso{
  \xrightarrow{\raisebox{-0.25 em}{\smash{\ensuremath{\sim}}}}%
}
\usepackage{dsfont}
\usepackage{upgreek}
\usepackage{stmaryrd}
\usepackage{color}
\usepackage{amsthm,amssymb,amsmath}
\usepackage{lastpage}
\usepackage{centernot}
\usepackage{graphicx}
\usetikzlibrary{cd}
\usetikzlibrary{trees}
\usetikzlibrary{fit}

\tikzset{%
  highlight/.style={rectangle,rounded corners,fill=blue!15,draw,fill opacity=0.5,thick,inner sep=0pt}
}

\usepackage{nicematrix}
\usepackage{url}
\usetikzlibrary{hobby,backgrounds,calc,trees}
\usetikzlibrary{matrix,arrows,decorations.pathmorphing}

\def\1{\mathds{1}}

\def\cC{\mathcal{C}}

\def\cO{\mathcal{O}}
\def\cH{\mathcal{H}}

\def\cD{\mathcal{D}}

\def\cM{\mathcal{M}}

\def\A{\mathds{A}}

\def\Z{\mathds{Z}}
\def\N{\mathds{N}}
\def\Q{\mathds{Q}}
\def\R{\mathds{R}}
\def\C{\mathds{C}}

\def\T{\mathbf{T}}

\def\U{\mathbf{U}}
\def\G{\mathbf{G}}
\def\H{\mathbf{H}}

\def\Hom{\text{Hom}}

\def\Spec{\text{Spec}}
\def\Gal{\text{Gal}}

\newcommand{\authnote}[2][]{\noindent {\if!#1!  {\bf TODO} \else {\small \bf #1} \fi: #2}}

\newcounter{tasknumber}[subsection]
\newcommand{\task}[2][]{%
  \addtocounter{tasknumber}{1}%
  \begin{center}%
  \framebox[1.1\width]{\begin{minipage}{0.9\textwidth}%
  \textbf{Task \arabic{tasknumber}} \textit{\if!#1(unassigned)!\else (#1)\fi}: {#2}%
  \end{minipage}}%
  \end{center}%
}

\newcounter{assumptionnumber}
\newcommand{\assumption}[2][]{%
  \addtocounter{assumptionnumber}{1}%
  \begin{center}%
  \framebox[1.1\width]{\begin{minipage}{0.9\textwidth}%
  \textbf{Assumption \arabic{assumptionnumber}} \textit{\if!#1!\else (#1)\fi}: {#2}%
  \end{minipage}}%
  \end{center}%
}

\newtheorem{innercustomtheorem}{Theorem}
\newenvironment{customtheorem}[1]
{\renewcommand\theinnercustomtheorem{#1}\innercustomtheorem}
  {\endinnercustomtheorem}

\newtheorem{innercustomcorollary}{Corollary} \newenvironment{customcorollary}[1]{\renewcommand\theinnercustomcorollary{#1}\innercustomcorollary}{\endinnercustomcorollary}

\newtheorem{innercustomproposition}{Proposition} \newenvironment{customproposition}[1]{\renewcommand\theinnercustomproposition{#1}\innercustomproposition}{\endinnercustomproposition}

\usepackage[scr=boondoxo,scrscaled=1.05]{mathalfa}
\usepackage[intoc,refpage]{nomencl}
\def\@@@nomenclature[#1]#2#3{%
\def\@tempa{#2}\def\@tempb{#3}%
\protected@write\@glossaryfile{}%
{\string\glossaryentry{#1\nom@verb\@tempa @[{\nom@verb\@tempa}]%

nompageref{\begingroup\nom@verb\@tempb\protect\nomeqref{\theequation}}}%
{\thepage}}%
 \endgroup
 \@esphack}

\usepackage[intoc,refpage]{nomencl}
\def\@@@nomenclature[#1]#2#3{%
\def\@tempa{#2}\def\@tempb{#3}%
\protected@write\@glossaryfile{}%
{\string\glossaryentry{#1\nom@verb\@tempa @[{\nom@verb\@tempa}]%

nompageref{\begingroup\nom@verb\@tempb\protect\nomeqref{\theequation}}}%
{\thepage}}%
 \endgroup
 \@esphack}

\usepackage{ifthen}

\usepackage{apptools}
\AtAppendix{\counterwithin{theorem}{section}}

\usepackage{comment}
\usepackage{aligned-overset}
  
\begin{document}

\title{\textsc{A tale of parahoric--Hecke algebras, Bernstein and Satake homomorphisms.}}


\author{Reda Boumasmoud}
\email{reda.boumasmoud@imj-prg.fr \& reda.boumasmoud@gmail.com}
\address{Institut de Mathématiques de Jussieu-Paris Rive Gauche (IMJ-PRG)\\
Sorbonne Université and Université de Paris, CNRS, F-75006 Paris, France.}
\thanks{The author was supported by the Swiss National Science Foundation grant \#P2ELP2-191672.}

\makeatletter 
\@namedef{subjclassname@2020}{%
  \textup{2020} Mathematics Subject Classification}
\makeatother



\date{}

\begin{abstract}
Let $\G$ be a connected reductive group over a {non-archimedean local field} $F$. Let $K_\cF$ be the parahoric subgroup attached to a facet $\cF$ in the Bruhat--Tits building of $\G$. The ultimate goal of the present paper is to describe the center of the parahoric--Hecke algebra $\cH(\G(F)\sslash {K_\cF},\Z[q^{-1}])$ with level $K_\cF$ and prove the compatibility of generalized (twisted) Bernstein and Satake homomorphisms. 
\end{abstract}


 \maketitle


\tableofcontents

\section{Introduction}
Congruences between modular forms, automorphic forms or automorphic representations have received a significant attention in the past few decades for several reasons. 
Their study has helped investigating numerous questions that have been wandering in number theory’s paysage, e.g. first and foremost, Andrew Wiles' proof of Fermat's last theorem.

An omnipresent theme occurring in this landscape are Hecke algebras. 
These are as fundamental as Galois theory is. 
Actually, in more than one respect, the two go hand in hand. 
These special algebras are concerned with the following kind of situation. We are given a reductive group $\G$ defined over a $p$-adic field $F$ and 
a smooth\footnote{Every vector $v\in V$ has an open stabilizer in $\G(F)$.} complex representation of $\G(F)$, say $(V,r)$. If $C\subset \G(F)$ is an open compact subgroup, then we introduce the Hecke algebra $\cH(\G (F)\sslash C,\C)$, which is the convolution of double cosets in $C \backslash \G (F) /C$. 
Similar to how the group algebra of a finite group allows us to translate representation theory questions into module-theoretic ones, $\cH(\G (F)\sslash C,\C)$ provides a mechanism that allows us to dissect the representations $(V,r)$. 
In more concrete terms, we have a bijection between the set of isomorphism classes of irreducible smooth complex $\G(F)$-representations $V$ such that $V^C \neq 0$ and the set of isomorphism classes of simple smooth $\cH(\G (F)\sslash C,\C)$-modules.  

The purpose of this process is to obtain algebraic or number theoretical information which may yield new insight about $\G(F)$ or about its representations theory.
Therefore, it is decisive to further deepen our understanding of Hecke algebras and classify their categories of modules. 

An important step toward this goal is offered by Satake's transform, which yields an explicit description of $\cH(\G(F)\sslash C,\C)$ for two crucial cases: maximal open compact subgroups by Satake \cite{Satake1963} and special maximal parahoric subgroups by Haines--Rostami \cite{HR10}. 
Besides, Hecke algebras for smaller levels also hold a special place and have intriguing applications, for instance when $C=K_\cF$ is a parahoric subgroup, Pappas and Zhu proved in \cite{papzhu2013} a conjecture of Kottwitz which asserts that the semi-simple trace of Frobenius on the nearby cycles of a Shimura variety gives a function which is central in $\cH(\G(F)\sslash K_\cF,\C)$. 
Thanks to the characterization of the minuscule central Bernstein basis functions by Haines \cite{H01}, one can actually compute these traces. 
However, there remains the task of comprehending the center of these parahoric Hecke algebras. 
Such an understanding, is provided by the framework of Bernstein presentations for Iwahori--Hecke algebra when $C=I$ is an Iwahori subgroup, which is due to \cite{Lu89} for the extended affine case and more recently to Rostami \cite{R15} and Vigneras \cite{vigneras_2016} for the general reductive case. 

This paper attempts to extend and generalize to general reductive $p$-adic groups various results known for Hecke algebras for split, unramified or semi-simple and simply connected $p$-adic groups.
In this direction, we are particularly interested in the study of these Hecke algebras when their ring of coefficients is as small as possible; $\Z$ or $\Z[q^{-1}]$. 
Our methods are of geometric flavor and involve extensively Bruhat--Tits theory. 
They provide a uniform treatment for Hecke algebras with parahoric levels and their $\flat$-avatars (\S \ref{flatpartout}). 
This allows us to take full advantage of Lusztig's work \cite{Lu89} and extend his results to the general reductive case. 
Following the ideas of Bernstein \cite{Deligne84}, 
we focus our attention on the universal unramified principal series module, which leads us naturally to the Bernstein elements $\dot{\Theta}_m$ and accordingly the Bernstein presentation of the Iwahori--Hecke algebras. 
The payoff for this, we obtain an explicit description of the center of general parahoric Hecke algebras with coefficients in $\Z[q^{-1}]$, and also integral twisted Satake isomorphism thus generalizing \cite{HR10} and \cite[\S 9]{Satake1963} together with compatibilities between integral twisted Satake isomorphism and twisted Bernstein isomorphism. 
Along the way, we also prove various auxiliary results that may be useful in other contexts. Furthermore, we give as much as we can original proofs for all the previously known results.

A hidden motivation for this paper, is to prepare the ground field for introducing and studying the ring of $\mathbb{U}$-operators, which is the subject of the sequel part II of this paper \cite{UoperatorsII2021}. 
This is the first step of an ongoing project exploring arithmetic application of these operators in Euler systems theory \cite{Unitarynormrelations2020,Vunitarynormrelations2020} and Eichler--Shimura congruence relations for general Shimura varieties \cite{Seedrelations2020,ESboumasmoud}.

\subsection{Main results}
Let $\G$ be a reductive group defined over a {non-archimedean local field} $F$. 
Let $\Sbf$ be maximal $F$-split subtorus of $\G$ and $\cA$ its corresponding apartment in the reduced Bruhat--Tits building of $\G$. 
Set $\Mbf$ and $\Nbf$ for the centralizer and the normalizer of $\Sbf$ in $\G$, respectively. We have a homomorphism of groups $\nu_N \colon \Nbf(F) \to \Aff(\cA)$. 
Let $\a\subset \cA$ be an alcove, and $a_\circ \in \overline{\a}$ a special point. 
Set $I\subset K$ for the Iwahori subgroup and the special maximal parahoric subgroup corresponding to  ${\a}$ and $a_\circ$, respectively. 

Let $M_1 = \ker \kappa_M$ and $G_1=\ker \kappa_G$ be the kernels of the Kottwitz map for $\Mbf$ and $\G$ (see \S \ref{Kottwitzhom}). 
Fix any open compact subgroup $ M_1 \subset M^{\flat}\subset M$. 
By \cite[Proposition 1.2]{Lan96}, $M^{\flat}$ is contained in the maximal open compact subgroup $M^{1}$. 

For any $X\subset G$, set $X^{\flat}=M^{\flat} X$. Whenever we use this notation, the structure of the object we are interested in is not altered, e.g. $K^{\flat}$ or $G_1^{\flat}$ are still a groups.

For any facet $\cF \subset \cA$, set $K_\cF$ for the parahoric subgroup attached to it, $N_{1, \cF}:=\Nbf(F)\cap K_\cF$ and $W_{\cF}^{\flat}=N_{1, \cF}^{\flat}/M^{\flat}$. Because $a_\circ$ is special, $W_{a_\circ}^{\flat}\simeq W:= \Mbf(F)/\Nbf(F)$. 

Set $\Lambda_M^{\flat}:=\Mbf(F)/M^{\flat}$, $\Lambda_{\aff}^{\flat}:=(\Mbf(F) \cap G_1)^{\flat}/M^{\flat}$, $ \mathcal{R}^{\flat}:=\Z[\Lambda_M^{\flat}]$ and $ \mathcal{R}_{\aff}^{\flat}:=\Z[\Lambda_{\aff}^{\flat}]$.

We write $\dot{W}$, for the Weyl group to specify that it acts on $\mathcal{R}^{\flat}$ by the dot-action, this is the standard action twisted by a factor coming form the modulus function (see Definition \ref{dotaction}).

If $A$ is a commutative ring and $H\subset \G(F)$ an open compact subgroup, define $\cH_H(A)$ to be the $A$-algebra of functions $f \colon \G(F) \to A$, that are $H$-invariant on the right and the left.

\subsubsection{The Iwahori--Matsumoto presentation}
We start with a well known result, at least for Iwahori--Hecke attached to extended affine Weyl groups. See for example \cite[\S 6.3]{garrett:buildings}. 
For the general case, it was showed in \cite[Theorem 2.1]{vigneras_2016} or \cite[Proposition 4.1.1]{R15} with $\C$-coefficients. 
We will reprove it here using a different argument.
\begin{customtheorem}{\ref{IMpresentation}}[The Iwahori--Matsumoto Presentation]
The Iwahori--Hecke ring $\cH_{I^{\flat}}(\Z)$ is the free $\Z$-module with basis $\{i_w={\bf 1}_{I^{\flat} w I^{\flat}}, {w\in \widetilde{W}^{\flat}}:= N/M^{\flat}\}$, endowed with the unique ring structure satisfying
\begin{itemize}[topsep=0pt]
\item
The braid relations: 
$i_w^{\flat} *_{I^{\flat}} i_{w'}^{\flat} = i_{ww'}^{\flat} \text{ if }w, w'\in \widetilde{W}^{\flat}\text{ such that }\ell(w) + \ell(w') = \ell(ww')$.
\item The quadratic relations: 
$(i_s^{\flat})^2 = q_s  i_1^{\flat}+ (q_s - 1)i_s^{\flat}\text{ if }s \in \cS^{\flat}$, where $ q_s:=[IsI:I]$ and $\cS^{\flat} \subset \widetilde{W}^{\flat}$ corresponds to the set of orthogonal reflections with respect to the walls of the fixed alcove.
\end{itemize}
\end{customtheorem}
\subsubsection{Freeness of the universal unramified principal series module}

Let $\cF \subset \cA$ be any fixed facet. 
For any commutative ring $A$, define the universal unramified principal series right $(\mathcal{R}^{\flat}, \cH_{K_{\cF}^{\flat}}(A))$-bi-module $\cM_{K_{\cF}^{\flat}}(A) = \cC_c (M^{\flat}U^+\backslash G/K_{\cF},A)$, this is the set of $A$-valued functions supported on finitely many double cosets.

\begin{customproposition}{\ref{HIRbasis}}
The $\mathcal{R}^{\flat}$-module $\cM_{K_\cF^{\flat}}(\Z)$ is free of rank $r_\cF:=|W/\jmath_W(W_\cF^\flat)|$, with basis $\{v_{w,\cF}^{\flat}:={\bf 1}_{U^+ w K_\cF^{\flat}},\, w\in D_\cF^{\flat}\}$ where $D_\cF^{\flat}\subset W_{a_\circ}$ is any fixed set of representatives for $W/ \jmath_W(W_{\cF}^\flat)$ and $\jmath_W\colon N \to W=N/M$ is the quotient map. 
In particular, $\{v_{w,\a}^{\flat},\, w\in W_{a_\circ}^\flat\}$ is a canonical basis for the $\mathcal{R}^{\flat}$-module $\cM_{I^{\flat}}(\Z)$.
\end{customproposition}

\subsubsection{Bernstein presentation of the Iwahori--Hecke algebra}
{Following the ideas of Bernstein \cite{Deligne84}, as exposed by the approach of \cite{HKP10} (which treats the split case)}, we use the universal unramified principal series module $\cM_{I^{\flat}}(\Z)$ to study the Iwahori--Hecke algebra. 
A first result into this direction is:
\begin{customtheorem}{\ref{structureMI}}
The following homomorphism of right $\cH_{I^{\flat}}(\Z[q^{-1}])$-modules is an isomorphism
\begin{align*}\hbar_{I^{\flat}}\colon \cH_{I^{\flat}}({\Z[q^{-1}]}) \longrightarrow \cM_{I^{\flat}}({\Z[q^{-1}]}), \quad h\longmapsto  v_{1,\a}^{\flat}*_{I^{\flat}}h.\end{align*}
\end{customtheorem}

In \S \ref{decomptwistbern}, we construct an embedding $\dot{\Theta}_{Bern}^{\flat}\colon \mathcal{R}^{\flat} \hra \cH_{I^{\flat}}(\Z[q^{-1}]), m \mapsto \dot{\Theta}_{m}^{\flat}$. This provides an alternative description for the Iwahori--Hecke algebra:

\begin{customcorollary}{\ref{HIbasis}}
The algebra $\cH_{{I^{\flat}}}({\Z[q^{-1}]})$ is a free left (and right) $\mathcal{R}^{\flat}\otimes_\Z{\Z[q^{-1}]}$-module, with canonical basis $\{i_w^{\flat} \colon w \in W_{a_\circ}^{\flat}\}$.
The sets $\{\dot{\Theta}_m^{\flat} *_{{I^{\flat}}} i_w^{\flat}\colon m \in \Lambda_M^{\flat}, w \in W_{a_\circ}^{\flat}\}$ 
and $\{ i_w^{\flat} *_{{I^{\flat}}}\dot{\Theta}_m^{\flat} \colon m \in \Lambda_M^{\flat}, w \in W_{a_\circ}^{\flat}\}$ are both ${\Z[q^{-1}]}$-basis for $\cH_{{I^{\flat}}}({\Z[q^{-1}]})$.
\end{customcorollary}

Set $\cH_{I^{\flat}}^{\aff}(\Z):=\cH(G_{1}^{\flat} \sslash I^{\flat},\Z)$ and let $\widetilde{\Upomega}^{\flat}$ denote the subgroup of stabilizers of $\a$ in $\widetilde{W}^{\flat}$. 
The following diagram is commutative (see Proposition \ref{generalizedIWdecomp})
$$
\begin{tikzcd}
\cH_{I}(\Z) \arrow{r}{\simeq}\arrow[two heads, swap]{d}{\iota^{\flat}} & \Z[ \widetilde{\Upomega}]\otimes^\prime \cH_{I}^{\aff}(\Z)\arrow[two heads]{d}{ \square^{\flat} \otimes \iota^{\flat}}  \\
\cH_{I^{\flat}}(\Z) \arrow{r}{\simeq } & \Z[ \widetilde{\Upomega}^{\flat}]\otimes^\prime \cH_{I^{\flat}}^{\aff}(\Z)
\end{tikzcd}$$
where, $\otimes^\prime$ indicates that the tensor modules is endowed with a "untwisted" multiplication: $(\sigma \otimes i_w)(\sigma' \otimes i_{w'})=\sigma\sigma' \otimes i_{\sigma'^{-1}w \sigma'}*_{I^{\flat}} i_{w'} $ for any $\sigma,\sigma' \in \widetilde{\Upomega}^{\flat}$ and any $w,w'\in \widetilde{W}^{\flat}$. 
The restriction of the surjective $\iota^{\flat}$ to $\cH_{I}^{\aff}(\Z)$ is injective (Remark \ref{commuwaffs}) and yields an isomorphism of algebras $\cH_{I}^{\aff}(\Z) \iso \cH_{I^{\flat}}^{\aff}(\Z)$.

This fact, allows us to extend various results due to Lusztig in \cite{Lu89} for the affine extended case ($M^\flat=M^1$) to the general reductive case:  
\begin{customproposition}{\ref{propoTheta}}
For any $s \in  \mathcal{T}_{a_\circ}^{\flat}= \cS^{\flat} \cap W_{a_\circ}^{\flat}$ and any $m\in \Lambda_M^{\flat}$, we have
$$ \dot\Theta_m^{\flat} *_{I^{\flat}}  (i_s^{\flat} +1)-(i_s^{\flat}+1)*_{I^{\flat}}  \dot\Theta_{\dot{s}( m)}^{\flat}= ( \dot\Theta_{m}^{\flat}-\dot\Theta_{\dot{s}( m)}^{\flat})\mathcal{G}(\alpha) \in \mathcal{R}^{\flat}$$
where, $\mathcal{G}(\alpha) \in \mathcal{L}_{\aff}^{\flat}$ is the element defined in \cite[\S 3.8 \& \S 3.13]{Lu89} for the algebra $ \cH(\cD_{\aff},{W}_{\aff},v=q, L)$ (see \ref{rootdatata}).
\end{customproposition}

This result was also obtained by Rostami \cite[Proposition 5.4.2]{R15} for the untwisted $\Theta$-elements, using a slightly different argument.

\subsubsection{Twisted Satake isomorphism with $\Z$-coefficients}

Motivated by arithmetic problems for Shimura varieties, Haines and Rostami established in \cite{HR10}, $\cS_M^G\colon \cH_K(\C)\iso \mathcal{R}^W\otimes_\Z \C$ for the untwisted Satake morphism. 
In \S \ref{sectionsatake}, considering the twisted Satake isomorphism, allows us to extend their result to the Hecke algebra with integral coefficients $\cH_K(\Z)$. 

The twisted Satake transform $\dot{\mathcal{S}}_M^{G^{\flat}} \colon \cH_{K^{\flat}} \to \mathcal{R}^{\flat}$ is characterized by $v_{1,a_\circ}^{\flat}*_{K^{\flat}} h = \dot{\mathcal{S}}_M^{G^{\flat}}(h)\cdot v_{1,a_\circ}^{\flat}, \text{ for all }{h\in \cH_{K^{\flat}}(\Z)}$. 
\begin{customtheorem}{\ref{Ztwistedsatakeisomorphism}}
The twisted Satake homomorphism $\dot{\cS}_M^{G^{\flat}}\colon \cH_{K^{\flat}}(\Z) \to {\mathcal{R}^{\flat}}^{\dot{W}}$ is a canonical isomorphism of ${\Z}$-algebras.
\end{customtheorem}
We propose two proofs for this result; the first is "algebraic" while the second is of "combinatorial" taste. 
Upon completion of the first version of this paper, the author learned that Henniart and Vignéras already proved (differently) this in \cite{HV15}.

We also show the natural generalization of Rapoport's positivity result for unramified groups \cite[Theorem 1.1]{Ra}, see also Haines'  \cite{Haines00}. 
The key tool is due to C. Schwer \cite[Corollary 4.8]{Schwer06}, in which he gives the desired positivity statement for the extended affine case. 

\begin{customtheorem}{\ref{positivitythm}}[Positivity of the twisted Satake isomorphism]
For any $x\in \Lambda_M^-$, write $\dot{\cS}_M^G(h_x)= \sum_{m \in \Lambda_M^-} s_{x,m} r_{m}$. The following statements are equivalent
\begin{enumerate}
\item $s_{x,m}>0$,
\item $U^+m \cap KxK \neq \emptyset$,
\item $U^+m \cap Iwxw'I \neq \emptyset$, for some $w,w' \in W_{a_\circ}$,
\item ${x} \succsim {m}$, i.e. $xm^{-1}\in \Mbf(F)\cap G_1$ and $\nu_N(x{m}^{-1})=\sum n_\alpha \alpha^\vee, n_\alpha \in \N$, the sum being over the simple relative coroots.
\end{enumerate}
\end{customtheorem}

\subsubsection{Centers of parahoric--Hecke algebras}
Using normalized intertwiners in \S \ref{normintertwin}, we prove the following result, which is the key ingredient to describe centers of parahoric--Hecke algebras.
\begin{customtheorem}{\ref{skewringLaff}\& \ref{sgrmatrix}\& \ref{Laffgalois}}
Let $\mathcal{L}_{\aff}^{\flat}$ be the fraction field of $ \mathcal{R}_{\aff}^{\flat}$ and $\mathcal{L}^{\flat}:=\mathcal{L}_{\aff}^{\flat} \otimes_{\mathcal{R}_{\aff}^{\flat}} \mathcal{R}^{\flat} $. 
\begin{enumerate}
\item We have an isomorphism of $\mathcal{L}^{\flat} $-algebras: 
$  \mathcal{L}^{\flat} *\dot{W}\to {\mathcal{L}_{\aff}^{\flat}}^{\dot{W}} \otimes_{{\mathcal{R}_{\aff}^{\flat}}^{\dot{W}}}  \mathcal{H}_{{I}^{\flat}}({\Z[q^{-1}]})$, where $\mathcal{L}^{\flat} *\dot{W}$ is the skew group ring of $\dot{W}$ over $\mathcal{L}^{\flat}$ (See Definition \ref{sgring})
\item The ring extension $\mathcal{L}^{\flat}/{\mathcal{L}^{\flat}}^{\dot{W}}$ is ${\dot{W}}$-Galois.
\item The $\mathcal{L}^{\dot{W}}$-algebra $ A=\mathcal{L}_{\aff}^{\flat} \otimes_{\mathcal{R}_{\aff}^{\flat}}  \mathcal{H}_{{I}^{\flat}}({\Z[q^{-1}]})$ is compressible, that is $Z(e A e) = eZ(A)=e\mathcal{L}^{\dot{W}}$
for any idempotent $e$ of $A$.
\end{enumerate}
\end{customtheorem}
With this result in hand, we generalize Haines's \cite[Theorem 3.1.1]{Haines09}  which treats the case of unramified $\G$: 
\begin{customtheorem}{\ref{center}}
For any facet $\cF\subset \overline{\a}$, we have
$$Z\left(\cH_{K_{\cF}^{\flat}}({\Z[q^{-1}]})\right)=\dot{\Theta}_{\text{Bern}}^{\flat}({\mathcal{R}^{\flat}}^{\dot{W}}\otimes_\Z\Z[q^{-1}])*_{I^{\flat}}{\bf 1}_{K_\cF^{\flat}}.$$
In particular, the set $\{r_m^{\flat}*_{I^{\flat}}{\bf 1}_{K_\cF^{\flat}}: m \in \Lambda_M^{-,\flat}\}$ forms a basis for the $\Z[q^{-1}]$-module $Z\left(\cH_{K_{\cF}^{\flat}}({\Z[q^{-1}]})\right)$.
\end{customtheorem}
{
{I would like to point out that once the first version of this article was completed, the author learned of Haines's paper \cite{Haines2014}. 
Within the Appendix of {\em loc. cit.}, 
some general facts about the center of parahoric Hecke algebras (i.e. $M^{\flat}=M_1$) with $\C$-coefficients, compatibility with constant terms and change of parahorics are treated based on the theory of Bushnell--Kutzko types. 
In {\em loc. cit.}, Haines focuses on Iwahori levels ($\cF=\a$) and affirms the passage to general parahoric levels should rely on other properties of intertwiners for principal series representations as in \cite{HKP10}. 
However, the only written reference for these properties that the author is aware of treat only the split case \cite{HKP10} or the unramified case \cite{Haines07}. 
The more elementary approach of the present paper treats all parahoric groups and their $\flat$-avatars directly and works with integral coefficients. 
}}
\subsubsection{Compatiblity of twisted Satake and Bernstein morphisms for parahoric levels}
For any facet $\cF \subset \overline{\a}$, define the generalized twisted Satake transform $\dot\cS_{\cF}^{\flat}\colon Z(\mathcal{H}_{K_{\cF}^{\flat}}({\Z[q^{-1}]})) \to {\mathcal{R}^{\flat}}^{\dot{W}}\otimes_\Z \Z[q^{-1}]$ to be the morphism of $\Z[q^{-1}]$-algebras that is characterized by $\dot\cS_{\cF}^{\flat}(h)*_{I^{\flat}}{\bf 1}_{K_{\cF}^{\flat}} = h$. 
When $\cF=\{a_\circ\}$, then $\dot\cS_{\cF}^{\flat}=\dot\cS_M^{G^{\flat}}$.
\begin{customtheorem}{\ref{compatibility}}
Let $\cF \subset \overline{\cF}' \subset  \overline{\a}$, be two facets, i.e. $K_{\cF}\supset K_{\cF'}\supset I$. 
The following diagram of $\Z[q^{-1}]$-algebras is commutative:
$$\begin{tikzcd}  Z(\cH_{K_\cF^{\flat}}(\Z[q^{-1}])) \arrow[]{rr}{\dot{\cS}_{\cF}^{\flat} }[swap]{\simeq}&& {\mathcal{R}^{\flat}}^{\dot{W}}\otimes_\Z\Z[q^{-1}]\arrow[]{d}{\dot\Theta_{Bern}^{\flat}}[swap]{\simeq}\\
Z(\cH_{K_{\cF'}^{\flat}}(\Z[q^{-1}])) \arrow[]{u}{-*_{K_{\cF'}^{\flat}}{\bf 1}_{K_\cF}^{\flat}}[swap]{\simeq} &&Z(\cH_{I^{\flat}}(\Z[q^{-1}]))\arrow{ll}{-*_{I^{\flat}}{\bf 1}_{K_{\cF'}^{\flat}}}[swap]{\simeq}
\end{tikzcd}$$
\end{customtheorem}
This commutative diagram is also compatible with change of $\flat$-avatars $M^\flat \subset {M^{{\flat}^\prime}}$ (See remark \ref{transitionhecke}). 
To our knowledge the first proof of the compatibility between the Satake and Bernstein isomorphism, i.e. commutativity of the diagram above is due to J.-F. Dat \cite[\S 3]{Dat99} for split reductive groups and for a fixed Iwahori level, i.e. $\cF=\cF'=\{a_\circ\}$. 

\subsection{Acknowledgements}
Some parts of this article originated (partially) from my doctoral thesis, directed by Dimitar Jetchev, to whom I am very grateful. I am thankful to Christophe Cornut for his support and meticulous reading, Anne-Marie Aubert and Marie-France Vigneras for many helpful comments and suggestions and Thomas Haines, for constructive advice and pointing me to his paper \cite{Haines2014}.

\section{A review of Bruhat--Tits theory}
Bruhat and Tits made in their seminal work \cite{BT72,BT84} a profound exploration of reductive groups over local fields by constructing for them a combinatorial avatar; for a reductive group $\G$ over a local field $F$, they associate to $\G^{ad}(F)$ an affine building $\cB(\G,F)_{\red}$ and to $\G(F)$ a building $\cB(\G,F)_{\ext}$. 
The first one is called the reduced Bruhat--Tits building, the latter one is called the extended Bruhat--Tits building. 
Here is a gentle mise-en-bouche: the building $\cB(\G,F)_{\red}$ is a complete metric space, that has a structure of a poly-simplicial complex. 
This building is obtained by "gluing" a family of distinguished subsets, called apartments, these are affine spaces for some real affine space. 
This inner poly-simplicial structure comes with an isometric action of the group $\G(F)$. 

The purpose of this section is to extract a reasonably brief exposition of Bruhat--Tits theory and expose the properties that will be needed in the proofs of our results. 
Moreover, most of the notations introduced for this goal will be used later. 

\subsection{Notations}\label{notationsCh3}
\begin{itemize}[nosep]
\item $F$ a finite extension of $\Q_p$ for some prime $p$, $\O_F$ its ring of integers, $\varpi$ a fixed uniformizer in $\O_F$, $\kappa(F)$ its residue field and $q$ its cardinality,

\item $\omega\colon F^\times \to \Z$ the normalized discrete valuation, i.e. $\omega(\varpi)=1$ and $|\cdot|_F=q^{-\omega(\cdot)}$ for the corresponding normalized absolute value,
\item $F^{sep}$ a fixed separable closure,
\item For any connected reductive $F$-group $\Hbf$, there exists a homomorphism of groups $\nu_{H}\colon \H(F) \to \Hom_\Z (X^*(\H)_F,\Z)$ characterized by $$\langle \nu_{H}(h) , \chi \rangle= \nu_{H}(h) (\chi)=-\omega(\chi(h)),\text{ for all }h\in \H(F)   \text{ and } \chi\in X^*(\H)_F.$$
\item Fix $\G$ a connected reductive group over $F$, we use the convention that all reductive groups are connected,
\item Let $\Zbf_c$ denote the maximal central F-torus of $\G$ and $\Zbf_{c,sp}$ its maximal $F$-split $F$-subtorus,
\item Let $\mathbf{S}$ be a maximal $F$-split subtorus of $\G$, 
\item For any algebraic subgroup $\H\subset \G$, we will denote by $N_\G(\H)$ and $Z_\G(\H)$ the normalizer and the centralizer of $\H$ in $\G$ respectively, 
\item The root system of $\G$ with respect to $\Sbf$ will be denoted $\Phi:=\Phi(\G,\mathbf{S})$, $\Phi_{\red}=\{\alpha \in \Phi\colon \frac{1}{2} \alpha \notin \Phi\}$, $\Phi^+$ a system of positive roots in $\Phi$, $\Delta$ the associated base 
of simple roots, $\Delta^{\vee}$ its dual and the Weyl group $W=W(\G,\mathbf{S})(F)=N_\G(\Sbf)(F)/Z_\G(\Sbf)(F)$, 
\item Let $\Psi$ be a closed\footnote{A subset $\Psi \subset\Phi_+ \cap \Phi_{\red}$ such that for any $\alpha,\beta \in \Psi$ one has $[\alpha,\beta]:=\{n\alpha+m \beta\colon \text{ for all }n,m \in\Z_{>0} \}\cap \Phi \subset \Psi$.} 
subset of $\Phi_{\red}^+:=\Phi^+ \cap \Phi_{\red}$, e.g. $\Psi=\{\alpha,2\alpha\}$ if $2\alpha \in \Phi$. There exists a unique closed, connected, unipotent $F$-subgroup $\U_\Psi\subset \G$, that is normalized by $ Z_\G(\Sbf)$ and such that the product morphism
\begin{equation}
\prod_{\alpha\in \Psi\cap \Phi_{\red}} \U_\alpha\to \U_\Psi,\tag{$\ast$}\end{equation}
is an isomorphism of $F$-varieties, here the product is taken in any order for $\Psi$ 
\item Any subset $J\subset \Delta$ is a base for a root system $\Phi_J:=\Z J \cap \Phi$. 
The subgroup $\langle Z_\G(\Sbf), \U_{\Phi_J\cup\Phi^+}\rangle$ of $ \G$ is a standard parabolic subgroup, we denote it by $\Pbf_J$. 
We have $R_u(\Pbf_J)=\U_{\Phi^+\setminus  \Phi_J^+}$\label{parabolicJ}, where $\Phi_J^\pm:= \Z J\cap  \Phi^\pm$.
\item The Levi factor $\Mbf_J$ of $\Pbf_J$ is the subgroup generated by $\langle Z_\G(\Sbf),\U_{\Phi_J}\rangle$. Actually, this Levi subgroup is the centralizer of the split torus $\Sbf_J:=(\cap_{\alpha\in \Phi_J}\ker \alpha)^{\circ}$. When $J=\emptyset$ we get a minimal parabolic subgroup $\Pbf_\emptyset$ with Levi factor $Z_\G(\Sbf)$. If $J=\Delta$, then $\Pbf_\Delta=\Mbf_\Delta=\G$ and $\Sbf_\Delta$ is the unique maximal split torus of the center $Z_\G$. We refer the reader to \cite[\S21]{Bor91} for more details.
\item Set $\Mbf:=\Mbf_\emptyset=Z_\G(\Sbf)$ for the centralizer (a minimal Levi $F$-subgroup of $\G$), 
$\Nbf:=N_\G(\Sbf)$ for the normalizer and $\B:=\Pbf_\emptyset$ for the minimal parabolic with Levi factor $\Mbf$ and unipotent radical $ \U^+=\U_{\Phi^+}$, we have $\B=\U_{\Phi^+} \rtimes \Mbf$.
\item We will sometimes use the notation $\square\otimes k=\square_k$ for a field $k$.
\end{itemize}

\subsection{The standard apartment}\label{emptyaprtment}
We have a commutative diagram
$$\begin{tikzcd}[column sep=small]
X_*(\Mbf)_F\arrow[equal]{d}{} 
  & \times 
  & X^*(\Mbf)_F  \arrow[hookrightarrow]{d}  \arrow{rr}
  && \Z \arrow[equal]{d} 
 \\
  X_*(\Sbf)
  & \times
  &X^*(\Sbf) \arrow{rr}{}
  && \Z
\end{tikzcd}$$
where only the bottom line is a perfect pairing. 
In the following lemma we will show that, the injective finite-cokernel group homomorphism $X^*(\Mbf)_F \hookrightarrow X^*(\Sbf)$ induces a unique homomorphism $\Mbf(F)\to X_*(\Sbf)\otimes_\Z\Q$ that extend $\nu_{M}\colon \Mbf(F) \to \Hom_\Z (X^*(\Mbf)_F,\Z)$.
\begin{lemma}\label{nu}There exists a unique homomorphism $\nu_{M}\colon \Mbf(F)\to X_*(\Sbf)\otimes_\Z\Q$\nomenclature[B]{$\nu_{M}$}{Translation action of $\Mbf(F)$ on extended apartment $X_*(\Sbf)\otimes \R$} such that $$\langle \nu_{M}(z),\chi|_{\Sbf}\rangle=-\omega(\chi(z)),\text{ for all }z\in \Mbf(F)   \text{ and } \chi\in X^*(\Mbf)_F (\hookrightarrow X^*(\Sbf)).$$\end{lemma}
\begin{proof}
Because the restriction map $X^*(\Mbf)_F\to X^*(\Sbf)$ is injective and has an image of finite index, i.e. for any $\chi \in X^*(\Sbf)$, the character  $[X^*(\Sbf):X^*(\Mbf)_F]\chi\in X^*(\Sbf)$ extends to a unique $\widetilde{\chi}\in X^*(\Mbf)_F$. 
Consider the following map $$M(F) \to \Hom(X^*(\Sbf)\otimes_\Z \Q,\Q), \quad z \mapsto \left(\chi\otimes r \mapsto -r\frac{1}{[X^*(\Sbf):X^*(\Mbf)_F]} \omega(\widetilde{\chi}(z))\right).$$
This is clearly a homomorphism of groups. 
Moreover, $X^*(\Sbf)\otimes_{\Z}\Q$ and $X_{*}(\Sbf)\otimes_{\Z}\Q$ remain a dual pair of $\Q$-modules for the canonical pairing $\langle\, , \rangle_{\Q}$ that extend the pairing $\langle\,,\rangle \colon X_{*}(\Sbf)\times X^{*}(\Sbf)\to \Z$. 
Thus, any homomorphism in $\Hom(X^*(\Sbf)\otimes_\Z \Q,\Q)$ is of the form $\chi\otimes r\mapsto \langle \chi, \lambda\rangle rr'$ for some $\lambda\otimes r' \in X_*(\Sbf)\otimes_\Z \Q$. 
This defines a unique homomorphism of groups $\nu_{M} \colon \Mbf(F) \to X^*(\Sbf)\otimes_\Z \Q$ verifying
$$\langle\nu_{M}(z) , \chi\otimes 1 \rangle_{\Q}=   -\frac{1}{[X^*(\Sbf):X^*(\Mbf)_F]} \omega(\widetilde{\chi}(z)) \quad \text{for all } z\in \Mbf(F) \text{ and } \chi \in X^*(\Sbf),$$
or equivalently (since $X^*(\Mbf)_F$ generates $X^*(\Sbf)\otimes_\Z \Q$)
\begin{align*} \langle\nu_{M}(z) , \chi_{|_{\Sbf}}\otimes 1 \rangle_{\Q}&= -\omega({\chi}(z))\quad\text{for all } z\in \Mbf(F) \text{ and } \chi \in X^*(\Mbf)_F.\qedhere\end{align*}
\end{proof}
By \cite[Proposition 1.2]{Lan96}, the subgroup $\Mbf(F)^1:=\ker \nu_{M}$\nomenclature[B]{$\Mbf(F)^1$}{The kernel of $\nu_M$} is the maximal compact open subgroup of $\Mbf(F)$. 
We then have a short exact sequence of groups
$$0\to \Mbf(F)/\Mbf(F)^1\to \Nbf(F)/\Mbf(F)^1 \to W= \Nbf(F)/\Mbf(F)\to 0,$$
in which $\Mbf(F)/\Mbf(F)^1$ is a free abelian group containing $X_*(\Sbf)$ and having same rank \cite[Lemma 1.3]{Lan96}.

Consider the $\R$-vector space $V=X_*(\Sbf/\Zbf_{c,sp})\otimes_\Z\R$\nomenclature[B]{$V$}{$X_*(\Sbf/\Zbf_{c,sp})\otimes_\Z\R$} and identify its dual space $V^{*}$ with $X^{*}(\Sbf/\Zbf_{c,sp})\otimes_{\Z}\R$ using the canonical extended pairing $\langle\, , \rangle_{\R}$. 
We denote by $$\nu\colon \Mbf(F) \to V,$$\nomenclature[B]{$\nu$}{Translation action of $\Mbf(F)$ on $V$} the composition of the map $\nu_{M}\colon\Mbf(F) \rightarrow X_{*}(\Sbf)\otimes_{\Z}\R$ with the natural projection $X_{*}(\Sbf)\otimes_{\Z}\R \twoheadrightarrow V$. 

The action of $W$ on $\Sbf$ by conjugation induces a faithful linear action on $X_*(\Sbf)$, which induces a canonical homomorphism of groups
$$j\colon W\to \GL(X_*(\Sbf)\otimes_\Z\R),$$ 
and accordingly a homomorphism $W\to \GL(V)$ since $\Nbf(F)$ acts trivially on $\Zbf_{c,sp}$.
\begin{proposition}\label{A(G,S)}
There is a canonical affine space $\A_{\red}(\G,\Sbf)$\nomenclature[B]{$\A_{\red}(\G,\Sbf)$}{Reduced standard apartment} under $V$ (unique up to a unique isomorphism of affine spaces) together with a group homomorphism $\nu_N\colon \Nbf(F)\to \Aff(\A_{\red}(\G,\Sbf))$\nomenclature[B]{$\nu_{N}$}{Extension of $\nu$ to the action of $\Nbf(F)$ on $\A_{\red}(\G,\Sbf)$} extending $\nu$. Accordingly, we have a commutative diagram
$$\begin{tikzcd}
     0  \arrow{r}  
  & \Mbf(F)/\Mbf(F)^1\arrow{r}\arrow[]{d}{\nu}
  & \Nbf(F)/\Mbf(F)^1 \arrow{r}\arrow[]{d}{\nu_{N}}
  & W\arrow[hook]{d}{j}\arrow{r}
  & 0 
  \\
  0  \arrow{r}  
  & V\arrow{r}
  & \Aff(\A_{\red}(\G,\Sbf)) \arrow{r}
  & \GL(V)\arrow{r}
  & 0
\end{tikzcd}$$
For all $z\in \Mbf(F)$, $\nu_N(z)$ is the translation $a \mapsto \nu(z)+a$ and for any $n\in \Nbf(F)$ the linear part of $\nu_N(n)$ is equal to $j(\jmath_W(n))$, where $\jmath_W(n)$ denotes the image of $n$ in $W$.
\end{proposition}
\begin{proof}See the proof of \cite[Proposition 1.8]{Lan96}.
\end{proof}
\begin{remark}
(i) The kernel $\ker \nu_N$ contains $\Mbf(F)^1$ and the center of $\G(F)$. 
(ii) The reductivity of $\G$ forces the action of the Weyl group $W$ on $V_0= X_*(Z_\G)\otimes_\Z \R$ to be trivial, which implies the injectivity of the morphism $W \hra \GL(V)$ and accordingly 
$\ker \nu_N= \ker \nu$.
(iii) All possible extensions of $\nu$ are of the form
$$\nu_{N,v}(n)\colon a \mapsto -v+\nu_N(n)(v+a),\, \forall a\in \A_{\red}(\G,\Sbf),$$
for some fixed $v\in V$. 
\end{remark}
We are now ready to define the central combinatorial object of this section
\begin{definition}The affine space $\A_{\red}(\G,\Sbf)$ together with the group homomorphism $\nu_N\colon \Nbf(F)\to \Aff(\A_{\red}(\G,\Sbf))$ is called the standard apartment of $\G$ with respect to $\Sbf$.
\end{definition}
\subsection{A discrete valuation of the generating root datum}\label{discretevaluation}
Define for each $\alpha\in \Phi$ the set $M_\alpha$ of elements $n\in \Nbf(F)$ whose image in the Weyl group $\jmath_W(n) \in W$ is the reflection $s_\alpha\colon V \to V , v \mapsto v-\langle v,\alpha^\vee \rangle \alpha $, where (by abuse of notation) $\alpha^\vee$ denotes the image of the co-root in $V$.

A good part of the theory developed by Borel and Tits in \cite{BR65} (see \cite[\S 6.1]{BT72}) may be condensed in the following fundamental theorem:
\begin{theorem}\label{generating}
The datum $\left(\Mbf(F),\left(\mathbf{U}_\alpha(F), M_\alpha\right)_{\alpha\in \Phi}\right)$ is a generating root datum 
of type $\Phi$ in $\G(F)$.\end{theorem}
The nomenclature "generating root datum" means that the above datum verifies the list of axioms \cite[\S 6.1.1]{BT72}.
\begin{corollary}\label{m(u)}
For any $\alpha \in \Phi$ and any $u\in \U_{\alpha}(F)\setminus\{1\}$ there exists a unique pair $$(u',u'') \in {\U_{-\alpha}(F)\setminus\{1\} \times \U_{-\alpha}(F)\setminus\{1\}}$$ such that $m_\alpha(u):=u'uu''\in \Nbf(F)$ and verifies $m_\alpha(u) \U_{\alpha}(F) m_\alpha(u)^{-1}=\U_{-\alpha}(F)$. 

In particular, $\U_{-\alpha}(F) u \U_{-\alpha}(F) \cap \Nbf(F)=\{m_{\alpha}(u)\}$. 
Moreover, $m_\alpha(u)$ normalizes $\Sbf$ and its image in $W$ acts on $V$ by the reflexion $s_\alpha$\nomenclature[B]{$s_\alpha$}{Reflexion on $V$ associated with $\alpha\in \Phi$}. 
Accordingly, the set $M_\alpha$ defined above Theorem \ref{generating} is precisely the right coset $m_{\alpha}(u)\Mbf(F)$.
\end{corollary}
\proof See \cite[\S 6.1.2]{BT72}.
\qed

{Fix any point (origin) $a_\circ$ in $\A_{\red}(\G,\Sbf)$. For all $\alpha \in \Phi$ and all $u\in \U_{\alpha}(F)\setminus\{1\}$ we have
$$\nu_N(m_\alpha(u))(a)=s_\alpha(a-a_{\circ}) + \nu_N(m_\alpha(u))(a_\circ)= -\alpha(a-a_{\circ}) \alpha^\vee+a-a_{\circ} +  \nu_N(m_\alpha(u))(a_\circ)$$
The set of fixed points of $\nu_N(m_\alpha(u))$ is an affine hyperplane in $\A_{\red}(\G,\Sbf)$ with direction $\ker(s_\alpha - \text{Id}_V)$. Let $b_{a_\circ,\alpha,u}$ be any fixed element of this hyperplane, one has
$$ \nu_N(m_\alpha(u))(a_\circ)=\alpha(b_{a_\circ,\alpha,u}-a_{\circ}) \alpha^\vee+a_{\circ}.$$Therefore,
$$\nu_N(m_\alpha(u))(a)=s_\alpha(a-a_{\circ}) +\alpha(b_{a_\circ,\alpha,u}-a_{\circ}) \alpha^\vee+a_{\circ}.$$
Now, by setting $\varphi_\alpha^{a_\circ}(u):=-\alpha(b_{a_\circ,\alpha,u}-a_{\circ})\in \R$ one can rewrite $\nu_N(m_\alpha(u))$ as follows:
$$\nu_N(m_\alpha(u))(a)
=-\left(\alpha(a-a_{\circ})+\varphi_\alpha^{a_\circ}(u)\right)\alpha^\vee + a.$$
Accordingly, the element $\nu_N(m_{\alpha}(u))$ acts as a reflection with respect to the hyperplane $\{a\in \A_{\red}(\G,\Sbf)\colon \alpha(a-a_{\circ})=-\varphi_\alpha^{a_\circ}(u) \}$ (See \cite[Remark (b) 6.2.12]{BT72}).}

One of the fundamental results obtained by Bruhat and Tits is 
\begin{theorem}The family $\varphi^{a_\circ}=(\varphi_\alpha^{a_\circ}\colon \U_\alpha(F) \to \R \cup \{\infty\})_{\alpha \in \Phi}$\nomenclature[B]{$\varphi_\alpha$}{} is a discrete valuation of the generating root datum $\left(\Mbf(F),\left(\mathbf{U}_\alpha(F), M_\alpha\right)_{\alpha\in \Phi}\right)$, 
meaning that $\varphi^{a_\circ}$ verifies the properties listed in \cite[\S 6.2.1 \& \S 6.2.21]{BT72}.
\end{theorem}
\begin{proof} 
For a proof we refer the reader to \cite[\S 5.1.20 \& \S 5.1.23]{BT84} and \cite[Remarque 6.2.12 (b)]{BT72}.\end{proof}
Consequently, for all $\alpha \in \Phi$ and $r \in \R\cup \{\infty\}$, the set $U_{\alpha+r}^{a_\circ}:=(\varphi_\alpha^{a_\circ})^{-1}([r,\infty])$ is a compact open subgroup of $\U_\alpha(F)$ and $U_{\alpha+ \infty}^{a_\circ}=\{1\}$. Moreover, if $2\alpha \in \Phi$ then $\varphi_{2\alpha}^{a_\circ}$ is the restriction of $2\varphi_\alpha^{a_\circ}$ to $U_{2\alpha}$. 

For any $v\in V$, the map $a_\circ \mapsto v+ a_\circ $ yields the map $$\varphi^{a_\circ} \mapsto v+ \varphi^{a_\circ}:=\varphi^{v+a_\circ}=\left(\varphi^{v+a_\circ}_\alpha \colon  \U_\alpha(F) \to \R \cup \{\infty\}, u \mapsto \varphi^{a_\circ}(u) + \alpha(v)\right)_{\alpha \in \Phi}.$$
Two discrete valuations in $\mathds{A}_{\varphi^{a_\circ}}:=\{v+ \varphi^{a_\circ} \colon v \in V\}$ are called equipollent. As noted by \cite[6.2.6]{BT72}, the action of $V$ on $\mathds{A}_{\varphi^{a_\circ}}$ described above endow this latter with a structure of an affine space under $V$.
\begin{remark}\label{NactionAphi}Under the isomorphism of affine spaces $\mathds{A}_{\varphi^{a_\circ}} \simeq \A_{\red}(\G,\Sbf)$, the action of $N$ on $\A_{\red}(\G,\Sbf)$ described in Proposition \ref{A(G,S)} corresponds to the following action of $N$ on $\mathds{A}_{\varphi^{a_\circ}}$:
$$N \times \mathds{A}_{\varphi^{a_\circ}} \to \mathds{A}_{\varphi^{a_\circ}}, \quad [n, v+ \varphi^{a_\circ}]\mapsto \jmath_W(n)(v)+ \varphi^{\nu_N(n)(a_\circ)}.$$
\end{remark} 
For any $\alpha \in \Phi$, the valuation $\varphi^{a_\circ}$ is discrete; $\Gamma_\alpha(a_\circ):=\varphi_{\alpha}^{a_\circ}( \U_{\alpha}(F)\setminus\{1\})$ is a discrete subset of $\R$ \cite[\S 6.2.21]{BT72} and $\Gamma_\alpha(a_\circ)=\Gamma_{-\alpha}(a_\circ)$. 
Set 
$$\Gamma_\alpha'(a_\circ):=\{\varphi_{\alpha}^{a_\circ}(u) \in \Gamma_\alpha(a_\circ) \colon \varphi_{\alpha}^{a_\circ}(u)=\sup \varphi_{\alpha}^{a_\circ}(u \U_{2\alpha}(F))  \}$$
(If $2\alpha \not\in \Phi$ we put $U_{2 \alpha}=\{1\}$, i.e. $\Gamma_\alpha'(a_\circ)=\Gamma_\alpha(a_\circ)$).
We have
$\Gamma_\alpha(a_\circ)=\Gamma_\alpha'(a_\circ) \cup \frac{1}{2}\Gamma_{2\alpha}(a_\circ)$ \cite[\S 6.2.2]{BT72}.
\begin{remark}
The set $\Gamma_\alpha'(a_\circ)$ can be empty, in our discrete case, this happens if and only if $U_{2 \alpha}$ is dense in $U_{\alpha}$ \cite[\S 6.2.2]{BT72}.
\end{remark}
\begin{definition}We say that a discrete valuation
$\varphi^{a} \in \mathds{A}_{\varphi^{a_\circ}}$ (for some $a\in \mathds{A}_{\red}(\G,\Sbf)$) is special if, for any $\alpha \in \Phi_{\red}$, one has $0 \in \Gamma_\alpha(a)$.
\end{definition}
\begin{remark}\label{specialpointvaluation}
By definition $\varphi^{a} = (\varphi_\alpha^{a}\colon u \mapsto -\alpha(b_{a,\alpha,u}-a))$, we see then that $\varphi^{a}$ is special if and only if for any $\alpha \in \Phi_{\red}$, there is a $u_\alpha \in \U(F)$ such that $a=b_{a,\alpha,u_\alpha}$, this is equivalent to saying that $a$ is fixed by $m_\alpha(u_\alpha)$ for all $\alpha \in \Phi_{\red}$, in which case we will say that $a$ is a special point of $\mathds{A}_{\red}(\G,\Sbf)$.
\end{remark}
\begin{lemma}
There exists a special discrete valuation in $\mathds{A}_{\varphi^{a_\circ}}$.
\end{lemma}
\proof This is \cite[Corollaire 6.2.15]{BT72}.\qed

By the previous lemma we may and will assume from now on that $a_\circ$ is a special point. From now on we will omit indicating $a_\circ$ in the notation of $\varphi^{a_\circ},\Gamma_\alpha(a_\circ), \Gamma_\alpha'(a_\circ) ,\varphi_\alpha^{a_\circ}$ and $U_{\alpha+r}^{a_\circ}$ for $\alpha \in \Phi, r \in \R$, i.e. we will write $\varphi, \varphi_\alpha$ and $U_{\alpha+r}$ ...
\begin{proposition}\label{nalpha}
For any $\alpha \in \Phi$, there exists a positive integer $n_\alpha\nomenclature[B]{$n_\alpha$}{}$ such that $\Gamma_\alpha= n_\alpha^{-1}\Z$, which satisfies the following properties: 
$n_{\alpha}=n_{-\alpha}$, $n_{w(\alpha)}=n_\alpha$ for any $w\in W=W(\Phi)$ 
and $n_{\alpha}$ is divisible by $n_{2\alpha}$ if $\alpha,2\alpha\in \Phi$.
\end{proposition}
\begin{proof}
This follows from \cite[Lemma I.2.10]{SchUlr97} and \cite[6.2.23]{BT72}. 
\end{proof} 
\subsection{Affine roots}\label{affinerootssection}
In the notation $U_{\alpha+r}$ for any $\alpha \in \Phi$ and any $r \in \Gamma_\alpha$, $\alpha +r$ should be interpreted as the closed half-space of $\A_{\red}(\G,\Sbf)$, on which the affine mapping $a \mapsto \alpha(a-a_{\circ})+r$ is positive. 
\begin{definition}\label{affineroots}
Let $\Phi_{\aff}$ denote the set of all affine roots $\left\{\alpha + r \colon \alpha \in \Phi,  r \in \Gamma_\alpha' \right\}$ \cite[6.2.6]{BT72}.
\end{definition}
For use later, we describe in the following lemma the action of $\Nbf(F)$ on root groups $U_{\alpha+r}$ for affine roots.
\begin{lemma}\label{NactionUalphar}
Let $n\in \Nbf(F)$, with image $w$ in $W$. 
We have for all $\alpha \in \Phi$ and all $r\in \Gamma_\alpha'$:
$$nU_{\alpha+ r}n^{-1}=U_{\beta},$$	
where, $\beta= w(\alpha)+ r- w(\alpha)(\nu_N(n)(a_\circ)-a_\circ)$. In particular, $mU_{\alpha+ r}m^{-1}=U_{\alpha+r- \langle \nu(m),\alpha \rangle}$ for  $m \in \Mbf(F)$.
\end{lemma}	
\begin{proof}Let $\alpha+r \in \Phi_{\aff}$ and $n\in N$ of image $w$ in $W$. As suggested by the proof of \cite[6.2.10. (iii)]{BT72}:
\begin{align*}
nU_{\alpha+r} n^{-1}&=(\varphi_\alpha)^{-1}([r,\infty])\\
&=n \{u \in \U_\alpha(F)\colon \varphi_\alpha(u)\ge r\} n^{-1}\\
&= \{u \in \U_{w(\alpha)}(F)\colon \varphi_\alpha(n^{-1}un)\ge r\}\\
&=\{u \in \U_{w(\alpha)}(F)\colon (\nu_N(\varphi))_{w(\alpha)}(u)\ge r\} \\
\overset{\text{Remark \ref{NactionAphi}}}&{=}\{u \in \U_{w(\alpha)}(F)\colon (w(0_V)+\varphi^{\nu_N(a_\circ)})_{w(\alpha)}(u)\ge r\}\\
&=\{u \in \U_{w(\alpha)}(F)\colon (\varphi^{\nu_N(a_\circ)})_{w(\alpha)}(u)\ge r\} \\
\overset{\varphi^{v+a_\circ}=v+\varphi^{a_\circ}}&{=}\{u \in \U_{w(\alpha)}(F)\colon (\nu_N(n)(a_\circ)-a_\circ + \varphi)_{w(\alpha)}(u)\ge r\} 
\\
&=\{u \in \U_{w(\alpha)}(F)\colon   \varphi_{w(\alpha)}(u)+ w(\alpha)(\nu_N(n)(a_\circ)-a_\circ) \ge r\}\\
&=U_{\beta= w(\alpha)+ r- w(\alpha)(\nu_N(n)(a_\circ)-a_\circ)}\qedhere
\end{align*}
\end{proof}
For each affine root $\alpha+r\in \Phi_{\aff} $ consider the affine hyperplane $H_{\alpha+r}:=\big\{a\in \A_{\red}(\G,\Sbf) \colon \alpha(a-a_{\circ})=-r\}\big\}$. 
For any $u_{r}\in \varphi_{\alpha}^{-1}(r)$, one has $H_{\alpha+r}=\{a\in \A_{\red}(\G,\Sbf) \colon m_{\alpha}(u_{r})(a)=a\}$. 
If $\alpha,2\alpha \in \Phi$, then for any $r \in \Gamma_{2 \alpha}$, we have $H_{2\alpha+r}=H_{\alpha+r/2} $.
These hyperplanes define a poly-simplicial structure on the standard apartment in the following way: 
\begin{definition}\label{definitionfacets}Define the equivalence relation on $\A_{\red}(\G,\Sbf)$ by $a\sim b$ if for any affine root $\beta=\alpha + r$ the sign of $\alpha(a-a_{\circ})+r$ and $ \alpha(b-a_{\circ})+r $ is the same or are both equal to zero; \footnote{Equivalently, $a \sim b$ if for any affine hyperplane $H$, either $a,b \in H$ or they are in the same connected component of $\A_{\red}(\G,\Sbf) \setminus H$.}
the equivalence classes are called the facets. 
A vertex is a point which is a facet, e.g. the point $a_\circ$. 
A facet with maximal dimension is called an alcove, it is also a connected component $\A_{\red}(\G,\Sbf)\setminus \bigcup_{\beta\in \Phi_{\aff} }H_\beta$.
\end{definition}
From now on, we fix an alcove $\a\subset \A_{\red}(\G,\Sbf)$\nomenclature[B]{$\a$}{Fixed alcove containing $a_\circ$ in its closure} containing in its closure the special vertex $a_\circ$.
\begin{definition}\label{signroot}
Let $\alpha \in \Phi$ and $r\in \Gamma_\alpha$. We say that $\alpha + r$ is $\a$-positive\nomenclature[B]{$\a$-\text{positive}}{} (resp. $\a$-negative) if $\alpha(a-a_\circ)+r >0$ (resp. $<0$) for some $a\in \a$ (then for all, since {the sign does not depend on the choice of $a\in \a$}). Let $\Phi_{\aff}^+$ (resp. $\Phi_{\aff}^-$) denote the set of affine $\a$-positive (resp. negative) affine roots.
\end{definition}
 For any non-empty subset $\Upomega \subset \A_{\red}(\G,\Sbf)$ and $\alpha \in \Phi$, we define $f_\Upomega\colon \Phi\to \R\cup \{\infty\}\nomenclature[B]{$f_\Upomega\colon \Phi\to \R\cup \{\infty\}$}{Function attached to $\Upomega$}$,
$$f_{\Upomega}(\alpha):=\inf\{r\in  \Gamma_\alpha \colon (\alpha+r)(\Upomega) \subset \R^+\}.$$
Define also the group $U_{\Upomega}$ \label{Uomega} to be the subgroup of $\G(F)$ generated by $\cup_{\alpha \in \Phi_{\red}}U_{\alpha+f_{\Upomega}(\alpha)}$. 
Note that for $a \in \A_{\red}(\G,\Sbf)$ and $\alpha\in \Phi$, the real $f_{\{a\}}(\alpha)$ depends only on the facet containing $a$. 
\begin{proposition}\label{propertiesII3.3}
For any non-empty subset $\Upomega \subset \A_{\red}(\G,\Sbf)$, the groups defined above have the following important properties:
\begin{enumerate}[(a)]
\item \label{1}For any $n\in \Nbf(F)$ we have $nU_\Upomega n^{-1}=U_{\nu_N(n)\cdot  \Upomega}$, so in particular $N_\Upomega:=\{n\in \Nbf(F)\colon \nu_N(n)\cdot x=x \text{ for all } x\in \Upomega\}$ normalizes $U_\Upomega$. 
\item\label{intersectionUpomega} For any $\alpha \in \Phi_{\red}$ we have $U_\Upomega \cap U_\alpha=U_{\alpha+ f_\Upomega(\alpha)}$.
\item\label{P_omega} The set $P_\Upomega:= N_\Upomega U_\Upomega= U_\Upomega  N_\Upomega$ is a group. 
We have $P_\Upomega\cap \Nbf(F)=N_\Upomega$.
\item[]For any decomposition into positive and negative roots $\Phi=\Phi^+ \sqcup \Phi^-$, we have
\item\label{homeoproduct} The following product map is an 
$$\begin{tikzcd}\prod_{\alpha \in \Phi_{\red}^\pm} U_{\alpha+f_{\Upomega}(\alpha)}\arrow{r}{\simeq}&U_\Upomega\cap \U(F)^\pm=:U_\Upomega^\pm.\end{tikzcd}$$
homeomorphism whatever ordering of the factors we take.
\item \label{5} $U_\Upomega=U_\Upomega^+ U_\Upomega^- (U_\Upomega \cap \Nbf(F))\).
\end{enumerate}
\end{proposition}
\begin{proof}
See \cite[6.2.10(iii),6.4.9, 7.1.3\& 7.1.8]{BT72}. 
We warn the reader that the groups denoted here $N_{\Upomega}$ and $N_{\Upomega}$ are denoted by $\widehat{N}_{\Upomega}$ and $\widehat{P}_{\Upomega}$.
\end{proof}
From now, we will adopt the following notation: when $\Upomega=\{x\}$ we will write $\square_x$ instead of $\square_{\{x\}}$ for  $\square\in\{f,U,N,P\}$.
\begin{example}\label{exfomega}
1 - For $\Upomega=\{a_\circ\}$, we have $f_{\{a_\circ\}}(\alpha)=0$ for all $\alpha\in \Phi$. Now, let us see the case $\Upomega=\a$. 
We first have
$$\a=\{a\in \A_{\red}(\G,\Sbf)\colon 0< \alpha(a-a_\circ) < n_{\alpha}^{-1}  \text{ for all } \alpha \in \Phi \text{ that are $\a$-positive}\}.$$
Therefore, if $\alpha\in \Phi $ is $\a$-positive, then $f_{\a}(\alpha)=0$. 
If now $\alpha \in \Phi_{\red}$ is $\a$-negative, then $-\alpha$ is $\a$-positive, hence for all $a\in \a$ we have $0< - \alpha(a-a_\circ) < n_{-\alpha}^{-1}$. The real $f_{\a}(\alpha)\in \Gamma_\alpha$ being the smallest element such that $\alpha(a-a_\circ)\ge -f_{\a}(\alpha)$, we see that $f_{\a}(\alpha)=n_{-\alpha}^{-1}$. 
We know that $\Gamma_{\alpha}=\Gamma_{-\alpha}$, this implies that $n_{\alpha}^{-1}=n_{-\alpha}^{-1}$. 
In conclusion: 
$$f_{\a}(\alpha)=\begin{cases} 0 &\text{ if $\alpha \in \Phi$ is $\a$-positive},\\
n_{\alpha}^{-1}&\text{ if $\alpha \in \Phi$ is $\a$-negative}
\end{cases}$$
2 - If $\Upomega_1 \subset \Upomega_2$, then $f_{\Upomega_1}(\alpha)\ge f_{\Upomega_2}(\alpha)$, for all $\alpha \in \Phi$. For example, if $\cF$ is facet of $ \A_{\red}(\G,\Sbf)$ such that $a_\circ \in \overline{\cF}\subset \overline{\a}$, then 
$$f_{\cF}(\alpha)=0 \text{ if $\alpha \in \Phi$ is $\a$-positive}, \text{ and }
 f_{\cF}(\alpha) \in \{0, n_{\alpha}^{-1}\} \text{ if $\alpha \in \Phi$ is $\a$-negative}.$$ 
3 - For any non-empty subset $\Upomega \subset \A_{\red}(\G,\Sbf)$ and any $n\in \widetilde{W}$ whose image in $W$ is $w$, we have by Lemma \ref{NactionUalphar}
$$f_{\nu_N(n) \cdot \Upomega}(\alpha)=w(\alpha)(\nu_N(n)(a_\circ)-a_\circ)  + f_{\Upomega}(w^{-1}(\alpha)).$$
\end{example}
\subsection{Affine Weyl groups}\label{affineweyl}
For any affine root $\beta=\alpha+r \in\Phi_{\aff}$, let $s_{\beta}\in \nu_N(\Nbf(F))$ denote the orthogonal reflection with respect to the hyperplane $H_{\beta}$:
$$s_\beta \colon a\mapsto-\left(\alpha(a-a_{\circ})+r\right)\alpha^\vee + a, \, \forall a\in \A_{\red}(\G,\Sbf).$$
\begin{definition}
We define the affine Weyl group $W_{\aff}\subset \nu_N(\Nbf(F))$\nomenclature[B]{$W_{\aff}$}{The affine Weyl group} to be the group generated by the reflections $s_\beta$ for all $\beta \in  \Phi_{\aff}$. It is a normal subgroup of $ \nu_N(\Nbf(F))$ \cite[6.2.11]{BT72}.
\end{definition}
By \cite[\S 4.2.21 \& \S 5.1.19]{BT84}, we have $\Phi':=\{\alpha\in \Phi\colon \Gamma_{\alpha}'\neq 0\}=\Phi$. 
The set $\Sigma := \{n_\alpha \alpha \colon \alpha \in \Phi\}$ is a root system and the map
$$\begin{tikzcd}[column sep= small]\Phi_{\red,\aff}:=\cup_{\alpha\in \Phi_{\red}}(\alpha +\Gamma_\alpha') \arrow{r}{\simeq}& \Sigma_{\aff}:=\cup_{\alpha\in \Sigma}(\alpha +\Z) ;\quad (\alpha+r)  \arrow[r,mapsto]& n_\alpha(\alpha+r),\end{tikzcd}$$
is bijective and respects positivity. 
The group $W_{\aff}=W_{\aff}(\Sigma)$ is the affine Weyl group associated with $\Sigma$ and $\Phi_{\red,\aff}$ is the corresponding affine roots system
\cite[6.2.22]{BT72}. 
This is also the normal subgroup of $\nu_N(\Nbf(F))$ generated by the set of orthogonal reflections $s_{\alpha+r}$ for all $\alpha \in \Phi$ and all $r\in \Gamma_\alpha$.

Define $\Lambda_{\aff}$ to be the subgroup of translations in $W_{\aff}$, we can then identify $\Lambda_{\aff}$ with the coroot lattice $\Z[\Sigma^\vee]$ \cite[Proposition 6.2.20]{BT72}.
The affine Weyl group $W_{\aff}$ acts simply transitively on the set of alcoves in $\A_{\red}(\G,\Sbf)$ (VI \S 2.1 {\em loc. cit.}).

Let $a$ be vertex in $\A_{\red}(\G,\Sbf)$, we denote by $\Phi_{\aff}^a$ the set of affine roots that vanish at $a$, set $W_{\aff}^a=\langle s_\beta \colon \beta \in \Phi_{\aff}^a \rangle \subset W_{\aff}$. The vertex $a$ is special\label{specialvertex} in the sense of Remark \ref{specialpointvaluation} if and only if the composition of the following maps
$$\begin{tikzcd}W_{\aff}^a\arrow[r,hook] & W_{\aff}\arrow[r,hook]& \nu_N(\Nbf(F)) \arrow[r, twoheadrightarrow] & W,\end{tikzcd}$$ 
is an isomorphism\footnote{This is equivalent to $W_{\aff}^a\simeq W$ or just to $\#W_{\aff}^a=\#W$ since the induced map $W_{\aff}^a\hookrightarrow W_{\aff}\to W$ is always injective.}. Special vertices exists by \cite[Corollaire 6.2.16]{BT72}.  
We have then by \cite[1.3 \& 6.2.19]{BT72} a decomposition for the affine Weyl group
$$W_{\aff}\simeq \Lambda_{\aff} \rtimes W_{\aff}^a\simeq \Z[\Sigma^\vee]\rtimes W_{\aff}^a,$$    
for any fixed special vertex $a\in \A_{\red}(\G,\Sbf)$.
\begin{definition}\label{wallsgeneratingcox}
A face (resp. facet) of the alcove $\a$ is a facet contained in a single affine hyperplane (resp. in $\overline{\a}$).
A wall of the alcove $\a$ is a hyperplane containing a face of $\a$. 
Set
$$S(\a):=\{s_H \colon \text{$H$ a wall of $\a$}\}.$$
Define the type of a facet $\cF$ of $\a$ to be the set
$$\mathcal{T}_\cF=\{s_H \colon\text{$H$ a wall of $\a$, }\cF\subset H \}\subset S(\a).$$
Let $W_{\aff}^{\cF}$ be the group generated by $\mathcal{T}_\cF$. Then $(W_{\aff}^{\cF}, \mathcal{T}_\cF)$ is a finite Coxeter system.
\end{definition}
\begin{remark}
A facet is determined by its type
$$\cF=\{a \in \overline{\a} \colon (a\in H \Longleftrightarrow \cF \subset H) \text{ for any wall $H$ of $\a$}.\}$$
\end{remark}

Since a facet of $\cF$ is the image by an element of $W_{\aff}$ of a unique facet of $\a$ and each facet of $\a$ is determined by its type, one can then define the type for all facets \cite[1.3.5]{BT72}. So alcoves have empty types ($\mathcal{T}_\a=\emptyset$) and special points have full type and $W_{\aff}^{a_\circ}$ is generated by $ \mathcal{T}_{a_\circ}$. 
\subsection{The reduced Bruhat--Tits building}
We arrive now at the most important object of this section, that is the reduced Bruhat--Tits building:
\begin{definition}Define the set 
$$\cB(\G,F)_{\red}:=\G(F) \times \A_{\red}(\G,\Sbf)/\sim,$$\nomenclature[B]{$\cB(\G,F)_{\red}$}{The reduced Bruhat--Tits building}
where $\sim$ is the equivalence relation on $\G(F)\times \A_{\red}(\G,\Sbf)$ defined by $$(g,x)\sim(h,y)\,\text{ if }\exists n\in \Nbf(F)\text{ such that } nx=y \text{ and } g^{-1}hn\in U_{x}.$$
\end{definition}
The group $\G(F)$ acts on $\cB(\G,F)_{\red}$ on the left:
$$g \cdot [(h,y)]:=[(gh,y)] \quad\text{ for }g\in \G(F) \text{ and } (h,y)\in\G(F) \times \A_{\red}(\G,\Sbf).$$
Moreover, the map $ \A_{\red}(\G,\Sbf)\to\cB(\G,F)_{\red}$ given by $x\mapsto [(1,x)]$ is an $\Nbf(F)$-equivariant embedding. We will denote its image by $\cA$.  We have $g\cdot \cA=\cA$ (resp. $g\cdot x=x,\, \forall x\in \cA$) if and only if $g\in \Nbf(F)$ (resp. $g\in \ker \nu_N$) \cite[7.4.10]{BT72}.
\begin{definition}
An apartment of the building $\cB(\G,F)_{\red}$ is a subset of the form $g\cA$ for some $g\in \G(F)$. A facet (resp. alcove) of $ \cB(\G,F)_{\red}$ is a subset of the form $g\cF$ for some $g\in \G(F)$ and a facet (resp. alcove) $\cF\subset \cA$.
\end{definition}
Since all maximal $F$-split $F$-tori of $\G$ are $\G(F)$-conjugate \cite[Theorem 20.9]{Bor91}, 
all apartments of $\mathcal{B}(\G,F)_{\red}$ are in bijection with maximal split tori of $\G$.
\begin{proposition}\label{buildproperties}
In this proposition, we will collect some important properties regarding the building $\cB(\G,F)_{\red}$:
\begin{enumerate}[(a)]
\item (Fixators)\label{abuildproperties} Let $\Upomega \subset \cA$. We have an alternative characterization for the subgroups $P_\Upomega=N_\Upomega U_\Upomega$ (defined in Proposition \ref{propertiesII3.3} - \ref{P_omega}):
$$P_\Upomega= \text{Fix}(\G(F),\Upomega) :=\{g\in \G(F)\colon gx=x \quad\forall x\in \Upomega\},$$
in other words $P_\Upomega$ is the subgroup that fixes all points of $\Upomega$ \cite[7.4.4]{BT72}. 
\item\label{buildproperties2} For any $g\in \G(F)$ we can find a $n\in \Nbf(F)$ verifying $gx=nx$ for all $x$ in the closed subset  $\cA\cap g^{-1} \cA$ \cite[7.4.8]{BT72}.
\item\label{transitivity} (Transitivity) Let $\Upomega \subset \cA$. The group $U_\Upomega$ acts transitively on the set of all apartments containing $\Upomega$ \cite[7.4.9]{BT72}.
\item\label{containment} For any two facets in the building, there exists an apartment that contains both of them \cite[7.4.18]{BT72}.
\item Let us fix an $W$-invariant euclidean metric $d$ on $\cA$. The previous properties ensure that the distance $d$ extends uniquely to a $\G(F)$-invariant metric on the set $(\cB(\G,F)_{\red},d)$.
\end{enumerate}
\end{proposition}
\begin{remark}The term "fixators" refers to pointwise stabilizers, in contrast the term "stabilizers" will be used for setwise stabilizers.\qedhere
\end{remark}
\begin{definition}The reduced Bruhat--Tits building of $\G(F)$ is the pair $(\cB(\G,F)_{\red},d)$ together with its isometric $\G(F)$-action and the poly-simplicial structure defined by its facets.
\end{definition}

\subsection{The extended Bruhat--Tits building}\label{extendedbuilding}
\subsubsection{The extended apartment}
When the center $Z(\G)$ has split rank $>0$, the stabilizer $P_x$ of a point $x\in\cA \subset \cB(\G,F)_{\red}$ is no longer a compact subgroup of $\G(F)$. To remedy this "issue" we define in this section a larger building following \cite[4.2.16]{BT84}.
 
We have a decomposition $X_*(\Sbf)\otimes_\Z \Q=(X_*(\Sbf^{der})\otimes_\Z\Q)\oplus (X_*(\Zbf_{c,sp})\otimes_\Z\Q)$, 
where $\Sbf^{der}=(\G^{\der}\cap \Sbf)^\circ$. 
This is the "dual" statement of the decomposition (2) in Lemma \cite[Lemma 11.3.3]{Conrad19}. 
This decomposition allows us to inject $\Phi^\vee$ in $V^\vee$. So, the "dual" statement of (3) in {\em loc. cit.} together with \cite[Theorem 201.6 \& Remark 21.1]{Bor91} imply that $V=\text{span}_\R(\Phi^\vee) = X_*(\Sbf^{der})\otimes_\Z \R$. 
Accordingly, the apartment $\cA\subset \cB(\G,F)_{\red}$ corresponding to $\Sbf$ can be viewd as a torsor for the real vector space 
$$X_*(\Sbf)\otimes_\Z \R\big{/}X_*(\Zbf_{c,sp})\otimes_\Z \R.$$
Let us construct a homomorphism $\nu_{G} \colon \G(F) \to X_*(\Zbf_{c,sp})\otimes_\Z \Q$: 
Using the isogeny $\Zbf_{c} \times \G^{\der}\to \G$, we see that 
 $X^*(\G)_F=X^*(\G/\G^{\der})_F$. Since the quotient map $\Zbf_{c} \to \G/\G^{\der}$ is an $F$-isogeny of $F$-tori, we see that $X^*(\G/\G^{\der})_F$ identifies with a subgroup of $X^*(\Zbf_{c})_F$ of finite index. Therefore, we have a natural isomorphism of $\Q$-vector spaces $$X^*(\G/\G^{\der})_F\otimes_\Z \Q \to  X^*(\Zbf_{c})_F\otimes_\Z\Q.$$
Identify the $\Q$-linear dual $\Hom_\Z(X^*(\Zbf_{c,sp}),\Q )$ of $X^*(\Zbf_{c,sp})\otimes_\Z \Q$ with $X_*(\Zbf_{c,sp})\otimes_\Z \Q$. In conclusion, we get a natural isomorphism 
 $$X^*(\G)_F\otimes_\Z \Q\simeq X^*(\Zbf_{c,sp})\otimes_\Z \Q.$$
 Consequently, we obtain a canonical isomorphism 
 $$X_*(\Zbf_{c,sp})\otimes_\Z \Q \simeq \Hom_\Z(X^*(\G)_F\otimes_\Z \Q,\Q ).$$
 This shows that there exists a unique homomorphism 
 $$ \G(F) \to X_*(\Zbf_{c,sp})\otimes_\Z \Q,\nomenclature[B]{$\nu_G$}{Translation action of $\G(F)$ on $V_G$}$$
extending $\nu_G\colon \G(F)\to  \Hom_\Z(X^*(\G)_F,\Z)$ (\S \ref{notationsCh3}); this is the unique homomorphism denoted by abuse of notation $\nu_G$ and satisfying\footnote{We abuse notation and denote also by $\langle\,, \rangle$ the pairing $X_*(\Zbf_{c,sp}) \times X^*(\Zbf_{c,sp}) \to \Z$.} (see proof of Lemma \ref{nu})
$$\langle \nu_{G}(g),\chi|_{\Zbf_{c,sp}}\rangle=-\omega(\chi(g)),\text{ for all } \chi\in X^*(\G)_F \text{ and } g\in \G(F).$$  
Put $\G(F)^1:=\ker \nu_{G}$\nomenclature[B]{$\G(F)^1$}{The kernel of $\nu_{G}$}\footnote{We refer the reader to \S \ref{Kottwitzhom} for more on the homomorphism $\nu_G$}. Then $ \G(F)/\G(F)^1$ is a finitely generated abelian group and there is an isomorphism
$$\G(F)/\G(F)^1 \otimes_\Z \Q \simeq  X_*(\Zbf_{c,sp})\otimes_\Z \Q.$$
Set $V_G:=X_*(\Zbf_{c,sp})\otimes_\Z \R$\nomenclature[B]{$V_G$}{$X_*(\Zbf_{c,sp})\otimes_\Z \R$} and let $\A_G$ be a fixed affine space under $V_G$. 
We have a morphism $\nu_G\colon \G(F) \to \Aff(\A_G)$ sending any $g$ to the translation $ (a\mapsto  \nu_{G}(g)+a)$.

We define the extended standard apartment $\A_{\ext}(\G,\Sbf)$\nomenclature[B]{$\A_{\ext}(\G,\Sbf)$}{Extended standard apartement} to be the product of $\A_{\red}(\G,\Sbf)\times \A_G$ together with the group homomorphism
$$\begin{tikzcd}[row sep=tiny]\nu_{N,\ext}\nomenclature[B]{$\nu_{N,\ext}$}{Action of $\Nbf(F)$ on $\Aff(\A_{\ext}(\G,\Sbf))$} \colon \Nbf(F)\arrow{r}& \Aff(\A_{\ext}(\G,\Sbf)),&  n \arrow[r, mapsto]& \nu_N(n) \oplus \nu_G(n).\end{tikzcd}$$
The decomposition $X_*(\Sbf)\otimes_\Z \R =V \oplus V_G$ shows that the extended standard apartment $\A_{\ext}(\G,\Sbf)$ (as defined above) is actually an affine space under $X_*(\Sbf)\otimes_\Z \R$ and the restriction of $\nu_{N,\ext}$ to $\Mbf(F)$ corresponds precisely to the translation action given by the homomorphism $\nu_M$ of Lemma \ref{nu}. 
In particular, 
$$
\Mbf(F)^1=\ker \nu_M =\ker \nu \cap \G(F)^1   =\ker \nu_N \cap \G(F)^1= \ker \nu_{N,\ext}. 
$$
\begin{remark}\label{remark13}
While the reduced system $(\A_{\red}(\G,\Sbf),\nu_N)$ is canonical, that is, unique up to unique isomorphism (Proposition \ref{A(G,S)}), the extended system $(\A_{\ext}(\G,\Sbf),\nu_{N,\ext})$ is only unique up to isomorphism. We "canonify" it (following G. Rousseau) by viewing (which we will adopt from now on) $\A_G$ as $V_G$ with a marked origin $\{0_{V_G}\}$.  
\end{remark}
Since there is a one-to-one correspondence between systems of positive roots for $\Phi(\G,\Sbf)$ and vectorial chambers in $\A_{\ext}(\G,\Sbf)$ with apex the fixed special point $a_\circ$, we may and will refine our choice of the alcove $\a$ to be the unique alcove with apex $a_\circ$ such that
$$\{\alpha \in \Phi\colon \alpha \text{ is $\a$-positive}\}=\Phi^+.$$ 
\subsubsection{The extended building}
\begin{definition}
The extended building is the following product of a poly-simplicial complex and a real vector space
$$\cB(\G,F)_{\ext}=\cB(\G,F)_{\red}\times V_G.\nomenclature[B]{$\cB(\G,F)_{\ext}$}{The extended Bruhat--Tits building}$$
The group $\G(F)$ acts isometrically on it as follows:
$$g\cdot (x,v)=(g\cdot x,v+\nu_{G}(g))\quad \forall g\in \G(F), x\in\cB(\G,F)_{\red}, v\in V_G.$$
\end{definition}
A facet (resp. alcove, Weyl chamber, apartment) of $\cB(\G,F)_{\ext}$ is defined to be a product of a same object in $\cB(\G,F)_{\red}$ and $V_G$. 
For instance let us fix the extended apartment $\cA_{\ext}=\cA \times V_G$.
We identify $\cB(\G,F)_{\red}$ with the subset $\cB(\G,F)_{\red}\times V_G$ in $\cB(\G,F)_{\ext}$. Then, the stabilizer of $\cB(\G,F)_{\red}$ for the action of $\G(F)$ in the extended building is exactly $\ker \nu_{G}$. 
In this building the minimal dimensional facets are of the form $\cF_x:=\{x\}\times V_G$ for some vertex $x\in \cB(\G,F)_{\red}$. 
The stabilizer of $(x, 0)$ is equal to $P_x\cap \ker \nu_{G}$, which is also the fixator of $\cF_x$, where $P_x$ is the subgroup defined in \S \ref{discretevaluation}.

\subsection{Parahoric subgroups}
\subsubsection{Further notations}

Let $F^{un}$ denote the maximal unramified extension of $F$ contained in $F^{sep}$ and $L$ the completion of $F^{un}$ with respect to the valuation on $F^{un}$ which extends the normalized valuation on $F$. The residue field of $F$ is perfect, thus $F^{un}=F^{\text{sh}}$ is the strict henselization of $F$ in the fixed separable closure $F^{sep}$. 
Let $L^{sep}$ the separable closure of $L$ containing $F^{sep}$. 
The arithmetic Frobenius automorphism $\sigma\in \Gal(F^{un}/F)$ (which induces $x \mapsto x^q$ on the residue field of $F^{un}$) extends continuously to an automorphism of $L$ over $F$, also denoted $\sigma$. 
Write $\text{In}=\Gal(F^{sep}/F^{un})$ for the inertia subgroup of $\Gal(F^{sep}/F)$. Since ${L}^{sep}=L\otimes_{F^{un}} {F}^{sep}$, one can identify the inertia subgroup $\text{In}$ with $\Gal(L^{sep}/L)$. 

\subsubsection{The Kottwitz homomorphism}\label{Kottwitzhom}
Let $\H$ be any connected reductive $F$-group, $\widehat{\H}$ be its connected Langlands dual and $Z(\widehat{\H})$ be the center of $\widehat{\H}$. 
Kottwitz defines in \cite[\S 7.7]{kottwitz_1997} a functorial surjective homomorphism 
$$\kappa_{\H}: \H({L}) \rightarrow X^*(Z(\widehat{\H})^\text{In}) =X^*(Z(\widehat{\H}))_\text{In},$$\nomenclature[B]{$\kappa_\H$}{The Kottwitz map of $\H$ and $\kappa_H$ its restriction to $\H(F)$}where, the subscript $\text{In}$ indicates coinvariants of the $\Gal(F^{sep}/F)$-module $X^*(Z(\widehat{\H}))$, {this latter is the Borovoi fundamental algebraic group (usually denoted $\pi_1(\H)$) and it is isomorphic to the quotient $X_*(\T)$ for a maximal $F$-torus $\T$ by the coroots lattice. The group $\pi_1(\H)$ acquires an action of $\Gal(F^{sep}/F)$ via its representation as $X_*(\T)/\sum_{\alpha \in \Phi(\H_{F^{sep}},\T)}\Z \alpha^\vee$}. There is a canonical surjective homomorphism 
$$q_{\H} \colon X^*(Z(\widehat{\H}))_{\text{In}} \to \Hom_\Z (X^*(\H)^{\text{In}},\Z)$$
whose kernel is the torsion subgroup of $X^*(Z(\widehat{\H}))_{\text{In}}$ \cite[7.4.4]{kottwitz_1997}\footnote{The codomain of (7.4.4) in {\em loc. cit.} is actually $\Hom_\Z (X_*(Z(\widehat{\H}))^{\text{In}},\Z))$, so one needs to use the isomorphism $X_*(Z(\widehat{\H})) \simeq X^*(\H)$.}. 
Kottwitz shows in \cite[\S 7.4]{kottwitz_1997} that the above two homomorphisms sit in the following commutative diagram:
$$\begin{tikzcd}[column sep=normal,row sep=large]
\H(L)  \arrow[twoheadrightarrow]{rr}[description]{\kappa_{\H}}\arrow[]{rrdd}[description]{-\nu_{\H}}&& X^*(Z(\widehat{\H}))_{\text{In}}\arrow[twoheadrightarrow]{dd}[description]{q_{\H}}
 \\ \\
 && \Hom_\Z(X^*(\H)^{\text{In}},\Z)
\end{tikzcd}$$
where, $\nu_{\H}$ is the natural homomorphism characterized by $$\langle \nu_{\H}(h) , \chi \rangle= \nu_{\H}(h) (\chi)=-\omega(\chi(h)),\text{ for all }h\in \H(L)   \text{ and } \chi\in X^*(\H)^{\text{In}}.$$
Therefore, $\H(L)_1:= \ker \kappa_\H \subset \H(L)^1:= \ker \nu_\H$ and 
$$\left(X^*(Z(\widehat{\H}))_{\text{In}}\right)_{\text{tors}}=\H(L)^1/ \H(L)_1.$$
\begin{remark}
Note here that $\nu_{\H}$ differs in a sign from the map $v_H$ in \cite[7.4.3]{kottwitz_1997}.
\end{remark}
Since our connected reductive group $\H$ is defined over $F$, the restriction of $\kappa_\H$ to $\H(F)$ provides a surjective homomorphism $\kappa_H$ that sits in the following commutative diagram (see \cite[\S 7.7]{kottwitz_1997} for the surjectivity of $\kappa_H$, see also \cite[\S 2.7]{R15})
$$\begin{tikzcd}[column sep=large,row sep=large]
\H(F)  \arrow[twoheadrightarrow]{rr}[description]{\kappa_{H}}\arrow[]{rrdd}[description]{-\nu_{H}}&& (X^*(Z(\widehat{\H}))_{\text{In}})^{\langle \sigma\rangle}\arrow[]{dd}[description]{q_{H}}
 \\ \\
 && \Hom_\Z(X^*(\H)_F,\Z)
\end{tikzcd}$$
We denote by $\H(F)_1= \H(F)\cap \H(L)_1 $\nomenclature[B]{$\H(F)_1$}{$\H(F)\cap \ker \kappa_{\H}$} and $\H(F)^1$ the kernel of $\kappa_H$ and $\nu_H$, respectively. 
Therefore, by \cite[\S 3.2, Lemma]{HV15}
$$\H(F)^1 = \H(L)^1\cap \H(F)=\left\{h\in \H(F)\colon \kappa_H(h)\in \left(X^*(Z(\widehat{\H}))_{\text{In}}^\sigma\right)_{\text{tors}} \right\}.$$
Set $\Lambda_H:=\H(F)/\H(F)_1$\nomenclature[B]{$\Lambda_H$}{The abelian group $\H(F)/\H(F)_1$}, thus
$$\H(F)^1/\H(F)_1 = (\Lambda_H)_{\text{tors}} \simeq \left(X^*(Z(\widehat{\H}))_{\text{In}}^\sigma\right)_{\text{tors}}.$$
\subsubsection{Passage to completion and descent}
Consider the extended building $\cB(\G,F^{un})_{\ext}$ of the group $\G_{F^{un}}=\G\times_FF^{un}$. This building is equipped with an action of $\G(F^{un})\rtimes \langle \sigma\rangle$. Moreover, there is a natural $\G(F)$-equivariant embedding $\iota\colon \cB(\G,F)_{\ext}\hookrightarrow \cB(\G,F^{un})_{\ext}$ such that \cite[5.1.25]{BT84}
$$\iota(\cB(\G,F)_{\ext})=\cB(\G,F^{un})_{\ext}^{\sigma}.$$
In other words, the extended building of $\G$ over $F$ is identified with the fixed points of $\sigma \in \text{Gal}(F^{un}/F)$ (in particular of $\sigma$) in the building of $\G$ over $F^{un}$. 

Since $F^{un}$ is henselian, Rousseau has shown in \cite[proposition 2.3.9]{rousseau77} that $\G$ has the same relative rank over $F^{un}$ and over $L$ and hence by \cite[proposition 2.3.9]{rousseau77} the building $\cB(\G,F^{un})_{\ext}$ identifies canonically with $\cB(\G,L)_{\ext}$ as $\G(F^{un})$-space. Using this identification, we see that the extended building of $\G$ over $F$, can also be identified with the fixed points of $\sigma \in \text{Aut}(L/F)$ in the building of $\G$ over $L$, i.e.
$$\iota(\cB(\G,F)_{\ext})=\cB(\G,L)_{\ext}^{\sigma}.$$
This gives a bijection between the set of $\sigma$-stable facets in $\cB(\G,L)_{\ext}$ and the set of facets in $\cB(\G,F)_{\ext}$. 
For more details one can consult \cite{Prasad17}.
\subsubsection{$F$-Parahoric subgroups}\label{Parahoric}
In \cite[\S 4.6.28, \S 5.2.6, \S 5.1.29 and \S 5.1.30]{BT84} (see also \cite[\S 9.3]{Yu16}), Bruhat and Tits associated two subgroups to any non-empty $\sigma$-stable set $\Upomega$ that is contained in an apartment of $\cB(\G,F^{un})_{\ext}$ and with bounded image in the reduced building: 
\begin{enumerate}\item The fixator subgroup $$\text{Fix}(\G(F^{un}),\Upomega) :=\{g \in \G(F^{un})\colon g\cdot a=a, \forall a\in \Upomega \}\subset \G(F^{un}).$$
They also showed in {\em loc. cit.}, the existence of a unique affine smooth group scheme $\mathbb{P}_{\Upomega}$ 
defined over $\Spec \,{\cO_F}$ with generic fiber $\G$ and which is uniquely characterized by the property $\mathbb{P}_{\Upomega}(\cO_{F^{un}})= \text{Fix}(\G(F^{un}),\Upomega)$.
\item The "connected fixator" subgroup 
$$\text{Fix}(\G(F^{un}),\Upomega)^\circ:=\mathbb{P}_{\Upomega}^\circ (\cO_{F^{un}}) \subset \text{Fix}(\G(F^{un}),\Upomega),\nomenclature[B]{$\text{Fix}(\G(F^{un}),\Upomega)^\circ$}{The connected fixator subgroup of $\Upomega$}$$
where $\mathbb{P}_{\Upomega}^\circ $ is the identity component of $\mathbb{P}_{\Upomega}$.
\end{enumerate}
A $F$-parahoric subgroup of $\G$ is by definition \cite[under 5.2.6]{BT84} the "connected fixator" of a facet $\cF \subset \iota(\cB(\G,F)_{\ext})$.
\begin{remark}
{
We warn the reader that, in \cite[\S 4 and \S 5]{BT84} notations, $\mathbb{P}_{\Upomega}$, $\mathbb{P}_{\Upomega}^\circ$, $\text{Fix}(\G(F^{un}),\Upomega)$, $\text{Fix}(\G(F^{un}),\Upomega)^\circ$ and $\mathbb{P}_{\Upomega}$ and $\mathbb{P}_{\Upomega}^\circ$ are denoted by $\widehat{ \mathfrak{G}}_{\Upomega}^\natural$, ${\mathfrak{G}}_{\Upomega}^{0\natural}$, $\widehat{P}_{\Upomega}^1$ and ${P}_{\Upomega}^0$, , respectively.
}
\end{remark}
\begin{lemma}\label{equalityofparah}
Two facets $\cF,\cF'\subset \cB(\G,F)_{\ext}$ are equal if and only if $$\mathbb{P}_{\iota(\cF)}^\circ (\cO_{F}) = \text{Fix}(\G(F^{un}),\iota(\cF))^\circ \cap \G(F) =  \text{Fix}(\G(F^{un}),\iota(\cF'))^\circ\cap \G(F) = \mathbb{P}_{\iota(\cF')}^\circ (\cO_{F}).$$
\end{lemma}
\begin{proof}
See \cite[5.2.8]{BT84}.
\end{proof}
This lemma justifies (by misuse of language) the following definition:
\begin{definition}\label{parahoricsubgroup}
For any bounded subset  $\Upomega$ contained in an apartment of $\cB(\G,F)_{\red}$, set
$$\widetilde{K}_{\Upomega}=\mathbb{P}_{\iota(\Upomega \times V_G)} (\cO_{F^{un}}) \cap \G(F) \text{ and }K_{\Upomega}:=\mathbb{P}_{\iota(\Upomega \times V_G)}^\circ (\cO_{F})$$ 
A parahoric subgroup of $\G(F)$ is $K_{\cF}$ for some facet $\cF\subset \cB(\G,F)_{\red}$. 
When the facet $\cF$ is an alcove $\mathfrak b$, the parahoric subgroup $K_{\cF}$ is called an Iwahori subgroup. The subgroups $\widetilde{K}_{\mathfrak b}$ and ${K}_{\mathfrak b}$ will be denoted by $\widetilde{I}_{\mathfrak b}$ and ${I}_{\mathfrak b}$, respectively. 
\end{definition}
\begin{remark}
The subgroup $\mathbb{P}_{\Upomega \times V_G} (\cO_{F})$ might be strictly smaller than $\widetilde{K}_{\Upomega}$, see \cite{Haines09cor} for examples. 
\end{remark} 
Haines and Rapoport showed in \cite{HR8} that parahoric subgroups over $\G(F^{un})$ are actually $\G(F^{un})_1$-fixators of facets in the Bruhat--Tits building over $F^{un}$.  
Accordingly, as we see in the following proposition, this result descend also to the base field $F$, i.e. parahoric subgroups of $\G(F)$ are fixators of facets in the kernel of $\kappa_G$, i.e. for any facet $\cF\subset \cB(\G,F)_{\red}$ we have $K_{\cF}= P_\cF \cap \G(F)_1$.
\begin{proposition}\label{HRpara}\label{parahr}
For any bounded subset  $\Upomega$ contained in an apartment of $\cB(\G,F)_{\red}$ we have 
$$\widetilde{K}_{\Upomega}=P_{\Upomega} \cap \G(F)^1 \text{ and } K_{\Upomega} =\widetilde{K}_{\Upomega} \cap \G(F)_1 =\mathbb{P}_{\iota(\Upomega \times V_G)} (\cO_{F}) \cap \G(F)_1$$
\end{proposition}
\begin{proof}
On the one hand,
$$\text{Fix}(\G(F^{un}),\iota(\Upomega \times V_G))\cap \G(F) = \text{Fix}(\G(F^{un}),\iota(\Upomega)) \cap \G(F^{un})^1 \cap \G(F).$$
On the second hand $\text{Fix}(\G(F^{un}),\iota(\Upomega))  \cap \G(F)= P_{\Upomega}$ (because the map $\iota$ is $\G(F)$-invariant and \ref{abuildproperties} Proposition \ref{buildproperties}), which shows the first equality of the Lemma, since $\G(F^{un})^1 \cap \G(F) = \G(F)^1 $. 
For the remaining equalities, thanks to \cite[Proposition 3, remarks 4 and 11]{HR8}, one has {$\mathbb{P}_{\iota(\Upomega \times V_G)}^\circ (\cO_{F^{un}})=\mathbb{P}_{\iota(\Upomega \times V_G)} (\cO_{F^{un}}) \cap \G(F^{un})_1$}, and accordingly:
\begin{align*}
\mathbb{P}_{\iota(\Upomega \times V_G)}^\circ (\cO_{F})& =\mathbb{P}_{\iota(\Upomega \times V_G)}^\circ (\cO_{F^{un}}) \cap \G(F)\\
&=\mathbb{P}_{\iota(\Upomega \times V_G)} (\cO_{F^{un}}) \cap \G(F^{un})_1 \cap \G(F)=\mathbb{P}_{\iota(\Upomega \times V_G)} (\cO_{F})\cap \G(F)_1.\qedhere
\end{align*}
\end{proof}
The map that associates to a facet $\cF$ its parahoric subgroup $ K_\cF$ is decreasing, in particular, if $\cF$ is any facet of the alcove $\a \subset \cB(\G,F)_{\red}$ then we have $I_{\a}\subset K_\cF$. In addition, the action of an element $g\in \G(F)$ on a facet $\cF$ translates for parahorics to $K_{g\cdot\cF}=g K_\cF g^{-1}$.
\subsubsection{The scheme $\mathcal{Z}$}\label{Z}
By replacing $\G$ by $\Mbf$ in all of the above, 
we see that $\cB(\Mbf,F)_{\red}$ is reduced to a single point $\{a_{\circ,M}\}$ \cite[\S 5.1.26]{BT84}. 
The fixator subgroup $\text{Fix}(\Mbf(F),a_{\circ,M})$ is equal to the $\cO_F$-points of a unique smooth $\cO_F$-group scheme $\mathcal{Z}$ which is the schematic closure of $\Mbf$ in $\mathbb{P}_{\iota(\Upomega\times V_G)}$ for any bounded subset $\Upomega \subset \mathbb{A}_{\red}(\G,\Sbf)$ (see \cite[5.2.1 \& 5.2.4]{BT84}) and the corresponding parahoric subgroup is
$$K_{a_{\circ,M}}=\text{Fix}(\Mbf(F),a_{\circ,M})^\circ=\mathcal{Z}^\circ(\cO_F).$$ 
Applying Proposition \ref{HRpara} to $\Mbf$ and $\Upomega=\{a_{\circ,M}\}$ one shows that $K_{a_{\circ,M}} = \Mbf(F)_1$. 
By \cite[\S 5.2.1]{BT84}, one has $\Mbf(F)^1=\ker \nu \cap \G(F)^1=\mathcal{Z}(\cO_F)$ and in particular $\Mbf(F)_1=\mathcal{Z}^\circ(\cO_F)$ is of finite index in $\Mbf(F)^1$. 
In summary:
\begin{corollary}\label{uniquepara}
The subgroup $\Mbf(F)_1$ is the unique parahoric subgroup of $\Mbf(F)$ and it is a finite index subgroup of $\Mbf(F)^1=\ker \nu_{M}$ (The maximal open compact subgroup of $\Mbf(F)$). 
\end{corollary}
\begin{lemma}\label{NcapIM1}
Let ${\mathfrak{b}} \subset \cA$ be an alcove. The $\Nbf(F)$-fixator of ${\mathfrak{b}}$ (resp. ${\mathfrak{b}} \times V_G$) is $N_{\mathfrak{b}}=\ker \nu_N$ (resp. $\Nbf(F)\cap \widetilde{I}_{\mathfrak{b}} = \Mbf(F)^1$). Therefore,
$$P_{\mathfrak{b}}= U_{\mathfrak{b}}^+  U_{\mathfrak{b}}^-  \ker \nu \, \text{ and }  \widetilde{I}_{\mathfrak{b}}= U_{\mathfrak{b}}^+  U_{\mathfrak{b}}^-  \Mbf(F)^1 .$$
\end{lemma}
\begin{proof}
{The first equality is stated at the end of \cite[\S 4.1.2]{BT72}. Here is a geometric argument: 
Because $\Nbf(F)$ acts transitively on the set of alcoves of $\cA$, it suffices to prove the Lemma for $\mathfrak{b}=\a$. 
Let $n\in \Nbf(F)\cap P_{\a}$. 
By assumption $\nu_N(n)(a)=a$ for any $a\in \overline{\a}$, in particular for all $a\in \text{vert}(\a)=\{\text{vertices of $\a$}\}$. 
The set $\left\{a-a_\circ \colon a\in \text{vert}(\a)\setminus\{a_\circ\}\right\}$ forms a basis for $V$ and hence $\nu_N(n)$ must equal the identity on $V$. Therefore, $ \Nbf(F)\cap P_{\a}=\Mbf(F)\cap P_{\a}= \ker \nu$. 
By \ref{P_omega} \ref{propertiesII3.3}, this is precisely the $\Nbf(F)$-fixator of ${\a}$.} 
Therefore, the first equality in the displayed equation in the Lemma follows from \ref{5} Proposition \ref{propertiesII3.3} and the second one by Proposition \ref{HRpara}.
\end{proof}
\subsubsection{Decompositions}\label{sectiondecomp}
By \cite[\S 5.2.4]{BT84}, we have the following decomposition of parahoric subgroups:
\begin{proposition}\label{decomppara}\label{KomegacapM}
For any bounded subset  $\Upomega \subset \cA$,  
the subgroup $K_{\Upomega}\subset \G(F)$ is equal to the subgroup generated by $\Mbf(F)_1=\mathcal{Z}^\circ(\cO_F)$ and $U_{\Upomega}$, more precisely it has the following decomposition 
$$K_{\Upomega}=U_{\Upomega} \Mbf(F)_1 =U_{\Upomega}^+ U_{\Upomega}^- U_{\Upomega}^+ \Mbf(F)_1=U_{\Upomega}^+ U_{\Upomega}^- (\Nbf(F)\cap K_{\Upomega}),$$
such that the factors commute in the right factorization and the product maps 
$$\begin{tikzcd}\prod_{\alpha \in \Phi_{\red}\cap \Phi^\pm} U_{\alpha+f_{{\Upomega}}(\alpha)}\arrow{r}{\sim}&U_{\Upomega}^\pm=U_{\Upomega}\cap \U(F)^\pm=K_{\Upomega}\cap \U(F)^\pm.\end{tikzcd}$$
is an homeomorphism, whatever ordering of the factors we take.
\end{proposition}
\begin{remark}\label{524corrected}
As observed by Haines \cite[\S 6]{Haines09cor}, the reference \cite[\S5.2.4]{BT84} contains a typographical error. All of the "hats" in the four displayed equations in {\em loc. cit.} should be removed.
\end{remark}
\begin{lemma}\label{NcapIM2} 
Let ${\mathfrak{b}} \subset \cA$ be an alcove. We have $\Nbf(F)\cap I_{\mathfrak{b}}=\Mbf(F)^1 \cap \G(F)_1=\Mbf(F)_1$. 
\end{lemma}
\begin{proof}
Clearly, $\Mbf(F)^1 \subset P_{\mathfrak{b}}$, hence $\Mbf(F)^1 \cap \G(F)_1 \subset \Nbf(F)\cap I_{\mathfrak{b}}$. The opposite inclusion follows from Lemma \ref{NcapIM1}. This shows the first equality. 
For the second, let $m \in \Mbf(F)^1 \cap \G(F)_1=\Nbf(F)\cap I_{\mathfrak{b}}$. 
Using the decomposition of Proposition \ref{decomppara} for $I_{\mathfrak{b}}$, write $m$ as $u_1^+ u_1^-u_2^+ m_1 \in U_{\mathfrak{b}}^+ U_{\mathfrak{b}}^- U_{\mathfrak{b}}^+ \Mbf(F)_1$, then, 
one finds that $mm_1^{-1}= \text{ad}(mm_1^{-1})(u_2^+) u_1^+ u_1^-\in \Mbf(F) \cap \U(F) \U^-(F)=\{1\}$, i.e. $m=m_1 \in \Mbf(F)_1$. 
This concludes the proof of the lemma.
\end{proof}

\subsubsection{The subgroup generated by all parahoric subgroups}\label{subgenpar}
As suggested by the condition \cite[(A 6) \S 5.2.1]{BT72}, Bruhat and Tits introduces in \cite[\S 5.2.11]{BT84} the subgroup generated by all parahoric subgroups of $\G(F)$ denoted here by $\G(F)^{\aff}$. 
They also prove in {\em loc. cit.} that $\G(F)^{\aff} = \Mbf(F)_1\cdot \varphi(\G_{sc}(F))$, where $\varphi \colon \G_{sc}\to \G,$ be the natural homomorphism from the simply connected cover of $\G^{der}$. 
For any subgroup $H\subset \G(F)$ set $H^{\aff}= H \cap \G(F)^{\aff}$.
\subsection{A double Tits system}

Set $H_1= G_1 \cap H$ for any subgroup $H\subset G$.
A good part of the following two sections relies heavily on the following theorem of Bruhat and Tits:
\begin{theorem}\label{G1titssystem}\label{Gaffadapted}
The homomorphism $\nu_N$ induces an isomorphism of groups $W_{\aff} \simeq \Nbf(F)^{\aff}/\Mbf(F)_1$. 
If $\cS \subset \Nbf(F)^{\aff}/\Mbf(F)_1$ is the subset corresponding to $\cS(\a)$, then 
the quadruple $$(\G(F)^{\aff},I_\a,\Nbf(F)^{\aff}, \cS)$$ is a double Tits system whose Weyl group is $W_{\aff}$ and the injection $  \G(F)^{\aff} \to  \G(F)$ is $I_\a$-$\Nbf(F)^{\aff}$-adapted of connected type.
\end{theorem}
\begin{proof}
This is \cite[Proposition 5.2.12]{BT84} together with the fact that $I_\a \cap \Nbf(F)^{\aff}= \Mbf(F)_1$ proved in Lemma \ref{NcapIM2}.
\end{proof}
The nomenclature "double Tits system" \cite[\S 1.2.6]{BT72} means that both $$( \G(F)^{\aff} ,I_\a,\Nbf(F)^{\aff})\text{ and }( \G(F)^{\aff}, \B(F)^{\aff},\Nbf(F)^{\aff})$$ 
are Tits systems. 
The second assertion means that \cite[\S 1.2.13]{BT72}: for any $g\in \G(F)$, there exists $h\in \G(F)^{\aff}$ such that $hI_\a h^{-1}=gI_\a g^{-1}$ and $h\Nbf(F)^{\aff} h^{-1}=g\Nbf(F)^{\aff} g^{-1}$. 
We refer the reader to the section \cite[\S 5.1]{BT72} for more details, especially Theorem 5.1.3 in {\em loc. cit.} for equivalent properties. 
Being of connected type (Definition 4.1.3 in {\em loc. cit.}) shows that the images of $\nu(\Nbf(F))$ and $\nu(\Nbf(F)^{\aff})$ in $\GL(V)$ are equal, i.e. $\Nbf(F)^{\aff} / \Mbf(F)^{\aff} \simeq W=\Nbf(F)/\Mbf(F)$.
{
Thanks to Theorem \ref{Gaffadapted} we have:
\begin{corollary}\label{BNtypecor}
\begin{enumerate}
\item If $a\in \cA$ is a special point, then by the claim below \cite[Definition 4.1.3]{BT72}, we have $ \Nbf(F) \cap P_{a} =  (\Nbf(F) \cap K_{a}) \cdot \ker \nu= (\Nbf(F) \cap U_{a}) \cdot \ker \nu $, thus $P_{a} =  \ker \nu \, \cdot U_{a}=  \ker \nu \, \cdot K_{a}$.
\item For any subset $\Upomega \subset \cA$, set $P_{\Upomega}^\dag$ for its $\G(F)$-stabilizer, i.e. $\{g\in \G(F) \colon g \cdot \Upomega= \Upomega \}$ and $N_{\Upomega}^\dag= P_{\Upomega}^\dag \cap \Nbf(F)$. 
We have 
$P_{\Upomega}^\dag=N_{\Upomega}^\dag K_{\Upomega}$. 
\end{enumerate}
\end{corollary}
\begin{proof}
See \cite[\S 4.1]{BT72} for these statement.
\end{proof}
}
\begin{corollary}\label{uniqueparahoricM}\label{iwahoriM} 
For any bounded subset $\Upomega \subset \cA$, we have 
$$\widetilde{K}_{\Upomega} \cap \Mbf(F)=\mathbb{P}_{\iota(\Upomega\times V_G)}(\cO_F) \cap \Mbf(F) = \Mbf(F)^1 \text{ and } K_{\Upomega} \cap\Mbf(F) =\Mbf(F)_1.$$
\end{corollary}
\begin{proof}
{
The equality $\widetilde{K}_{\Upomega} \cap \Mbf(F) = \Mbf(F)^1$ is clear by double inclusions, since $\Mbf(F)$ acts by translations, so any element of $\Mbf(F)$ fixing any point fixes all of $\cA_{\ext}$, i.e. is in $\Mbf(F)^1$. 
The last equality $K_{\Upomega} \cap\Mbf(F) =\Mbf(F)_1$ follows readily from the first one using Lemma \ref{NcapIM2}. 

The last equality, can also be proved using \cite[Lemma 4.1.1]{HR10}, which shows that for any facet $\cF$ in $\cA_{\ext}$ the subgroup $\Mbf(F) \cap K_{\cF}$ is a parahoric subgroup of $\Mbf(F)$ and hence must be $\Mbf(F)_1$ by Corollary \ref{uniquepara}.}
\end{proof}
\begin{lemma}\label{repofWinK}
For any special point $a \in \cA$,  the canonical injection$\begin{tikzcd}[column sep= small]\Nbf(F)\cap K_{a}/\Mbf(F)_1\arrow[r,hook]& W\end{tikzcd}$is an isomorphism. 
\end{lemma}
\begin{proof}
{
Because $a$ is special, we know by Remark \ref{specialpointvaluation} that for any $\alpha \in \Phi_{\red} $ there exists $m_\alpha \in U_{a}\cap \Nbf(F)$
whose image in $W=\langle s_\alpha\colon \alpha \in \Phi_{\red} \rangle$ is the reflection $s_\alpha$. 
}
\end{proof}
\begin{lemma}\label{staba}
Let ${\mathfrak{b}} \subset \cA$ be an alcove.
The subgroup ${P}_{\mathfrak{b}}^{\dag}=N_{\mathfrak{b}}^{\dag}  \cdot I_{\mathfrak{b}}$ is the $\G(F)$-normalizer of $P_{\mathfrak{b}}$, $\widetilde{I}_{\mathfrak{b}}$ and $I_{\mathfrak{b}}$. Moreover, 
$$N_{\mathfrak{b}}^{\dag} \cap \G(F)^{\aff}= \Mbf(F)_1 \text{ and } P_{\mathfrak{b}}^{\dag} \cap \Mbf(F)= \ker \nu.$$
\end{lemma}
\begin{proof}
{
The statement on the normalizer is clear. 
Since $P_{\mathfrak{b}}^{\dag} \cap \G(F)^{{\aff}}$ equals the $\G(F)^{{\aff}}$-stabilizer of ${\mathfrak{b}}$ which is the $\G(F)^{{\aff}}$-normalizer of $I$, but this latter is its own normalizer (in $\G(F)^{{\aff}}$), hence $P_{\mathfrak{b}}^{\dag} \cap \G(F)^{{\aff}}= I_{\mathfrak{b}}$. 
Intersecting this equality with $\Nbf(F)$ shows the first equality thanks to Lemma \ref{NcapIM2}.

Finally, an $m\in {P}_{\mathfrak{b}}^{\dag} \cap \Mbf(F)$ acts by translation, and we have $\cup_{k\in \Z_{\ge 0}} {\mathfrak{b}} + k \nu(m) \subset {\mathfrak{b}}$. 
But since ${\mathfrak{b}}$ is bounded, this inclusion can only happen if $\nu(m)=0$, i.e. ${P}_{\mathfrak{b}}^{\dag} \cap \Mbf(F)= \ker \nu$. 
We can argue also using $\Mbf(F) = \ker \nu \, \Mbf(F)^{\aff}$ (by Corollary \ref{Waffdec}), so ${P}_{\mathfrak{b}}^{\dag} \cap \Mbf(F)=\ker \nu \,  ({P}_{\mathfrak{b}}^{\dag} \cap  \Mbf(F)^{\aff})$ which equals $\ker \nu$ by the first part of this Lemma. }
\end{proof}
\begin{corollary}\label{Waffdec}
We have $\Mbf(F)^{\aff}/ \Mbf(F)_1\simeq \Lambda_{\aff}$ and 
$$ \Nbf(F)^{\aff}/\Mbf(F)_1\simeq W_{\aff} \simeq (\Mbf(F)^{\aff} / \Mbf(F)_1) \rtimes  W .$$
\end{corollary}
\begin{proof}
{
Given the isomorphism $\Nbf(F)^{\aff}/\Mbf(F)_1 \iso W_{\aff}$ (Theorem \ref{G1titssystem}),  $\Mbf(F)^{\aff}$ is precisely the subgroup of elements in $\Nbf(F)$ that act on $\cA$ by translation, hence the map $\nu$ induces an isomorphism 
$$ \Mbf(F)^{\aff}/ \Mbf(F)_1\iso \Lambda_{\aff}.$$ 
Accordingly, the isomorphism $W_{\aff}\simeq \Lambda_{\aff} \rtimes W_{\aff}^{a}$ for any special point $a\in \cA$ (\S \ref{affineweyl}) corresponds to the decomposition
$$\Nbf(F)^{\aff}/\Mbf(F)_1 =(\Nbf(F)^{\aff} \cap \Mbf(F))/ \Mbf(F)_1 \rtimes  (\Nbf(F)^{\aff} \cap K_a)/ \Mbf(F)_1.$$
This concludes the proof thanks to Lemma \ref{repofWinK}.
}
\end{proof}
\begin{corollary}\label{GG1NN1}
Let ${\mathfrak{b}} \subset \cA$ be an alcove. 
We have, 
$\G(F)= N_{\mathfrak{b}}^{\dag} \cdot \G(F)^{\aff}$ and $ \Nbf(F)= N_{\mathfrak{b}}^{\dag} \cdot \Nbf(F)^{\aff}$ and $\G(F)/\G(F)^{\aff} \simeq \Nbf(F) / \Nbf(F)^{\aff}\simeq  N_{\mathfrak{b}}^{\dag}/ \Mbf(F)_1$.
\end{corollary}
{
\begin{proof}
A immediate consequence of Theorem \ref{Gaffadapted} is $\G(F) \cdot {\mathfrak{b}}= \G(F)^{\aff} \cdot {\mathfrak{b}}$ and so $\G(F)= P_{\mathfrak{b}}^\dag \cdot \G(F)^{\aff}$, which proves the first equality given that $P_{\mathfrak{b}}^\dag=  N_{\mathfrak{b}}^\dag \cdot I_{\mathfrak{b}}$. 
Intersecting this with $\Nbf(F)$ yields the second equality of the corollary.

This shows that the natural maps $N_{\mathfrak{b}}^\dag \to \G(F)/ \G(F)^{\aff}$ and $N_{\mathfrak{b}}^\dag \to \Nbf(F)/ \Nbf(F)^{\aff}$ are both surjective and have the same kernel $\Mbf(F)_1 $ by Lemma \ref{staba}. 
\end{proof}}
\subsection{The subgroup $\G(F)_1$}\label{sectionG1}
Proposition \ref{parahr} shows that $\G(F)^{\aff} \subset \G(F)_1$ and as we will see in the following Lemma, this is actually an equality:
\begin{lemma}\label{GaffG1}
We have, $\G(F)^1=  (N_\a^\dag \cap \G(F)^1)\cdot \varphi(\G_{sc}(F)) \text{ and }\G(F)_1=\G(F)^{\aff}=\Mbf(F)_1\cdot \varphi(\G_{sc}(F))$.
\end{lemma}
\begin{proof} The equality $\G(F)_1=\G(F)^{\aff}$ was already proved in \cite[Lemma 1.3]{Richardz2016}. 
However, we present here another argument: 
{
Intersecting $\G(F)=  N_\a^\dag \cdot \G(F)^{\aff} $ of Corollary \ref{GG1NN1} with the kernel of $\nu_G$ and $\kappa_G$ yields
$$\G(F)_1=(N_\a^\dag \cap \G(F)_1) \cdot  \G(F)^{\aff}  \text{ and } \G(F)^1= (N_\a^\dag \cap \G(F)^1)\cdot \G(F)^{\aff}.$$
By Lemma \ref{staba} for $\G_{F^{un}}$ (or by \cite[Lemma 17]{HR8}), we have
$$P_{\iota(\a)}^{\dag} \cap \G(F^{un})^{\aff}=P_{\iota(\a)}^{\dag} \cap \G(F^{un})_1= \mathbb{P}_{\iota(\a \times V_G)}^{\circ}(\cO_{F^{un}}),$$
where $P_{\iota(\a)}^\dag$ is the $\G(F^{un})$-stabilizer of ${\iota(\a)}$. 
Now since clearly $P_\a^\dag \subset P_{\iota(\a)}^{\dag} \cap \G(F)$, we have
$$I_\a \subset P_\a^\dag \cap \G(F)_1=P_\a^\dag \cap \G(F^{un})_1 \cap \G(F)  \subset  \mathbb{P}_{\iota(\a \times V_G)}^{\circ}(\cO_{F^{un}}) \cap \G(F)= I_\a.$$
In conclusion, this shows $I_\a=P_\a^\dag \cap \G(F)_1$ and so $\Mbf(F)_1 =N_\a^\dag \cap \G(F)_1$ by Lemma \ref{NcapIM2}. 
Finally, recall that the last one was already seen in \S \ref{subgenpar}, which concludes the proof of the lemma.}
\end{proof}
\subsection{More on $\G(F)^1$-fixators}\label{G1stabprel}
\begin{lemma}\label{G1fixatorP}
Let $\cF \subset \cA$ be a facet containing in its closure a special point $a$, i.e. $ {K}_{\cF} \subset K_a$. We have
$$\widetilde{K}_{\cF}=\Mbf(F)^1 K_{\cF}  \text{ and }\widetilde{K}_{\cF}/{K}_{\cF}\simeq \Mbf(F)^1/\Mbf(F)_1.$$
\end{lemma}
{\begin{proof}
This Lemma generalizes \cite[Lemma 4.2.1]{R15} and \cite[Proposition 11.1.4]{HR10}. 
Here we give a different  argument:
By Corollary \ref{BNtypecor}, one has $P_{a}=K_{a} \cdot  \ker \nu$, hence
\begin{align*}
\widetilde{K}_{\cF} &=  ((K_{a} \cdot  \, \ker \nu) \cap \G(F)^1) \cap P_{\cF}&(\text{Prop. \ref{parahr}})\\
&= (K_a \cdot  (\ker \nu \cap \G(F)^1)) \cap P_{\cF} & (K_a  \subset  \G(F)^1) 
\\
&=(K_a \cdot  \Mbf(F)^1) \cap P_{\cF} &\\
&= (K_a  \cap P_{\cF})\cdot  \Mbf(F)^1 &( \Mbf(F)^1 \subset  P_{\cF})\\
&=K_{\cF}\cdot  \Mbf(F)^1.
\end{align*}
Consequently, Lemma \ref{uniqueparahoricM} yields $\widetilde{K}_{\Upomega}/{K}_{\Upomega}\simeq \Mbf(F)^1/\Mbf(F)_1$ .
\end{proof}}
\begin{remark}\label{KequalKtilde}
By Lemma \ref{NcapIM2} we see that $M^1\hra G^1$ induces an inclusion $(\Lambda_M)_{\text{tor}} \hra (\Lambda_G)_{\text{tor}}$ 
compatible with $\kappa_M$ and $\kappa_G$. 
We see then:
$$\begin{tikzcd}
(\Lambda_G)_{\text{tor}}=0 \arrow[rr, Leftrightarrow]&\hspace{1cm}\arrow[d, Rightarrow]& G^1=G_1\\
(\Lambda_M)_{\text{tor}}=0 \arrow[r, Leftrightarrow]& M^1=M_1 \arrow[r, Leftrightarrow]& \widetilde{K}_{\cF}=K_{\cF}.
\end{tikzcd}$$
where $\cF\subset \cB(\G,F)_{\red}$ is any facet containing a special point in its closure. The group $\Lambda_M$ (which plays the role of the cocharacters lattice in the split case \S \ref{SBM}.) 
has no torsion:
\begin{itemize}
\item If $\Mbf$ is a torus that splits over an unramified extension of $F$, i.e. $\G$ is unramified \cite[9.5]{Bor79}.
\item If $\G$ is semisimple and simply connected then by \cite[\S 4.6.32 and \S 5.2.9]{BT84} one has $\widetilde{K}_{\cF}= K_\cF$, 
\item See \cite[Corollary 11.1.5]{HR10} for other situations for which $\Lambda_M$ is torsion free.
\end{itemize}
\end{remark}

\section{Preliminary results around (Iwahori)--Hecke algebras}\label{PremIH}
\subsection*{Notations}
In addition to the notation fixed/adopted previously we need to introduce the following, that we hope will lighten a bit the exposition: For any algebraic $F$-groups $\H$ (bold style), we denote its group of $F$-points by the ordinary capital letter $H=\mathbf{H}(F)$. 
For any subgroup $X \subset G$ different from $M$ set $X_1:=X\cap G_1$. Recall that Lemma \ref{GaffG1} says that $G_1=G^{\aff}$, so for any such a $X$, we have $X_1=X^{\aff}$. 
For $M$, one might have $M_1=\ker \kappa_M \neq M \cap G_1$, for this reason, we keep the notation $M^{\aff}$ for $M \cap G_1$.

Set $K:=K_{a_\circ}$ for the special maximal parahoric subgroup attached to the special vertex $a_\circ \in \cA$ and $I:=I_{\a}$ be the Iwahori subgroup corresponding to the fixed alcove ${\a}\subset \cA$. 

Recall that by Lemma \ref{repofWinK}, the canonical injection $N\cap K/M\cap K \to W$ is an isomorphism. Therefore, we may and will assume from now on, that any representative in $N$ of a class in $W=N/M$ lies in $K$.

\subsection{$\flat$ is everywhere !}\label{flatpartout}
{
Let $M_1 \subset M^{\flat} \subset M$ be any open compact subgroup, by maximality of $M^1$ \cite[Proposition 1.2]{Lan96} we must have $ M^{\flat} \subset M^1$.

\begin{lemma}\label{Wtrivialontors}
The group $N$ acts trivially on $\ker \nu/M_1$. In particular, $M^{\flat}$ is normal in $N$. 
\end{lemma}
\begin{proof}
Let $n\in N$ and $m\in \ker \nu$. 
On the one hand, we clearly have $m^{-1} nmn^{-1} \in \ker \nu_N= \ker \nu$. 
On the other hand, the commutator element $m^{-1} nmn^{-1} \in G_1$. 
Therefore, Lemma \ref{staba} implies $m^{-1} nmn^{-1} \in \ker \nu \cap G_1 = M_1$ and accordingly $m M_1= nmn^{-1} M_1$.
\end{proof}

For any $X\subset G$, set $X^{\flat}=M^{\flat} X$. This procedure does not alter the structure of the object we are interested in, e.g. if $K_\cF$ is a parahoric subgroup attached to a facet $\cF \subset \cA$, then $K_\cF^{\flat}$ is still a group.

When $M^{\flat}=M_1$, the resulting objects will sometimes be described as "parahoric" and the supscript ${\flat}$ will be simply omitted, e.g. $I^{\flat}=I$ the Iwahori subgroup, $K_\cF^{\flat}=K_\cF$ the parahoric subgroup attached to a facet $\cF$, and so on ... 
When $M^{\flat}=M^1$, the resulting objects will be described as "geometric" and ${\flat}$ will be replaced simply by $1$, e.g. $I^{1} =\widetilde{I},K^{1}= \widetilde{K}$.

{
Thanks to Lemma \ref{NcapIM2}, we have an isomorphism of groups $W_{\aff}^{\flat} := N_1^\flat/M^\flat \simeq W_{\aff}$. 
Recall that $\cS(\a) = \cup_{\cF \text{ facet of }  {\a}} \mathcal{T}_{\cF}$ is the set of reflections at the walls of $\a$ and it generates $ W_{\aff}$. 
Here, 
$r$ is the number of irreducible components of the Dynkin diagram of $\Phi$. 
For any facet $\cF$ of $\a$, let $\mathcal{T}_{\cF}^{\flat}\subset \cS^{\flat} \subset W_{\aff}^{\flat}$ be the sets corresponding to $\mathcal{T}_{\cF} \subset \cS(\a) \subset W_{\aff}$.
Similarly, we put $\Lambda_{\aff}^{\flat}:=M^{\flat}M^{\aff}/M^{\flat}\simeq \Lambda_{\aff}$ and 
we define for any facet $\cF\subset \cA$, the group ${W}_{\cF}^{\flat}:= N_{1, \cF}^{\flat}/M^{\flat}\simeq {W}_{\cF}$ where $N_{1, \cF}^{\flat}:= N \cap K_\cF^{\flat}= M^{\flat} N_{1,\cF}$. 
The group ${W}_{\cF}^{\flat}$ is generated by $\mathcal{T}_{\cF}^{\flat}$. 
For example, $W_{a_\circ}^{\flat}$ is canonically isomorphic to $ W$. 
Given Lemma \ref{Waffdec}, we have
$$W_{\aff}^{\flat} = \Lambda_{\aff}^{\flat} \rtimes W_{a_\circ}^{\flat}.$$
}
\begin{corollary}\label{doubtitflat}
The quadruple $(G_1^{\flat},I^{\flat},N_1^{\flat}, \cS^{\flat})$ is a double Tits system whose Weyl group is isomorphic to $W_{\aff}$ and the injection $  G_1^{\flat} \hra  G$ is $I_\a^{\flat}$-$N_1^{\flat}$-adapted of connected type.
\end{corollary}
\begin{proof}This follows readily from Theorem \ref{G1titssystem}.
\end{proof}
Consequently, the pair $(W_{\aff}^{\flat},\cS^{\flat})$ is a Coxeter system \cite[\S 3]{BourbakiLieV} 
subject to the following properties:
\begin{enumerate}[nosep]
\item[$T_1.$] The subgroup $N_1^{\flat} \cap I^{\flat}=M^{\flat}$ is normal in $N_1^{\flat} $ and $G_1^{\flat} $ is generated by $I^{\flat}$ and $N_1^{\flat}$,
\item[$T_2.$] the elements of $\cS^{\flat}$ have order 2 and generate $
W_{\aff}^{\flat}$,
\item[$T_3.$] For all $s\in \cS^{\flat}, w\in W_{\aff}^{\flat} $, we have $sI^{\flat} w \subset I^{\flat} wI^{\flat} \sqcup I^{\flat} swI^{\flat} $,
\item[$T_4.$] For all $s\in \cS^{\flat}$, we have $sI^{\flat} s \neq I^{\flat}$.
\end{enumerate}
}
\begin{corollary}\label{cortitssystflat}
\begin{enumerate}
\item $G_1^{\flat}= \sqcup_{w \in W_{\aff}^{\flat}} I^{\flat}  w I^{\flat} $. 
\item For any facet $\cF \subset \overline{\a}$, the subgroup $K_\cF^\flat=I^\flat W_\cF^\flat I^\flat$ is its own normalizer in $G_1^{\flat}$.
\item For any two facet $\cF,\cF' \subset \overline{\a}$, and $w \in W_{\aff}^\flat$, we have $K_\cF^\flat  w K_{\cF'}^\flat = I^{\flat} W_\cF^\flat w W_{\cF'}^\flat I^{\flat} $ and the map $w \mapsto I^{\flat}wI^{\flat}$ defines a bijection
$$W_\cF^\flat \backslash W_{\aff}^{\flat} /W_{\cF'}^\flat \cong K_\cF^\flat \backslash G_1^{\flat}/ K_{\cF'}^\flat.$$
\end{enumerate}
\end{corollary}
\begin{proof}
These facts are consequences of the axioms $T_1, \dots, T_4$, for which proofs may be found in \cite[Ch. IV, \S 2 n$^\circ$ 5 \& 6]{BourbakiLieV} since for any $\cF \subset \overline{\a}$ we know that $W_\cF^\flat$ is precisely the subgroup of $W_{\aff}^{\flat}$ generated by $\cT_\cF$, the type of $\cF$.
\end{proof}

\begin{corollary}[Iwahori factorizations]\label{Iwahoridecomp}
Let ${\mathfrak{b}} \subset \cA$ be an alcove. 
The map $U_{\mathfrak{b}}^+ \times U_{\mathfrak{b}}^- \times M^{\flat} \to I_{\mathfrak{b}}^{\flat} $ is a bijection. 
In particular, $I_{\mathfrak{b}}^{\flat}$ admits the Iwahori factorization $I_{\mathfrak{b}}^{\flat}=U_{\mathfrak{b}}^+ U_{\mathfrak{b}}^- M^{\flat}$ in which the factors commute. 
For $\mathfrak{b}=\a$, the product maps $\prod_{\alpha \in \Phi^+\cap \Phi_{\red}}U_{\alpha+0}\to U_{\a}^+=:I^+ \text{ and } \prod_{\alpha \in \Phi^-\cap \Phi_{\red}} U_{\alpha+n_\alpha^{-1}}\to U_{\a}^-=:I^-$ are homeomeprisms. 
Similarly, the product maps $\prod_{\alpha \in \Phi^\pm\cap \Phi_{\red}}U_{\alpha+0}\to U_{a_\circ}^\pm = K\cap U^\pm$ are isomorphisms and $K^{\flat}=U_{a_\circ}^+ U_{a_\circ}^- \,U_{a_\circ}^+ M^{\flat}$.
\end{corollary}
\begin{proof}
The corollary is a consequence of Proposition \ref{decomppara}, Lemma \ref{NcapIM2}, \ref{homeoproduct} of Proposition \ref{propertiesII3.3} together with Example \ref{exfomega} for the values of $f_{\a}(\alpha)$ for all $\alpha \in \Phi_{\red}$.
\end{proof}
\subsection{Dominance in \texorpdfstring{$\Lambda_M^{\flat}$}{LM}}\label{Lambdasection}
Let $\Lambda_M^{\flat}:=M/M^{\flat}$. Recall that we have a canonical isomorphism (see \S \ref{Kottwitzhom})
$$\Lambda_M =M/M_1\simeq \kappa_M(M)=(X^*(Z(\widehat{\Mbf}))_{\text{In}})^{\langle \sigma \rangle}.\nomenclature[C]{$\Lambda_M$}{The finitely generated abelian group $M/M_1$}$$
This shows that $\Lambda_M$ is a finitely generated abelian group with finite torsion subgroup $(\Lambda_{M})_{\text{tor}}=M^1 /M_1$, which injects into the set of permutations of the vertices of $\a$. 
So $\Lambda_M^{\flat}$ is finitely generated abelian group with torsion subgroup $M^1/M^{\flat}$. Set $q_{tor}^{\flat}:= |M^1/M^{\flat}|$.

There exists a natural injective finite-cokernel homomorphism $ X_*(\Sbf) \hookrightarrow  X^*(Z(\widehat{\Mbf}))_{\text{In}}^\sigma$ \cite[\S 2.7]{R15}, which yields an isomorphisms
$$X_*(\Sbf)\otimes_\Z \R \simeq  X^*(Z(\widehat{\Mbf}))_{\text{In}}^\sigma\otimes_\Z \R \simeq \Lambda_M\otimes_\Z \R.$$
The map $\nu_{M}\colon M\to X_*(\Sbf)\otimes_\Z \R$ (Lemma \ref{nu}) identifies $\Lambda_M^1=M/M^1 (=\Lambda_M/(\Lambda_{M})_{\text{tor}})$ with a lattice $\underline{\Lambda}_M$ in $X_*(\Sbf)\otimes_\Z \R$ and $\text{rank}({\Lambda_M/(\Lambda_{M})_{\text{tor}}})=\dim_\R X_*(\Sbf)\otimes_\Z \R$. 
By the above isomorphisms, the extended apartment $\cA_{\ext}$ acquires the structure of an affine space over $ \Lambda_M\otimes_\Z \R\simeq \Lambda_M^{\flat}\otimes_\Z \R \simeq \underline{\Lambda}_M \otimes_\Z \R$, thus one gets an embedding $$\underline{\Lambda}_M \to \cA_{\ext}, \quad v \mapsto  \nu_M(v)+ (a_\circ, 0_{V_G}).$$
Define 
\begin{align*}M^\pm &:=\{m\in M \colon (\alpha+0)(\nu(m)+a_\circ)= 
\langle \nu_M(m) , \alpha \rangle\ge 0 , \,\, \forall \alpha \in \Phi_{\red}^\pm\}.
\end{align*}
We call $M^+$ (resp. $M^-$) the set of dominant (resp. antidominant) elements of $M$.
Let ${\Lambda}_M^{\pm,\flat}\subset {\Lambda}_M^{\flat}$ and $\underline{\Lambda}_M^\pm \subset \underline{\Lambda}_M$ denote the images of $M^{\pm,\flat}$ by the natural projections $M \twoheadrightarrow {\Lambda}_M^{\pm,\flat} \twoheadrightarrow \underline{\Lambda}_M^\pm$. 
In other words
$$\underline{\Lambda}_M^\pm:= \left((a_\circ+\underline{\Lambda}_M) \cap \cC^\pm\right)-a_\circ\subset\underline{\Lambda}_M ,$$
where $\cC^\pm$ are the two opposite vectorial chambers
$$\cC^\pm:=\{a \in \cA_{\ext}\colon (\alpha+0)(a)= \alpha(a-a_\circ)= \langle a-a_\circ, \alpha  \rangle \ge 0, \,\, \forall \alpha \in \Phi_{\red}^\pm\}.$$
\begin{remark}
An element $m\in M$ is in $M^+$ if and only if $\nu_N(m)(a_\circ)=\nu(m)+a_\circ \in \overline{\cC}^+$ (topological closure) and accordingly also $\nu_N(m)(\a)\subset \cC^+$ since by definition $\a\subset \cC^+$.
\end{remark}

\subsection{Iwahori--Weyl group}\label{Weylgroups}

Define the $\flat$-Iwahori--Weyl group for $G$ to be 
$$\widetilde{W}^{\flat}:= N/M^{\flat}.$$
Lemma \ref{Wtrivialontors} simply says that $\ker \nu/M_1\subset Z(\widetilde{W})$ the center of $\widetilde{W}$. 
We now present two semidirect product decompositions of the $\flat$-Iwahori--Weyl group. 
These decompositions will be useful for studying the Iwahori--Hecke algebra. 
Here is the first one: 
\begin{lemma}\label{weyiwa}The ${\flat}$-Weyl--Iwahori subgroup has a natural structure of a quasi-Coxeter group
$$
\widetilde{W}^{\flat}= W_{\aff}^{\flat} \rtimes \widetilde{\Omega}^{\flat}\simeq  W_{\aff}^\flat \rtimes \Lambda_G^{\flat},
$$
where $\widetilde{\Omega}^{\flat}:= N_\a^\dag /M^{\flat}$, that is the subgroup of stabilizers of $\a$ in $\widetilde{W}^{\flat}$ and $\Lambda_G^{\flat}:=G/G_1^{\flat}$.
\end{lemma}
\begin{proof}
It suffices to consider the short exact sequence\begin{tikzcd}[column sep= tiny]
1\arrow[]{r}& N_1^{\flat} /M^{\flat}\arrow[]{r}{}& N/M^{\flat} \arrow[]{r}{}& N/N_1^{\flat} \arrow[]{r} & 1\end{tikzcd}which splits given that $N_\a^\dag/M^{\flat}\simeq N/N_1^{\flat} \simeq G/G_1^{\flat}$ 
(Corollary \ref{GG1NN1}). 
\end{proof}
\begin{lemma}\label{berndecom}\label{decompwwf}
The ${\flat}$-Iwahori--Weyl group $\widetilde{W}^{\flat}$ admits the following decomposition 
$$\widetilde{W}^{\flat}=\Lambda_M^{\flat} \rtimes W_{a_\circ}^{\flat}\simeq \Lambda_M^{\flat} \rtimes W.$$ 
If $\cF\subset  \cA$ is a facet, we have a bijection of $\Lambda_M^{\flat} $-sets
$$\widetilde{W}^{\flat}/ W_{\cF}^{\flat} \cong \Lambda_M^{\flat} \times  (W/ \jmath_W(W_{\cF}))$$
where, $\jmath_W\colon N \to W=N/M$ is the quotient map.
\end{lemma}
\begin{proof}
Using the canonical isomorphism $W \simeq W_{a_\circ}$, proved in Lemma \ref{repofWinK}, one gets a short exact sequence
$$ \begin{tikzcd}[column sep= small]
1\arrow[]{r}& \Lambda_M^{\flat}\arrow[]{r}{}& \widetilde{W}^{\flat} \arrow[]{r}{}& N/M\simeq N_{1,a_\circ}^{\flat}/M^{\flat}\arrow[]{r} & 1\end{tikzcd},$$
which clearly splits, i.e. 
$\widetilde{W}^{\flat} \simeq \Lambda_M^{\flat} \rtimes W$.

For the bijection statement, fix a set of representatives $D_\cF=\{ n_{{I^{\flat}}}^\cF \in N \colon 1\le i \le |W/ \jmath_W(W_{\cF}^{\flat})|  \}$ for $W/\jmath_W(W_{\cF}^{\flat})$. 
Note that $\jmath_W(W_{\cF}^{\flat})=\jmath_W(W_{\cF})$. 
Given that $M\cap K_\cF^{\flat}=M^{\flat}$ (by Corollary \ref{uniqueparahoricM}), one easily shows that for any $x\in N$ there exist a unique pair $(m_x,n_x) \in \Lambda_M^{\flat} \times D_\cF$ such that $x N_{1,\cF}^{\flat} = m_x n_{x} N_{1,\cF}^{\flat} $. This completes the proof of the Lemma.
\end{proof}

\begin{remark}
When $M^{\flat}=M^1=\nu_{N,\ext}(\Nbf(F))$, ${\widetilde{W}}^1$ is called the extended affine Weyl group and usually denoted by ${\widetilde{W}}_{\aff}$ and it can be viewed as a group of affine-linear transformations on $\cA_{\ext}$ marked with $(a_\circ , 0_{V_G})$ as the "origin".
\end{remark}

\begin{remark}\label{commuwaffs}
(i) For any $x\in \widetilde{W}^{\flat}$, we write $m_x$ and $w_x$ for its projection in $\Lambda_M^{\flat}$ and $W_{a_\circ}^{\flat}$, respectively. 

(ii) The natural projection morphism $\square^{\flat}\colon\widetilde{W}  \twoheadrightarrow\widetilde{W}^{\flat} $ induces two isomorphisms; $W_{\aff} \iso W_{\aff}^{\flat}$ and $W_{a_\circ}\iso W_{a_\circ}^{\flat}$
for which the following diagram commutes
$$\begin{tikzcd}
\Lambda_M \rtimes W_{a_\circ}\arrow[two heads]{d}{ \square^{\flat}\times \text{Id} } &\widetilde{W}\arrow[swap]{l}{\simeq}\arrow{r}{\simeq}\arrow[two heads]{d}{\square^{\flat}}& W_{\aff} \rtimes \widetilde{\Omega} \arrow[two heads]{d}{ \text{Id}\times \square^{\flat} }\\
\Lambda_M^{\flat} \rtimes W_{a_\circ}^{\flat}&\widetilde{W}^{\flat} \arrow[swap]{l}{\simeq}\arrow{r}{\simeq}& W_{\aff}^{\flat} \rtimes \widetilde{\Omega}^{\flat}
\end{tikzcd}$$
In particular, the projection $\Lambda_M\to\Lambda_M^{\flat}$ is $W$-equivariant.
\end{remark}
\subsection{Bruhat and Iwasawa decompositions}
Let us recall Iwasawa and Bruhat decompositions for $G$:
\begin{proposition}[The Bruhat and Iwasawa decompositions]\label{bruhatdecomposition}
Let $K_\cF$ resp. $K_{\cF'}$ be parahoric subgroups associated with facets $\cF$ resp. $\cF'$ contained in the apartment associated with $\Sbf$.
$$G= U_\cF N U_{\cF'} \quad\text{(The Bruhat decomposition)},$$
$$G= U^\pm N B= U^\pm N U_\cF \quad \text{(The Iwasawa decomposition)}.$$
Moreover, the maps $n\mapsto BnB$, $n\mapsto K_\cF^{\flat} n K_{\cF'}^{\flat}$ and $n\mapsto U^\pm n K_{\cF}^{\flat}$ induce the following bijections
$$W\cong B\backslash G/B, \quad\,K_\cF^{\flat} \backslash G/K_{\cF'}^{\flat} \cong {W}_\cF^{\flat} \backslash \widetilde{W}^{\flat}/ {W}_{\cF'}^{\flat}\quad \text{ and }\quad  U^\pm \backslash G/K_{\cF}^{\flat} \cong \widetilde{W}^{\flat}/W_\cF^{\flat}.$$
\end{proposition}
\begin{proof}
\begin{itemize}
\item 
By Corollary \ref{GG1NN1} and Corollary \ref{cortitssystflat} we have $G= G_1   N_{\mathfrak{b}}^\dag = I^\flat N_1^\flat I^\flat  N_{\mathfrak{b}}^\dag$. But $N_{\mathfrak{b}}^\dag$ normlaizes $I$ (Lemma \ref{staba}) and hence
$$G= I^\flat N_1^\flat N_{\mathfrak{b}}^\dag I^\flat =U_\a N U_\a = U_{\cF} N U_{\cF'}. $$
Accordingly, the map $\widetilde{W}^\flat \mapsto K_\cF^\flat\backslash G/K_{\cF'}^\flat $ is surjective and factors through ${W}_\cF^\flat \backslash \widetilde{W}^\flat/ {W}_{\cF'}^\flat$. 
Let $w,w'\in \widetilde{W}^\flat$ (with representatives $n,n' \in N$) such that $ K_\cF^\flat n K_{\cF'}^\flat = K_\cF^\flat n' K_{\cF'}^\flat$. 
Now $K_\cF\subset G_1$, hence $\kappa_G( n M^\flat)= \kappa_G( n' M^\flat)$ so that $w= \tau \cdot w_a$ and $\tau \cdot w_a'$ for a unique $\tau \in \widetilde{\Upomega}^\flat$ and $w_a,w_a' \in W_{\aff}^\flat$, by Lemma \ref{weyiwa}. 
Consequently, $K_{\tau(\cF)}^\flat w_a K_{\cF'}^\flat = K_{\tau(\cF)}^\flat w_a' K_{\cF'}^\flat$ with $\tau(\cF) \subset \overline{\a}$, by definition of $\widetilde{\Upomega}^\flat$. So by (2) Corollary \ref{cortitssystflat}, $w_a=w_a'$ which ends the proof of the Bruhat decomposition and the bijection $K_\cF\backslash G/K_{\cF'} \simeq {W}_\cF \backslash \widetilde{W}/ {W}_{\cF'}$.  
\item Let us prove the remaining statements. 
Let ${\mathfrak{b}}$ be an alcove that contains in its closure $\cF$. 
By \cite[Theorem 5.1.3 (vi)]{BT72} together with Theorem \ref{G1titssystem}, we have
$G_1= B_1  N_1  I_{\mathfrak{b}}$. 
Using Corollary \ref{GG1NN1} and Lemma \ref{staba} again we get 
$$G=B_1  N_1  I_{\mathfrak{b}} N_{\mathfrak{b}}^\dag=U^+  N  I_{\mathfrak{b}}=U^+  N U_{\cF}.$$
Assume $n, n' \in N$ with $n'\in U^+ n K_\cF^{\flat}= U^+ K_{n \cdot \cF}^{\flat}n= U^+ U_{n \cdot \cF}^- (N \cap K_{n \cdot \cF}^{\flat}) n$, 
hence $n' \in (N \cap K_{n \cdot \cF}^{\flat})n=n (N \cap K_{ \cF}^{\flat}) $ (using $N\cap U^+U^-=\{1\}$). 
This proves that $U^+ \backslash G/K_{\cF}^{\flat} \simeq N/N_{1,  \cF}^{\flat} = \widetilde{W}^{\flat}/W_\cF^{\flat}$. 
{Finally, using the longest element $w_0\in W$, one obtains 
$G= U^- N U_\cF$ and this yields the corresponding bijection $U^- \backslash G/K_{\cF}^{\flat} \simeq \widetilde{W}^{\flat}/W_\cF^{\flat}$.}
\item  For $G= \sqcup_{w \in W}BwB$ we refer to \cite[Theorem 21.73]{milne_2017}. \qedhere
\end{itemize}
\end{proof}

\begin{corollary}[Iwasawa decomposition II]\label{Iwasawadeccor}
If $\P=\mathbf{L}\ltimes \U_P^+$ is any standard parabolic subgroup of $\G$ with Levi factor $\mathbf{L}$ and unipotent radical $\U_P^+$, then $G=P K$, i.e. $K$ is a good open compact subgroup of $G$.    
\end{corollary}
\proof
By Proposition \ref{bruhatdecomposition}, one has an Iwasawa decomposition $G= U^+ N K=\cup_{w \in N/M} U^+ M  w K $. As observed in Lemma \ref{repofWinK}, the classes of the quotient $W=N/M$ admit representatives in $K$, therefore $G= B \cdot K = P\cdot K$. See the proof of Proposition \ref{Iwahori2} for the factorization $B\cap K=  M_1 (U^+\cap K)$.\qed

\begin{proposition}[Cartan decomposition for $K$]\label{cartanK}\label{orbitLambda}The map $M\to K^{\flat}\backslash G\slash K^{\flat}$ defined by $m\mapsto K^{\flat}mK^{\flat}$, induces
a bijection $$\Lambda_M^{-,{\flat}}\cong W^{\flat} \backslash \Lambda_M^{\flat} \cong W_{a_\circ}^{\flat} \backslash \widetilde{W}^{\flat}/W_{a_\circ}^{\flat} \cong K^{\flat}\backslash G\slash K^{\flat}.$$
\end{proposition}
\begin{proof}
It suffices to prove the Lemma for $M^{\flat}=M_1$. In which case, the first bijection  is \cite[\S 6.3 Lemma]{HV15}, the second is a consequence of the decomposition of Lemma \ref{berndecom} and the third is given by Proposition \ref{bruhatdecomposition} for $\cF=\{a_{\circ}\}$.
\end{proof}

\subsection{Iwahori decompositions}
We now give an Iwahori decomposition for the special parahoric subgroup $K$. 
Note that we are no longer in the unramified case where we could have pulled up the Bruhat decomposition for the residue field of $F$.
\begin{proposition}[Iwahori decomposition of $K$ - I] \label{Iwahori1}
Let $\cF$ be a facet of the alcove $\a$. We have the decomposition 
$$K_\cF^{\flat}=\bigsqcup_{w \in {W}_\cF^{\flat}} I^{\flat}wI^{\flat}.
$$
\end{proposition}
\begin{proof}
This is an immediate consequence of $ G = \bigsqcup_{w \in \widetilde{W}^{\flat}} I^{\flat}wI^{\flat}$ (Proposition \ref{bruhatdecomposition} for $\cF=\cF'=\a$). 
\end{proof}
\begin{proposition}[Iwahori decomposition of $K$ - II]\label{Iwahori2} 
Let $\cF$ be a facet of the alcove $\a$. We have the decomposition 
$$K_\cF^{\flat}=\bigsqcup_{w\in W_\cF^{\flat}} U_\cF^+ w I^{\flat}=\bigsqcup_{w\in W_\cF^{\flat}} U_\cF^- w I^{\flat}.$$
\end{proposition}
\begin{proof} 
By the Iwahori decomposition $K_\cF^{\flat}= U_\cF^+ \cdot U_\cF^- \cdot  N_{1, \cF}^{\flat}$ one shows 
$(U^\pm \cdot N)\cap K_\cF^{\flat}=U_\cF^\pm  \cdot N_{1,\cF}^{\flat} 
$. 
Combining this equality with the decomposition $  G = \sqcup_{w \in \widetilde{W}^{\flat}} U^\pm wI^{\flat}$ (Proposition \ref{bruhatdecomposition}), it becomes clear that 
\begin{align*}
K_\cF^{\flat}&=U_\cF^\pm  \cdot N_{1, \cF}^{\flat}\cdot I^{\flat}=\sqcup_{w\in W_\cF^{\flat}} U_\cF^\pm w I^{\flat}. \qedhere \end{align*}
\end{proof}
\subsection{Conjugations}
\begin{lemma}\label{actionwonI}
Let $\cF$ be a facet of $\a$ that contains in its closure $a_\circ$. For any $n\in N_{1,a_\circ}$, we have 
$$nU_{\cF}^-n^{-1}\subset K_\cF \text{ and } (nK_\cF n^{-1}) \cap U^+ \subset I^+.$$
\end{lemma}
\begin{proof} 
Let $n\in N_{1,a_\circ}$, with image $w$ in $W$. By Lemma \ref{NactionUalphar} we have $nU_{\alpha+ f_\cF(\alpha)}n^{-1}=U_{\beta},$ where, $\beta\colon\cA \to \R$ is the affine map $ w(\alpha)+ f_\cF(\alpha)- w(\alpha)(\nu_N(n)(a_\circ)-a_\circ)$, thus $\beta =  w(\alpha)+f_\cF(\alpha)$ since $n$ fixes $a_\circ$. By Proposition \ref{decomppara}, we have for any fixed ordering $\prec$ of $\Phi^\pm\cap \Phi_{\red}$
$$nU_\cF^\pm n^{-1}=\prod_{\alpha \in \Phi^\pm\cap \Phi_{\red}}^{\prec} U_{w(\alpha)+f_\cF(\alpha)}.$$
If $w(\alpha) \in \Phi_{\red}\cap \Phi^+$ then $U_{w(\alpha)+n_\alpha^{-1}}\subset U_{w(\alpha)+0}\subset K_\cF$ (Example \ref{exfomega}). 
If $w(\alpha) \in \Phi_{\red}\cap \Phi^-$, then $U_{w(\alpha)+n_{\alpha}^{-1}}= U_{w(\alpha)+n_{w(\alpha)}^{-1}}\subset K_\cF$.  
(Lemma \ref{nalpha}). 
Therefore, $nU_{\cF}^-n^{-1}\subset K_\cF.$

For the second equality, note that $(nK_\cF n^{-1}) \cap U^+ = U_{\nu_N(n)(\cF)}^+ \subset U_{a_\circ}^+=I^+$ by Corollary \ref{Iwahoridecomp}.
\end{proof}
\begin{lemma}\label{normalizingI} 
Let ${\Upomega} \subset \cA$ be any bounded subset and $\Psi$ any closed subset of roots in $\Phi^+ \cap \Phi_{\red}$. For any $m\in M^-$, we have
$$m (U_\Psi \cap K_{\Upomega})m^{-1} \subset U_\Psi \cap K_{\Upomega}\text{ and } m^{-1} (U_{-\Psi} \cap K_{\Upomega})m \subset U_{-\Psi} \cap K_{\Upomega}.$$
In particular, if $\Psi=\Phi^+ \cap \Phi_{\red}$ one gets $m^{-1} U_{\Upomega}^{-} m \subset U_{\Upomega}^{-} \text{ and } m \, U_{\Upomega}^{+} m^{-1} \subset U_{\Upomega}^{+}$.
\end{lemma}
\begin{proof}
For any $r \in \Gamma_\alpha$, we have by Lemma \ref{NactionUalphar}
\begin{align*}mU_{\alpha+r}m^{-1}&=U_{\alpha +r-\alpha(\nu_N(m)(a_\circ)-a_\circ)}
=U_{\alpha +r-\langle \nu(m),\alpha \rangle}
\end{align*} 
and so, for all $r\in \Gamma_\alpha$
$$mU_{\alpha + r}m^{-1}\begin{cases}\subset U_{\alpha + r} &\text{if } \langle \nu(m),\alpha \rangle \le 0 (\Leftarrow m \in M^\pm \text{ and }  \alpha \in \Phi^\mp \cap \Phi_{\red})\\
\supset U_{\alpha + r}& \text{if }\langle \nu(m),\alpha \rangle \ge 0 (\Leftarrow  m \in M^\pm \text{ and }  \alpha \in \Phi^\pm \cap \Phi_{\red})\end{cases}$$
By Proposition \ref{decomppara}, we have
$$m^{-1}U_{\Upomega}^-m= \prod_{\alpha \in\Phi^-\cap \Phi_{\red}} ^{\prec}U_{\alpha+f_{\Upomega}(\alpha)  +\langle \nu(m), \alpha \rangle}.$$
for a fixed ordering ${\prec}$ of the set $\Psi$.
Therefore, if $m$ is antidominant: 
\begin{align*}m (U_\Psi \cap K_{\Upomega})m^{-1} \subset U_\Psi \cap K_{\Upomega}, \quad m^{-1}U_{\Upomega}^-m\subset U_{\Upomega}^- \quad &\text{and}\quad m \, U_{\Upomega}^{+} m^{-1} \subset U_{\Upomega}^{+} . \qedhere\end{align*}
\end{proof}

\subsection{Double cosets}

Recall that $(W_{\aff},\cS(\a))$ and $(W_{\aff}^{\flat},\cS^{\flat})$ are isomorphic Coxeter systems. 
Denote by $\ell^{\flat} \colon  W_{\aff}^{\flat} \to \N$ the length function of the Coxeter system. 
Inflate the map $\ell^{\flat}$ to a length function $\widetilde{W}^{\flat}\simeq W_{\aff}^{\flat}\rtimes \widetilde{\Omega}^{\flat} \to \N$, for which $\widetilde{\Omega}^{\flat}$ is exactly the subset of elements of length equals to $0$\footnote{Equivalently, the length of an element in $w\in \widetilde{W}$ is the number of walls between the fixed alcove $\a$ and the alcove $w(\a)$.}: $$\ell^{\flat}(w)=\ell^{\flat}(w_{\aff}) \text{ if }w=w_{\aff}u\text{ with }w_{\aff} \in W_{\aff}^{\flat}, u\in  \widetilde{\Omega}^{\flat}.$$
Furthermore, we consider the Chevalley--Bruhat (partial) order $\overset{\flat}{\le}$ \footnote{Let $(W, S)$ be a Coxeter system and define a partial order on $W$ as follows: Fix a reduced word $w = s_1s_2 \dots s_k$. We say $v \le w$ if and only if there is a reduced subword $s_{i_1} s_{i_2}\dots s_{i_j}=v$ such that $1\le i_1 <i_2 <\dots <i_j \le k$.} on $W_{\aff}^{\flat}$. 
Now extend this order to $\widetilde{W}^{\flat}$ as follows: we say $(w_1,u_1)\overset{\flat}{\le} (w_2,u_2) \in  W_{\aff}^{\flat} \rtimes  \widetilde{\Omega}^{\flat}$ if and only if $w_1\overset{\flat}{\le} w_2$ and $\lambda_1=\lambda_2$.
\begin{remark}
Thanks to Remark \ref{commuwaffs}, 
the projection $\square^{\flat}\colon \widetilde{W} \twoheadrightarrow \widetilde{W}^{\flat}$ preserves the length and the Chevalley--Bruhat order and we may and will write $\ell^{\flat}$ (resp. $\overset{\flat}{\le}$)  simply as $\ell$ (resp. ${\le}$). 
\end{remark}
\begin{remark}
{The Chevalley--Bruhat ordering of $(W_{\aff}^{\flat},\cS^{\flat})$ can be defined in several ways, their equivalence is established in \cite{Deodhar77}. 
One interesting equivalent definition is: For $w,w'\in W_{\aff}$, $w\le w'$ if there exists $\alpha_1 \cdots \alpha_k \in \Phi^+$ such that 
(i) $w'=w s_{\alpha_1} \cdots s_{\alpha_k}$ and (ii) $w s_{\alpha_1} \cdots s_{\alpha_{j-1}}(\alpha_j)\in \Phi^+$ for all $1\le j \le k$.}
\end{remark}
\begin{lemma}\label{bnrelation}
For any $w\in \widetilde{W}^{\flat}$ and $s \in \cS^{\flat}$ we have
$$I^{\flat} sI^{\flat}wI^{\flat} =\begin{cases} I^{\flat}swI^{\flat}, &\text{ if }  w<sw \quad  (\ell(sw) = \ell(w)+1),\\\hspace*{\fill}
I^{\flat}wI^{\flat}\sqcup I^{\flat}swI^{\flat}, &\text{ if }  sw <w\quad (\ell(sw) = \ell(w)-1).
\end{cases}$$
\end{lemma}
\begin{proof}
Since the quadruplet $(G_1^{\flat},I^{\flat},N_1^{\flat}, \cS^{\flat})$ is a Tits system, it verifies the axioms listed below Corollary \ref{doubtitflat}. 
Therefore, 
for any $w\in W_{\aff}^{\flat}$ and $s \in \cS^{\flat}$
$$I^{\flat}sI^{\flat}wI^{\flat} =\begin{cases} I^{\flat}swI^{\flat} , &\text{ if }  w<sw,\\\hspace*{\fill}
I^{\flat}wI^{\flat}\sqcup I^{\flat}swI^{\flat}, &\text{ if }  sw <w .
\end{cases}$$
To finish the proof, recall that elements of $ \widetilde{\Upomega}^{\flat}$ are of zero length and their lifts in $N_\a^\dag$ normalizes $I^{\flat}$, therefore for any triplet $( s , w_{\aff} ,u )\in \mathcal{S}^{\flat} \times W_{\aff}^{\flat} \times \widetilde{\Upomega}^{\flat}$, we have $I^{\flat}sI^{\flat}wI^{\flat}=I^{\flat}sI^{\flat}w_{\aff}I^{\flat}u$, which completes the proof. 
\end{proof}
\begin{remark}\label{Iwahoriaditivity}
The first equality in the lemma above can be easily generalized: for any $w,w'\in \widetilde{W}^{\flat}$ 
\begin{align*}I^{\flat} wI^{\flat}w'I^{\flat} &= I^{\flat}ww'I^{\flat}, \quad\text{ if and only if } \quad\ell(ww')=\ell(w)+\ell(w').\qedhere\end{align*}
\end{remark}
Let us end this section with the following results, they will be needed later. \ref{uppertriangulariy}.
\begin{corollary}\label{triangularitiy}
Let $x,y\in \widetilde{W}^{\flat}$ then 
$$I^{\flat}xI^{\flat}yI^{\flat}\subset \bigsqcup_{z\le y} I^{\flat}xzI^{\flat} \text{ and } I^{\flat}xI^{\flat}yI^{\flat}\subset \bigsqcup_{z\le x} I^{\flat}zyI^{\flat}.$$
\end{corollary}
\begin{proof}
One argues by induction on the size of a minimal word for $y$. Let $y=s_1 \in \cS^{\flat}$, we have
$$I^{\flat}xI^{\flat}s_1I^{\flat} =\begin{cases} I^{\flat}xs_1I^{\flat}, &\text{ if } x<xs_1 ,\\\hspace*{\fill}
I^{\flat}xI^{\flat}\sqcup I^{\flat}xs_1I^{\flat}, &\text{ if } xs_1<x.
\end{cases}$$
So it suffices to take $z \in \{y\}$ if $x<xs_1$ and $z\in\{s_1y,y\}$ if $xs_1<x$. 
Let $y=s_1\cdots s_r$ a reduced word for $y$. Put $y'=\prod_{i=2}^{r}s_{{I^{\flat}}}$, we then have
\begin{align*}I^{\flat}xI^{\flat}yI^{\flat}=I^{\flat}xI^{\flat}s_1 y'I^{\flat} \overset{\text{Lemma \ref{bnrelation}}}&{=} I^{\flat}xI^{\flat}s_1 I^{\flat}y'I^{\flat}  \\
 & \subset \begin{cases} \bigsqcup_{z\le y'} I^{\flat}x s_1 zI^{\flat}, &\text{ if } x<xs_1 ,\\
 \bigsqcup_{z\le y'} I^{\flat}x zI^{\flat}\cup  \bigsqcup_{z\le y'} I^{\flat}x s_1 zI^{\flat}, &\text{ if } xs_1<x
\end{cases}.
\end{align*}
where the last inclusion is just the recursion hypothesis. Now if $z\le y'$ then clearly $z < y$ and $s_1z \le y$, hence in both cases we have $I^{\flat}xI^{\flat}yI^{\flat} \subset \bigsqcup_{z\le y} I^{\flat}x zI^{\flat}$ for any $x \in \widetilde{W}^{\flat}$ and $y \in W_{\aff}^{\flat}$.

Now to extend this fact to any $y \in \widetilde{W}^{\flat}$, write $y=y_{\aff} u \in W_{\aff}^{\flat} \rtimes \widetilde{\Omega}^{\flat}$. Because any lift of $u$ in $N_\a^\dag$ normalizes $I^{\flat}$ one has
$$I^{\flat}xI^{\flat}y_{\aff}uI^{\flat}= I^{\flat}xI^{\flat}y_{\aff}I^{\flat}u  \subset \bigsqcup_{z\le y_{\aff}} I^{\flat}x zI^{\flat}u= \bigsqcup_{z\le y_{\aff}} I^{\flat}x zuI^{\flat}\subset\bigsqcup_{z'\le y} I^{\flat}x z'I^{\flat}.$$
This conclude the proof of the first inclusion, from which the second inclusion follows readily by upon applying the inverse map $w \mapsto w^{-1}$.
\end{proof}
\begin{corollary}\label{uppertriangulariy}
Let $x,y\in \widetilde{W}^{\flat}$. If $U^+ x \cap I^{\flat}yI^{\flat} \neq \emptyset$ then $x\le y$ in the Chevalley--Bruhat order. 
\end{corollary}
\begin{proof}
The proof resembles the proof of the claim in \cite[Lemma 1.6.1]{HKP10}. 
Let $x,y\in \widetilde{W}^{\flat}$ such that $U^+ x \cap I^{\flat}yI^{\flat} \neq \emptyset$. 
There exist then $u\in U^+$ such that $ux \in I^{\flat}yI^{\flat}$. 
Choose any $m \in M^-$ such that $m u m^{-1}\in U^+\cap I^{\flat}$. 
Therefore,
$$I^{\flat}m x I^{\flat} = I^{\flat}m u  x I^{\flat}= I^{\flat}m iyi'  I^{\flat} \subset I^{\flat}m I^{\flat}yI^{\flat}.$$
Given Corollary \ref{triangularitiy}, this shows
$I^{\flat}m x I^{\flat}  \subset \bigsqcup_{z\le y} I^{\flat}m zI^{\flat}$, 
and consequently $x \le y$.
\end{proof}

\subsection{Volumes and $q$-powers}
{
\begin{lemma}\label{sizetildeI}
Let $\cF \subset \cA$ be any facet. 
For any $n\in N$, we have $q_{\cF,n}:=[{K}_\cF n {K}_\cF:{K}_\cF]=[{K}_\cF^{\flat} n{K}_\cF^{\flat}:{K}_\cF^{\flat}]$. 
\end{lemma}
\begin{proof}
This is a generalization of \cite[Proposition 4.2.3]{R15}, in which a similar statement is proved when $n \in W$, $M^{\flat}=M^1$ and $\cF=\a$. Here, we present a different argument:

Using Lemma \ref{Wtrivialontors}, note that
${K}_\cF^{\flat} n{K}_\cF^{\flat} =K_\cF n{K}_\cF^{\flat}$, the map $K_\cF  \to {K}_\cF^{\flat} n{K}_\cF^{\flat}/{K}_\cF^{\flat}$, given by  $x \mapsto xn {K}_\cF^{\flat}$ is then clearly surjective. 
Two elements $x,y \in {K}_\cF$ give the same image if and only if $x n \in y n {K}_\cF^{\flat}= \sqcup_{m\in M^{\flat}/M_1}  y n m{K}_\cF$, if and only if there exists $m \in M^{\flat}$ such that $n^{-1} y^{-1}x n \in m {K}_\cF$.

Given that $K_\cF \subset \ker \kappa_G$, we see that
$\kappa_G (m)= \kappa_G (n^{-1} y^{-1}x n)=1$, i.e. $m\in G_1 \cap M^1= M_1$ (Lemma \ref{NcapIM2}). 
This shows that ${K}_\cF n {K}_\cF/{K}_\cF \cong  {K}_\cF/n{K}_\cF n^{-1} \cap {K}_\cF \iso {K}_\cF^{\flat} n {K}_\cF^{\flat}/{K}_\cF^{\flat}.$
\end{proof}}

\begin{lemma}\label{InIpowerq}
Let $w= s_1 \cdots s_r \in \widetilde{W}^{\flat}$ be a reduced word for $w$. Then, 
$q_w:=  q_{\a,w}= \prod_{i=1}^r q_{\a,s_i} $. Moreover, $q_w$ is a power of $q$ for any $w\in \widetilde{W}^{\flat}$.
\end{lemma}
\begin{proof}
Let $s \in \cS^{\flat}$ and $w \in  \widetilde{W}^{\flat}$ such that $\ell(sw)=\ell(w)+1$. 
By Lemma \ref{bnrelation}, we know that $I^{\flat}swI^{\flat}= I^{\flat}sI^{\flat} \cdot I^{\flat}wI^{\flat}$. 
By assumption, $s \neq w$, hence $I^{\flat}sI^{\flat} \cap I^{\flat}wI^{\flat}\neq\emptyset$ and accordingly $q_{sw}=q_sq_w$. 
Therefore, the equality $q_w:=  q_{\a,w}= \prod_{i=1}^r q_{\a,s_i} $ follows by induction on the length $r$ of $w$. 
Consequently, to show that $q_w$ is a power of $q$ it suffices to show the claim for all elements of $\cS$. 
For any $s\in \cS$ (corresponding to the affine root $\alpha_s \in \Sigma_{\aff}$) we have, as in the proof of \cite[Lemma 2.7.4]{Mac71}
\begin{align*}
IsI/I  & \cong U_{\alpha_s}sI/I  &\text{\cite[Lemma 2.6.7 (1)]{Mac71}}&\\
 & \cong U_{\alpha_s}/ U_{\alpha_s} \cap I_{\nu(s) \cdot \a} && \\
 & \cong U_{\alpha_s}/ U_{\alpha_s + 1}  &\text{\cite[Lemma 2.6.5]{Mac71}}&. \qedhere
\end{align*}
This completes the proof of the lemma since the size of the last quotient is a power of $q$.
\end{proof}
\begin{lemma}\label{powerq}
For any $n \in N$, we have a bijection
$$U^+I\cap nIn^{-1}I/I  \cong    \prod_{\Phi_{\red} \cap \Phi^+}^{\prec_+} U_{\alpha+  f_{\nu(n)\cdot \a}(\alpha)} /U_{\alpha+ \max(0, f_{\nu(n)\cdot \a}(\alpha))}$$
for any fixed ordering $\prec_\pm$ of $\Phi^\pm$. 
In particular, $|U^+I\cap nIn^{-1}I|_{I}$ is a power of $q$.
\end{lemma}
\begin{proof}
Write $n  I n^{-1} = I_{\nu(n) \cdot \a}$. 
Using Corollary \ref{Iwahoridecomp}, we have 
$$U^+I\cap I_{\nu(n) \cdot \a}I= (U^+ M_1 I^- \cap I_{\nu_N(n)\cdot \a}) I = I_{\nu_N(n)\cdot \a}^{+} I.$$
This equality yields the bijection $ I_{\nu_N(n)\cdot \a}^{+} I/I \cong I_{\nu_N(n)\cdot \a}^{+} /I_{\nu_N(n)\cdot \a}^{+} \cap I^+$.

By Corollary \ref{Iwahoridecomp} again together with \ref{homeoproduct} and \ref{5} of Proposition \ref{propertiesII3.3}, we have
$$I_{\nu_N(n)\cdot \a}^{+} =\prod_{\Phi_{\red} \cap \Phi^+}^{\prec_+} U_{\alpha+  f_{\nu(n)\cdot \a}(\alpha)} \text{ and } I_{\nu_N(n)\cdot \a}^{+} \cap I^+ = \prod_{\Phi_{\red} \cap \Phi^+}^{\prec_+}U_{\alpha+ \max(0, f_{\nu(n)\cdot \a}(\alpha))} $$
for any fixed ordering $\prec_\pm$ of $\Phi^\pm$. 
Therefore, we get a bijection 
$$  I_{\nu_N(n)\cdot \a}^{+} I/I   \cong    \prod_{\Phi_{\red} \cap \Phi^+}^{\prec_+} U_{\alpha+  f_{\nu(n)\cdot \a}(\alpha)} /U_{\alpha+ \max(0, f_{\nu(n)\cdot \a}(\alpha))}.$$
This completes the proof of the lemma since the size of each quotient in the right product is a power of $q$.
\end{proof}
\begin{corollary}\label{WInInI}
Let $n= m \cdot w$ for any $(m,w)\in \Lambda_M^- \times W_{a_\circ}$, we have $U^+I\cap nIn^{-1}I=I$.
\end{corollary}
\begin{proof}This corollary follows readily using Lemma \ref{actionwonI} and Lemma \ref{normalizingI}. We can also deduce it from  Lemma \ref{powerq}:  
Let $n= m \cdot w$ for any $m\in \Lambda_M^-$ and any $w\in W_{a_\circ}$.  
By Remark \ref{exfomega}, we have $f_{\nu_N(n) \cdot \a}(\alpha)=-\langle \nu(m),\alpha \rangle + f_{\a}(w^{-1}(\alpha))$. 
In particular, if $n \in \Lambda_M^- \cup W_{a_\circ}$, then 
$\langle \nu(m),\alpha \rangle \ge 0$ for $\alpha \in \Phi_{\red}\cap \Phi^+$ and 
$n_\alpha^{-1} \le f_{\a}(w^{-1}(\alpha)) \le 0$ then $I_{\nu_N(n)\cdot \a}^{+} \subset I$ and so $U^+I\cap I_{\nu(n) \cdot \a}I=I$ by Lemma \ref{powerq}. 
\end{proof}

\subsection{Double cosets decomposition}
\begin{lemma}\label{KmK}
Let $\cF$ and $\cF'$ be two facets of $\a$. 
For any $w \in \widetilde{W}^{\flat}$ we have
$$K_\cF^{\flat} w K_{\cF'}^{\flat}=  I^{\flat} W_{\cF} w W_{\cF'} I^{\flat}.$$
\end{lemma}
We will be interested later in the particular case  $w=m\in \Lambda_M^-$ and $\cF=\cF'$, for which we have  $K_\cF mK_\cF= \sqcup_{(w_1,w_2) \in J_{m,\cF}}Iw_1mw_2I$, where $J_{m,\cF} \simeq W_\cF^m \backslash W_\cF \times W_\cF$. 
Here, $w \in W_\cF^m$ (the isotropy subgroup of $m$ in $W_\cF$) acts on $(w_1,w_2)\in W_\cF \times W_\cF$ as follow $(w_1w, w^{-1} w_2)$.
\begin{proof}

For $w=m \in \Lambda_M^{-,\flat}$, the equality can easily be obtained by combining Proposition \ref{bruhatdecomposition} and Lemma \ref{normalizingI}. 
However, here is an argument for a general $w\in \widetilde{W}^{\flat}$.
For any $ s_1,\cdots, s_x \in \mathcal{T}_\cF^{\flat}$ and $t_1,\cdots, t_y \in \mathcal{T}_{\cF'}^{\flat}$, 
Lemma \ref{triangularitiy} shows 
$$I^{\flat} s_1\cdots s_x I^{\flat}wI^{\flat}  t_1 \cdots t_y I^{\flat}\subset  I^{\flat} W_{\cF} w W_{\cF'} I^{\flat}.$$
But $W_{\cF}^{\flat}$ (resp. $W_{\cF'}^{\flat}$) is generated by $\mathcal{T}_\cF^{\flat}$ (resp. $\mathcal{T}_{\cF'}^{\flat}$) and hence Lemma \ref{Iwahori1} yields 
$K_\cF^{\flat} w K_{\cF'}^{\flat}\subset  I^{\flat} W_{\cF} w W_{\cF'} I^{\flat}$. 
The opposite inclusion is obvious.
\end{proof}

\subsection{The modular function} 
Let $\delta_B$ be the modular function on the fixed minimal parabolic $B$ containing $M$ given by the normalized absolute value of the determinant of the adjoint action on $\text{Lie} \,U^+$:
$$\delta_B(m):=|\det (\Ad(m)_{\text{Lie}U^+})|_F,\quad \forall m\in M$$
where $|\cdot|_F$ is the fixed normalized absolute value of $F$. 
\begin{lemma}\label{qmdeltam}
Let $\Upomega\subset \cA_{\ext}$ be any bounded subset containing any alcove, then for any $m\in M^-$
$$[K_{\Upomega}mK_{\Upomega}:K_{\Upomega}]=\delta_B(m)^{-1}.$$ 
In particular, if $\Upomega$ is an alcove $\mathfrak{b}$, then $q_m=[I_\mathfrak{b}mI_\mathfrak{b}:I_\mathfrak{b}]=\delta_B(m)^{-1}$ and so $q_{w(m)}= \delta_B(m)^{-1}$ for all $m\in M^-$.
\end{lemma}
\proof
Consider the map $ i_1 \mapsto i_1 mK_{\Upomega} $ from the set $K_{\Upomega}$ to $K_{\Upomega}mK_{\Upomega}/K_{\Upomega}$. This yields a bijection
$$K_{\Upomega}mK_{\Upomega}/ K_{\Upomega}\iso K_{\Upomega}/K_{\Upomega}\cap mK_{\Upomega}m^{-1}.$$
Here, $K_{\Upomega}\cap mK_{\Upomega}m^{-1}=K_{\Upomega}\cap K_{\nu(m)(\Upomega)}$ is the $K_{\Upomega_m^{\text{conv}}}={P}(F^{un})_{\Upomega_m^{\text{conv}}}^\circ \cap G_1$ the $G_1$-fixator of the convex hull of $\Upomega_m^{\text{conv}}:=\Upomega \cup \nu(m)(\Upomega)$. Since $m\in M^-$, then $K_{\Upomega}\cap mK_{\Upomega}m^{-1}= U_{\Upomega}^-M_1 mU_{\Upomega}^+m^{-1}$ and so
\begin{align*}[K_{\Upomega}:K_{\Upomega}\cap mK_{\Upomega}m^{-1}]
&=[U_{\Upomega}^+:mU_{\Upomega}^+m^{-1}]=\delta_B^{-1}(m).
\end{align*}
For the last statement, it suffices to observe that $q_{w(m)}= q_{\cF, w^{-1}(\a)}= \delta_B(m)^{-1}$ for any $m \in M^-.$\qedhere
\begin{remark}
A similar proof of \cite[Proposition 3.2.15]{Mac71} shows that for any $m\in \Lambda_M^-$ we have
$$q_{a_\circ,m}=[KmK:K]= \delta_B(m)^{-1} \frac{\sum_{w\in W_{a_\circ}}q_{w}^{-1}}{\sum_{w\in W_{a_\circ}^m}q_{w}^{-1}}$$
where, $W_{a_\circ}^m$ is the isotropy group of $m$ in $W_{a_\circ}$.
\end{remark}
\subsection{Relative Hecke algebras}
For any ring $R$, set $\cC_c(G,R)$ for the $R$-module of locally constant and compactly supported functions $f \colon G \to R$. Let $H$ be any open compact subgroup of $G$. The group $G$ has a unique left invariant measure $\mu_H$ normalized by $H$ on $\Q$ \cite[\S 2.4]{Vigneras96}:
$$\cC_c(G,\Q)\longrightarrow  \Q, \quad f \longmapsto \int_G f(g) d\mu_H(g)
\text{ such that }
\int_G {\bf 1}_H(g)d\mu_H(g)=1.$$
The vector space $\cC_c(G,\Q)$ acquire the structure of a $\Q$-algebra without a unit, when endowed with the convolution product with respect to $\mu_H$:
$$f *_Hf'\colon x \mapsto \int_G f(g) f'(g^{-1}x) d\mu_H(g)\quad(f,f' \in \cC_c(G,\Q)).$$
\begin{remark}\label{invarianceconvolution}
The above expression of the convolution can be rewritten
\begin{align*}f *_Hf' (x) 
&=\int_G f(xg) f'(g^{-1}) d\mu_H(g).
\end{align*} 
This shows that for $f,f' \in \cC_c(G,\Q)$, if $f$ is $X$-invariant on the left and $f'$ is $Y$-invariant on the right, then $f *_Hf'$ is $X$-invariant on the left and $Y$-invariant on the right. 
\end{remark}
The group $G$ acts on the underlying $\Q$-vector space of this algebra by translation, on the left and on the right as follows:
$$\begin{tikzcd}[row sep=small] (G \times G) \times \cC_c(G,\Q)\arrow[r] &\cC_c(G,\Q)\\
((g,g'),f) \arrow[r,mapsto]&((g,g')\cdot f\colon x\mapsto f(g^{-1} xg')).
\end{tikzcd}$$
\begin{lemma}\label{actionhecke}
Let $X,Y$ be two open compact subgroups of $G$. For any $g,h \in G$, one has
\begin{align*}
(i)&& \mathbf{1}_{gY}*_H\mathbf{1}_{hX}&=|Y\cap hXh^{-1}|_H\mathbf{1}_{gYhX}\\
(ii)&& \mathbf{1}_{gY}*_H\mathbf{1}_{XhX}&=\frac{|Y\cap X|_H}{|X\cap hXh^{-1}|_H}\mathbf{1}_{gYX}*_H\mathbf{1}_{hX},\end{align*}
where, the notation $|\square|_H$ denotes the volume of $\square$ with respect to the measure $\mu_H$.
\end{lemma}
\begin{proof}(i) Observe that the function $$\mathbf{1}_{gY}*_H\mathbf{1}_{hX}(a)=\int_G \mathbf{1}_{gY}(b)\mathbf{1}_{hX}(b^{-1}a)d\mu_H(b),$$ can only be nonzero on the set $gYhX$. Let $a\in gYhX$ and write it as $a=gyhx$, thus
\begin{align*}\mathbf{1}_{gY}*_H\mathbf{1}_{hX}(a)
=|gY\cap gyhxXh^{-1}|_H
=|Y\cap hXh^{-1}|_H\end{align*}
where, the third equality holds by left invariance of $\mu_H$. 
For (ii), we have
\begin{align*}
\mathbf{1}_{gY}*_H\mathbf{1}_{XhX}&=\frac{1}{|X\cap hXh^{-1}|_H}(\mathbf{1}_{gY}*_H\mathbf{1}_X)*_H\mathbf{1}_{hX}
= \frac{|Y\cap X|_H}{|X\cap hXh^{-1}|_H}\mathbf{1}_{gYX}*_H\mathbf{1}_{hX}.\qedhere
\end{align*}
\end{proof}
\begin{example}\label{heckecalculationsbasis}
For $g,g'\in G$, write $Hg'H$ as disjoint union $\sqcup_{h'} h'g'H$, then by lemma above:
\begin{equation*}\label{6bis}\mathbf{1}_{gH}*_H\mathbf{1}_{Hg'H}= \mathbf{1}_{gHg'H}=\mathbf{1}_{\sqcup_{h'} gh'g'H}=\sum_{h'} \mathbf{1}_{gh'g'H}.\end{equation*}
If also $HgH=\sqcup_h hgH$, then
$ \mathbf{1}_{HgH}*_H\mathbf{1}_{Hg'H}=( \sum_{h}\mathbf{1}_{hgH})*_H\mathbf{1}_{Hg'H}=\sum_{h,h'} \mathbf{1}_{hgh'g'H}$.
\end{example}
Now, we associate to the pair $(G,H)$ different $\Z$-modules. 
\begin{definition}
For any commutative ring $A$, we define $\cC_c(G/H,A)$ to be the $A$-module of compactly supported functions $f \colon G \to A$ that are $H$-invariant on the right. 
It has the following canonical basis $\{{\bf 1}_{gH} \colon g \in G/H\}$. 
In addition, the group left action of $G$ on $\cC_c(G,\Q)$ as defined above, restricts to a left action on $\cC_c(G/H,\Z)$. 
We also define $\cH_c(G\sslash H,A)\subset \cC_c(G/H,A)$ to be the $A$-algebra\footnote{By Remark \ref{invarianceconvolution} and Example \ref{heckecalculationsbasis}, we see that $\cH_{H}(A):=\cC_c(G\sslash H,\Z)\subset \cC_c(G\sslash H,\Q)$ is stable under $*_H$.} of functions $f \colon G \to A$, 
that are also $H$-invariant on the left. 
We call $\cH_c(G\sslash H,A)$ the Hecke algebra relative to $H$ with values in $A$ and denote $\cH_H(A)$. 
We have $\cH_H(A)=\cH_H(\Z)\otimes_{\Z} A$. 
\end{definition}
The Hecke algebra $\cH_H(\Z)$ is a free $\Z$-module, it has the canonical basis $\{{\bf 1}_{HgH}\colon g\in HgH\}$. 
\begin{example}
According to Proposition \ref{bruhatdecomposition}, for any facet $\cF\subset \cA$, the Hecke algebra $\cH_{K_\cF^{\flat}}(\Z)$ has the following canonical $\Z$-basis
$$\{{\bf 1}_{K_\cF^{\flat} w K_\cF^{\flat}} \text{ for } w\in W_{\cF}^{\flat} \backslash \widetilde{W}^{\flat}/ W_{\cF}^{\flat}\}.
$$
When $\cF=\{a_\circ\}$, $\cH_{K^{\flat}}(\Z)$ will be called the ${\flat}$-special--Hecke algebra. 
By Lemma \ref{orbitLambda}, one has $\Lambda_M^{-,{\flat}} \cong W_{a_\circ}^{\flat} \backslash \widetilde{W}^{\flat}/W_{a_\circ}^{\flat} $, so one can exhibit a natural $\Z$-basis for $\cH_{K^{\flat}}(\Z)$ as follows $\{h_{m}^{\flat}:={\bf 1}_{K^{\flat} mK^{\flat}} \text{ for } m \in \Lambda_M^{-,{\flat}}\}
$. 
When $\cF=\a$, $\cH_{I^{\flat}}(\Z)$ is called the ${\flat}$-Iwahori--Hecke algebra and the following set forms a $\Z$-basis $\{i_w^{\flat}:={\bf 1}_{I^{\flat}wI^{\flat}} \text{ for } w\in \widetilde{W}^{\flat}\}.
$

When we are in the case $M^{\flat}=M_1$, we systematically omit the prefix ${\flat}$ in all of the above nomenclatures and notations.
\end{example}
Let $\End_{\Z[G]} \cC_c(G/H,\Z)$ denote the ring of $G$-equivariant endomorphisms of $\cC_c(G/H,\Z)$.
\begin{proposition}\label{ringintertwiners}The following map is an isomorphism of rings\footnote{the superscript {\em opp} indicates the opposite ring}:
$$\begin{tikzcd}[row sep= small]\cH_H(\Z) \arrow{r}{\simeq} & \End_{\Z[G]} \cC_c(G/H,\Z)^{\text{\emph{opp}}}\\ h \arrow[r, mapsto] &e_h\colon f \mapsto f*_H h,\end{tikzcd}$$
\end{proposition}
\begin{proof}
We first prove that the above map is an isomorphism of $\Z$-modules. {The injectivity being clear, we prove surjectivity. Any $G$-equivariant endomorphism $e$ of $\cC_c(G/H,\Z)$ is uniquely defined by $e({\bf 1}_H)$. Since ${\bf 1}_H$ is $H$-invariant on the left, $e({\bf 1}_H)$ must also be $H$-invariant on the left. This shows that $e$ is the image of the Hecke element $e({\bf 1}_H)$}.

Finally, for any $h_1,h_2 \in \cH_H(\Z)$, we clearly have 
\begin{align*}e_{h_1 *_H h_2}({\bf 1}_H)
&=e_{h_2}\circ e_{h_1} ({\bf 1}_H)
\end{align*}
As we previously said, elements of $\End_{\Z[G]} \cC_c(G/H,\Z)$ are uniquely determined by the image they give to ${\bf 1}_H$, this shows that $e_{h_1 *_H h_2}=e_{h_2}\circ e_{h_1}$ and ends the proof of the lemma.
\end{proof}
\begin{remark}\label{heckeinvolution}
The map $$\begin{tikzcd}[row sep= small]\cH_H(\Z) \arrow{r}{\simeq} & \cH_H(\Z)\\ h \arrow[r, mapsto] &h^\vee \colon g \mapsto h(g^{-1}),\end{tikzcd}$$
is an involution and for any $h_1,h_2  \in \cH_H(\Z)$ we have $(h_1*_Hh_2)^\vee = h_2^\vee *_H h_1^\vee$.
\end{remark}
The following lemma is meant to clarify the multiplicative structure of the algebra $\cH_H(\Z)$.
\begin{lemma}\label{relativeheckerule}
For $g,g' \in G$, we have
$${\bf 1}_{HgH}*_H {\bf 1}_{Hg'H}=\sum_{g''\in C_{g,g'}} c^H(g,g',g'') {\bf 1}_{Hg''H}, \quad c^H(g,g',g''):=\big|HgH\cap g''Hg'^{-1}H\big|_H$$
where, $C_{g,g'}$ denotes a set of representatives for $H\backslash HgHg'H/H$.
\end{lemma}
\begin{proof}First, we know that the function ${\bf 1}_{HgH}*_H {\bf 1}_{Hg'H}$ is $H$-biinvariante (Remark \ref{invarianceconvolution}), therefore, using the canonical basis of $\cH_H(\Z)$, it can be written as follows
$${\bf 1}_{HgH}*_H {\bf 1}_{Hg'H}=\sum_{g''\in C_{g,g'}} c^H(g,g',g'') {\bf 1}_{Hg''H}, \quad c^H(g,g',g'')\in \Z_{\ge 0} $$
for some finite set $C_{g,g'}$ and  integral coefficients $c^H(g,g',g'')\neq 0 $, for each $g''\in C_{g,g'}$. 

Secondly, the integral defining ${\bf 1}_{HgH}*_H {\bf 1}_{Hg'H}(a)$
is nonzero only if $a\in HgHg'H$. Therefore, if $Hg''H\subset HgHg'H$, we have $c^H(g,g',g'')={\bf 1}_{HgH}*_H {\bf 1}_{Hg'H}(g'')=\big|HgH \cap g'' Hg'^{-1}H\big|_H\in \Z_{\ge 0}$. 
For each double coset $Hg''H\subset HgHg'H$, write $g''= gh'' g'$ for some $h'' \in H$, thus
$$c^H(g,g',g'') =\big|HgH \cap gh'' g'Hg'^{-1}H\big|_H \ge \big|gH\big|_H=1.$$
This shows that $C_{g,g'}$ is indeed a set of representatives for $H\backslash HgHg'H/H$.
\end{proof}
\begin{remark}\label{morphismhecktoZ}
For any $g,g'\in G$, we have
\begin{align*}
\int_G {\bf 1}_{HgH}*_H {\bf 1}_{Hg'H}(a)d\mu_H(a)
&=\int_G \left(\int_G {\bf 1}_{Hg'H}(b^{-1}a)d\mu_H(a)\right){\bf 1}_{HgH}(b)d\mu_H(b)\\
&=\left|Hg'H \right|_H\int_G {\bf 1}_{HgH}(b)d\mu_H(b)
=\left|HgH \right|_H\left|Hg'H \right|_H.
\end{align*}
This shows that the linear functional $\begin{tikzcd}[row sep= small]
d_H\colon \cH_H(\Z) \arrow[r] & \Z,\end{tikzcd}$
defined on the canonical basis elements by ${\bf 1}_{HgH} \longmapsto\left|HgH \right|_H$, is an homomorphism of rings. Consequently,  
$${\bf 1}_{HgH}\in \cH_H(\Z)^\times \Longleftrightarrow d_H\left({\bf 1}_{HgH}\right)=1  \Longleftrightarrow g\in N_G(H).\qedhere$$ 
\end{remark}
{\begin{remark}Let $\Upomega\subset \cA_{\ext}$ be any bounded subset containing the alcove $\a$. Proposition \ref{KomegacapM} shows that the product map yields an isomorphism $U^-_{\Upomega} \times M_{1} \times U^+_{\Upomega} \iso K_{\Upomega}$.

For any $m_1,m_2 \in M^-$, applying Lemma \ref{normalizingI} shows that 
$K_{\Upomega} m_1K_{\Upomega} m_2K_{\Upomega}= K_{\Upomega} m_1m_2K_{\Upomega}$ and $K_{\Upomega}m_1K_{\Upomega} \cap m_1m_2 K_{\Upomega}m_2^{-1} K_{\Upomega}= m_1 K_{\Upomega}$. 
Accordingly, by Lemma \ref{relativeheckerule}, we have
$$
{\bf 1}_{K_{\Upomega} m_1K_{\Upomega}}*_{K_{\Upomega}}{\bf 1}_{K_{\Upomega} m_2K_{\Upomega}}={\bf 1}_{K_{\Upomega} m_1m_2K_{\Upomega}}.
$$
Therefore, we get an homomorphism of rings \begin{tikzcd}\Z[\Lambda_M^-]\arrow{r}& \cH_{K_{\Upomega}}(\Z)\end{tikzcd} with commutative image.
\end{remark}

\begin{lemma}\label{identificationsubsuphecke}Let $H'\subset H$ be two open compact subgroups of $G$ (e.g. ${H'}=I$ and ${H}=K$). 
For any ring $R$ in which $[H:H']$ is invertible, we have a natural isomorphism of algebras 
$$\begin{tikzcd}[row sep= small]
(\cH_{H}(R),  +,*_{H}) \arrow{r}{\simeq} &(e_{H} *_{H'} \cH_{H'}(R) *_{H'} e_{H} ,+,*_{H'})  , \quad h \arrow[mapsto]{r}& {|{H}|_{H'}}^{-1} h.\end{tikzcd}$$
where $e_{H} :={|{H}|_{H'}}^{-1}{\bf 1}_{H}$ is an idempotent of the relative Hecke algebra $ \cH_{H'}(R)$.
\end{lemma}
\proof This is clear. \qed
\subsection{"Averaging" maps are homomorphisms}
\begin{lemma}\label{flatKcenter}
Let $\cF\subset \cA$ be any facet. 
For any $w \in N$, we have
$${\bf 1}_{ {K}_\cF^{\flat} w  {K}_\cF^{\flat}}=\mathbf{1}_{K_\cF wK_\cF} *_{K_\cF} \mathbf{1}_{{K}_\cF^{\flat}}=   {\bf 1}_{{K}_\cF^{\flat}} *_{K_\cF} {\bf 1}_{K_\cF wK_\cF} = \mathbf{1}_{K_\cF w{K}_\cF^{\flat}}= \sum_{m \in M^{\flat}/M_1}{\bf 1}_{K_\cF wmK_\cF}.$$
In particular, $\mathbf{1}_{{K}_\cF^{\flat}} \in Z(\cH_{K_{\cF}}(\Z))$.
\end{lemma}
\begin{proof}
Using Lemma \ref{actionhecke}, one shows that for any $w \in N$
$$|{K}_\cF \cap w {K}_\cF^{\flat} w^{-1} |_{{K}_\cF}\mathbf{1}_{{K}_\cF w{K}_\cF^{\flat}}= |{K}_\cF\cap w{K}_\cF w^{-1}|_{{K}_\cF}\mathbf{1}_{{K}_\cF w{K}_\cF} *_{{K}_\cF} \mathbf{1}_{{K}_\cF^{\flat}} $$
Note that ${K}_\cF \cap w {K}_\cF^{\flat} w^{-1} \subset G_1$, hence it equals ${K}_\cF \cap w ({K}_\cF^{\flat}\cap G_1) w^{-1}= {K}_\cF \cap w {K}_\cF w^{-1}$. So $\mathbf{1}_{{K}_\cF w{K}_\cF} *_{{K}_\cF} \mathbf{1}_{{K}_\cF^{\flat}}= \mathbf{1}_{{K}_\cF w{K}_\cF^{\flat}} .$ 
By Lemma \ref{Wtrivialontors}, we have ${K}_\cF w {K}_\cF^{\flat} = {K}_\cF^{\flat} w{K}_\cF^{\flat}$, hence
\begin{align*} {\bf 1}_{{K}_\cF w {K}_\cF^{\flat}} 
&={\bf 1}_{ {K}_\cF^{\flat} w  {K}_\cF^{\flat}}=  {\bf 1}_{{K}_\cF^{\flat}} *_{{K}_\cF}{\bf 1}_{{K}_\cF w{K}_\cF}.
\end{align*}
Finally, since $M^{\flat}$ normalizes $K_\cF$ and ${K}_\cF^{\flat}= \sqcup_{m\in M^{\flat}/M_1} m{K}_\cF$, one has 
$\sum_{m \in M^{\flat}/M_1}{\bf 1}_{{K}_\cF wm{K}_\cF}={\bf 1}_{{K}_\cF w {K}_\cF^{\flat}}$.
\end{proof}
\begin{corollary}\label{ontoaveraging}
The "averaging" map $\iota^{\flat}\colon\cH_{K_\cF}(\Z)\to \cH_{{K}_\cF^{\flat}}(\Z)$ defined on basis elements by
$${\bf 1}_{K_\cF wK_\cF}\mapsto \mathbf{1}_{{K}_\cF^{\flat} w {K}_\cF^{\flat}}$$
is a surjective homomorphism of algebras. 
\end{corollary}
\begin{proof}
Clearly it is a surjective morphism of $\Z$-modules. 
It suffices then to verify that $\iota$ respects the algebra structure. 
For any two elements $w,w'\in N$, Lemma \ref{flatKcenter} shows
\begin{align*}\iota({\bf 1}_{{K}_\cF w{K}_\cF } *_{{K}_\cF} {\bf 1}_{{K}_\cF w'{K}_\cF })&=  {\bf 1}_{{K}_\cF w{K}_\cF} *_{{K}_\cF} {\bf 1}_{{K}_\cF w'{K}_\cF} *_{{K}_\cF} {\bf 1}_{{K}_\cF^{\flat}} \\
&={|M^{\flat}/M_1|}^{-1}{\bf 1}_{{K}_\cF w{K}_\cF} *_{{K}_\cF} {\bf 1}_{{K}_\cF^{\flat}}*_{{K}_\cF} {\bf 1}_{{K}_\cF w'{K}_\cF}*_{{K}_\cF} {\bf 1}_{{K}_\cF^{\flat}} \\
&= {\bf 1}_{{K}_\cF^{\flat} w{K}_\cF^{\flat}}*_{{K}_\cF^{\flat}} {\bf 1}_{{K}_\cF^{\flat} w'{K}_\cF^{\flat}} \qedhere
\end{align*}
\end{proof}
\subsection{Iwahori--Matsumoto presentation and other properties}\label{HeckeIwahoriAlg}
Recall that $\cS \subset N_1/M_1$ corresponds to the set of orthogonal reflections with respect to the walls of the fixed alcove $\mathfrak{a}$ (Using the identification $N_1/M_1 \simeq W_{\aff}$). 
\begin{theorem}[The Iwahori--Matsumoto Presentation]\label{IMpresentation}
The Iwahori--Hecke ring $\cH_{I^{\flat}}(\Z)$ is the free $\Z$-module with basis $(i_w^\flat)_{w\in \widetilde{W}^{\flat}}$ endowed with the unique ring structure satisfying
\begin{itemize}[topsep=0pt]
\item\label{braidrel}
The braid relations: 
$i_w^{\flat} *_{I^{\flat}} i_{w'}^{\flat} = i_{ww'}^{\flat} \text{ if }w, w'\in \widetilde{W}^{\flat}\text{ such that }\ell(w) + \ell(w') = \ell(ww')$.
\item The quadratic relations: 
$(i_s^{\flat})^2 = q_s  i_1^{\flat}+ (q_s - 1)i_s^{\flat}\text{ if }s \in \cS^{\flat}$.
\end{itemize}
\end{theorem}
This is a very known result, at least for Iwahori--Hecke attached to extended affine Weyl groups. See for example \cite[\S 6.3]{garrett:buildings}, \cite{Geck-Pfeiffer},  \cite{curtis-reiner},\cite{Carter93}.
For the general case, it was showed in \cite[Theorem 2.1]{vigneras_2016} for $\Z$-coefficients or \cite[Proposition 4.1.1]{R15} for $\C$-coefficients. 
However, since we have enough ingredients, we prove here using a slightly different argument:}
\begin{proof}
{
 \begin{enumerate}
\item The Braid relation : Let $w, w'\in \widetilde{W}^{\flat}$ such that $\ell(w) + \ell(w') = \ell(ww')$, then by Remark \ref{Iwahoriaditivity} together with Remark \ref{relativeheckerule} we obtain that 
$i_w^{\flat} *_{I^{\flat}}i_{w'}^{\flat} = c^{I^{\flat}}(w,w',ww') i_{ww'}^{\flat}$ where $c^{I^{\flat}}(w,w',ww')=|I^{\flat}wI^{\flat} \cap ww' I^{\flat} w'^{-1}I|_{{I^{\flat}}}$. 
Combining Remark \ref{morphismhecktoZ} and Remark \ref{Iwahoriaditivity} again shows that $c^{I^{\flat}}(w,w',ww')$ must be $1$.
\item For the quadratic relations one has by  Lemma \ref{bnrelation} $I^{\flat}sI^{\flat}sI^{\flat}= I^{\flat} \cup I^{\flat}sI^{\flat}$ for any $s \in \mathcal{S}$. 
Thus, $(i_s^{\flat})^2 = c^{I^{\flat}}(s,s,1)  i_1^{\flat}+ c^{I^{\flat}}(s,s,s) i_s^{\flat}$ where $c^{I^{\flat}}(s,s,1)=|I^{\flat}sI^{\flat}|_{I^{\flat}}=q_s$ and $c^{I^{\flat}}(s,s,s)=|I^{\flat}sI^{\flat}\cap sI^{\flat}sI^{\flat}|_{I^{\flat}}$ by Remark \ref{relativeheckerule}.

Clearly $ sI^{\flat}sI^{\flat} \subset I^{\flat}sI^{\flat}sI^{\flat} = I^{\flat} \sqcup I^{\flat} s I^{\flat}$, thus $sI^{\flat}sI^{\flat} \setminus I^{\flat} \subset I^{\flat} s I^{\flat}$ but since $I^{\flat} \cap  I^{\flat}sI^{\flat} = \emptyset$ we must have 
$I^{\flat}sI^{\flat}\cap sI^{\flat}sI^{\flat} = sI^{\flat}sI^{\flat} \setminus I^{\flat}$, hence $|I^{\flat}sI^{\flat}\cap sI^{\flat}sI^{\flat}|_{I^{\flat}}= |sI^{\flat}sI^{\flat}|_{I^{\flat}}-1=|I^{\flat}sI^{\flat}|_{I^{\flat}}-1=q_s-1$.\qedhere
\end{enumerate}
}
\end{proof}
{{
\begin{lemma}\label{equiviIellq}
For any $w,w'\in \widetilde{W}^{\flat}$, we have 
$$i_w^{\flat} *_{I^{\flat}} i_{w'}^{\flat}=i_{ww'}^{\flat} \Leftrightarrow I^{\flat} ww'I^{\flat}  = I^{\flat} wI^{\flat} w'I^{\flat} \Leftrightarrow \ell(w) + \ell(w') = \ell(ww')\Leftrightarrow q_wq_{w'}=q_{ww'} .$$
\end{lemma}
\begin{proof}
The first two equivalences can be easily deduced using Remark \ref{Iwahoriaditivity} and Lemma \ref{relativeheckerule}. Let us prove the last one. 
If $i_w^{\flat}*_{I^{\flat}} i_{w'}^{\flat} = i_{ww'}^{\flat}$, Remark \ref{morphismhecktoZ} implies that $ q_wq_{w'}=q_{ww'}$. 
Now, assume that $ q_wq_{w'}=q_{ww'}$, then by Remark \ref{morphismhecktoZ} again and Lemma \ref{relativeheckerule}
$$
(c^{I^{\flat}}(w,w',ww')-1)\int_G i_{ww'}^{\flat}(a)d\mu_{I^{\flat}}(a)   +\sum_{w''\in C_{w,w'}\setminus \{ww'\}} c^{I^{\flat}}(w,w',w'') \int_G i_{w''}^{\flat}(a)d\mu_{I^{\flat}}(a)=0.$$
As we saw in the proof of Theorem \ref{IMpresentation}, we have $c^{I^{\flat}}(w,w',ww')=1$. 
Accordingly,
$$\sum_{w''\in C_{w,w'}\setminus \{ww'\}} c^{I^{\flat}}(w,w',w'') \int_G i_{w''}^{\flat}(a)d\mu_{I^{\flat}}(a)=0.$$
But if  $\{ww'\} \subsetneq C_{w,w'}$, then all $c^{I^{\flat}}(w,w',w'')=\big|I^{\flat}wI^{\flat}\cap w''I^{\flat}w'^{-1}I^{\flat}\big|_{I^{\flat}}\ge 1$. 
This can not be possible, unless $C_{w,w'}= \{ww'\}$, i.e. $I^{\flat} ww'I^{\flat}  = I^{\flat} wI^{\flat} w'I^{\flat} $, i.e. $i_w^{\flat} *_{I^{\flat}} i_{w'}^{\flat}=i_{ww'}^{\flat} $. 
This concludes the proof of the lemma.
\end{proof}}}
The above Iwahori--Matsumoto presentation yields the following consequences:
\begin{corollary}\label{additionalIMpres}
\begin{enumerate}[(i)]
\item For any $w,w'\in \widetilde{W}^{\flat}$, if $\ell(w)+\ell(w')=\ell(ww')$ then  
$I_{w^{-1}(\a)}I \cap I_{w'(\a)}I=(I_{w^{-1}(\a)} \cap I_{w'(\a)})I=I$. 
\item\label{groupalgomegaH}   
The $\Z$-linear map $z\mapsto i_z^{\flat}$ embeds the group algebra $\Z[\widetilde{\Omega}^{\flat}]$ into $\cH_{I^{\flat}}(\Z)$.
\item\label{actiongroupalgG} For all $w \in W_{\aff}^{\flat}$ and $z\in \widetilde{\Omega}^{\flat}$, one has $i_w^{\flat}i_z^{\flat}
=i_z^{\flat} i_{z^{-1}wz}^{\flat}$.
\item Let $w \in W_{\aff}^{\flat}$ and $w = s_1s_2 \cdots s_{\ell(w)}$ a reduced expression. 
Using Lemma \ref{bnrelation}, one has 
\begin{equation}\label{qwreduced}q_w = q_{s_1}q_{s_2} \cdots q_{s_{\ell(w)}}.\end{equation} 
\item\label{v}
Let $w= w_{\aff} \cdot z \in \widetilde{W}^{\flat}$ with $w_{\aff}=s_1 \cdots s_{\ell(w)}$ a reduced word in $W_{\aff}^{\flat}$ and $z\in \widetilde{\Omega}^{\flat}$, we have 
\begin{align*}(i_w^{\flat})^{-1}
=i_{z^{-1}}^{\flat} (i_{w_{\aff}}^{\flat})^{-1}=q_w^{-1} {i_{z^{-1}}^{\flat}(i_{s_{\ell(w)}}^{\flat}-q_{s_{\ell(w)}}+1) \cdots (i_{s_{1}}^{\flat}-q_{s_{1}}+1)}.
\end{align*}
which lives in $\cH_{I^\flat}(\Z[q^{-1}])$ given Lemma \ref{InIpowerq}.
\item If we do not include $q^{-1}$ in the coefficients ring, we still have $$i_w^{\flat} (i_w^{\flat})^*= (i_w^{\flat})^* i_w^{\flat}=q_{w} , \quad\text{in }\cH_{I^{\flat}}(\Z),$$
where, $(i_w^{\flat})^*:=(i_{s_{\ell(w)}}^{\flat}-q_{s_{\ell(w)}}+1) \cdots (i_{s_{1}}^{\flat}-q_{s_{1}}+1)$.
\item For any $s \in \cS^{\flat}$, set $\mathscr{P}_s^l:=\{w\in \widetilde{W}^{\flat} \colon w<sw, \, i.e. \,  \ell(sw) = \ell(w)+1 \}$ and $\mathscr{P}_s^r:=\{w\in \widetilde{W}^{\flat} \colon w > ws, \, i.e. \,  \ell(ws) = \ell(w)- 1 \}$ then $(\mathscr{P}_s^l)^{-1}= \mathscr{P}_s^r$ 
and 
$\widetilde{W}=\mathscr{P}_s^l \sqcup  s\mathscr{P}_s^l=\mathscr{P}_s^r \sqcup  \mathscr{P}_s^r s$ \cite[Ch. IV, \S 1.7]{BourbakiLieV}. 
For any $s \in \cS^{\flat}$ and any $w\in \widetilde{W}^{\flat}$, we have
 $$i_s^{\flat}*_{{I^{\flat}}} i_w^{\flat} =  \begin{cases} i_{sw}^{\flat}, &\text{ if } w\in \mathscr{P}_s^l ,\\\hspace*{\fill}
 q_s  i_w^{\flat}+ (q_s - 1)i_{sw}^{\flat}, &\text{ if } w\in  s\mathscr{P}_s^l .
\end{cases} \text{ and } i_w^{\flat} *_{{I^{\flat}}} i_s^{\flat}=  \begin{cases} i_{ws}^{\flat}, &\text{ if } w\in \mathscr{P}_s^r ,\\\hspace*{\fill}
 q_s  i_w^{\flat}+ (q_s - 1)i_{ws}^{\flat}, &\text{ if } w\in  \mathscr{P}_s^rs.
\end{cases}$$
\end{enumerate}
\end{corollary}
\begin{proof}  These properties follow readily from Theorem \ref{IMpresentation}.\end{proof} 
\begin{proposition}\label{generalizedIWdecomp}
Set $\cH_{I^{\flat}}^{\aff}(\Z):=\cH(G_{1}^{\flat} \sslash I^{\flat},\Z)$ and let $\cH_{I^{\flat}}^{\Upomega}(\Z)$ be the subalgebra of functions supported on cosets of the form $\sigma I^{\flat}$ for $\sigma \in \widetilde{\Upomega}^{\flat}$. 
We have an isomorphism of $\Z$-algebras
$$\cH_{I^{\flat}}(\Z) \simeq \Z[ \widetilde{\Upomega}^{\flat}]\otimes^\prime \cH_{I^{\flat}}^{\aff}(\Z)$$ 
where, $\otimes^\prime$ indicates that the tensor modules is endowed with a "twisted" multiplication: for any $\sigma,\sigma' \in \widetilde{\Upomega}^{\flat}$ and any $w,w'\in \widetilde{W}^{\flat}$
$$(\sigma \otimes i_w)(\sigma' \otimes i_{w'})=\sigma\sigma' \otimes i_{\sigma'^{-1}w \sigma'}*_{I^{\flat}} i_{w'} .$$
\end{proposition}
\begin{proof}
This an immediate consequence of Proposition \ref{IMpresentation} and Lemma \ref{weyiwa} together with \ref{groupalgomegaH} and \ref{actiongroupalgG} of Corollary \ref{additionalIMpres}.
\end{proof}
\begin{remark}\label{9111}
The restriction of the surjective morphism (Corollary \ref{ontoaveraging}) of $\iota^{\flat}\colon\cH_{I}({\Z})\surjto \cH_{I^{\flat}}({\Z})$ to $\cH_{I}^{\aff}(\Z)$ is injective (Remark \ref{commuwaffs}) and yields an isomorphism of algebras $\cH_{I}^{\aff}(\Z) \iso \cH_{I^{\flat}}^{\aff}(\Z)$ since $G_{1}^{\flat} \sslash I^{\flat} \simeq G_{1} \sslash I$.
Thanks to Remark \ref{commuwaffs}, the following diagram is commutative
$$
\begin{tikzcd}
\cH_{I}(\Z) \arrow{r}{\simeq}\arrow[two heads, swap]{d}{\iota^{\flat}} & \Z[ \widetilde{\Upomega}]\otimes^\prime \cH_{I}^{\aff}(\Z)\arrow[two heads]{d}{ \square^{\flat} \otimes \iota^{\flat}}  \\
\cH_{I^{\flat}}(\Z) \arrow{r}{\simeq } & \Z[ \widetilde{\Upomega}^{\flat}]\otimes^\prime \cH_{I^{\flat}}^{\aff}(\Z)
\end{tikzcd}$$
\end{remark}
\subsection{Root data}\label{rootdatata}
Recall that we have defined in \S \ref{affineweyl} the root system $\Sigma$ and that we have $\Lambda_{\aff}^{\vee} = \Z[ \Sigma ]$. 
Set $\Sigma^{\flat} \subset \Lambda_{\aff}^{\flat}$  for the subset corresponding to $\Sigma$ under the isomorphism $\Lambda_{\aff}\simeq \Lambda_{\aff}^{\flat}$ (\S \ref{flatpartout}). 
We consider the root datum $\cD_{\aff}^{\flat}:=({\Lambda}_{\aff}^{\flat}, \Sigma^{\flat,\vee},{\Lambda}_{\aff}^{\flat,\vee}, \Sigma^{\flat})\simeq \cD_{\aff}$. 

Define the free abelian group $\Lambda_M^{1,\vee}\colon=\{x \in X_*(\Sbf)\otimes_\Z \R \colon \langle x , \nu_M (\Lambda_M^1) \rangle_{\R} \subset \Z \}$. 
We obtain in this way a perfect pairing $\langle \cdot,\cdot \rangle \colon \Lambda_M^1 \times \Lambda_M^{1,\vee}\to \Z$ (\S \ref{emptyaprtment}).  
Therefore, we have in our hands another root datum
$$\cD:=({\Lambda}_M^1, \Sigma^{1,\vee},{\Lambda}_M^{1,\vee}, \Sigma^1,\Delta^{1,\vee}).$$
In \cite[\S 3.2]{Lu89}, Lusztig attaches to any quadruplet $(\cD',\mathbf{W}, v=?,L)$ (where, $\cD'$ is a root datum, $\mathbf{W}$ the extended affine group associated with $\cD'$, $v$ an indeterminate \S 3.2 in  {\em loc. cit.} and $L\colon \mathbf{W} \to \N$ a function \S 3.1 in  {\em loc. cit.} verifying an additivity condition) a Hecke algebra $\cH(\cD',\mathbf{W}, v=?,L)$ generated by elements $\{T_w\colon w \in \mathbf{W}\}$, denoted by $H$ in  {\em loc. cit.}.
\begin{corollary}\label{2rootsystems}
The map $i_w \mapsto T_w$ yields two isomorphisms of Hecke algebras
$$\cH_{{I}^1}(\Z[q^{-1}]) \iso \cH(\cD,\widetilde{W}_{\aff},v=q, L) \text{ and }\cH_{{I}}^{\aff}(\Z[q^{-1}]) \iso \cH(\cD_{\aff},{W}_{\aff},v=q, L)$$
where, $L\colon W \to \mathbb{N}, w \mapsto -\frac{1}{2} \omega(q_w)$.
\end{corollary}
\begin{proof}
The corollary follows readily from the Iwahori--Matsumoto presentation in Theorem \ref{IMpresentation} and Remark \ref{9111}, because $\cH(\cD',\mathbf{W}, v=?,L)$ is characterized precisely by the same Braid relation \cite[\S 2.1]{Lu89} and quadratic relations \cite[\S 3.2]{Lu89}.
\end{proof}

\subsection{Properties of dominance \texorpdfstring{$\Lambda_M^{\flat}$}{LM}}
\begin{lemma}\label{propertiesdominance}
We have the following properties regarding $\Lambda_M^{\flat}$:
\begin{enumerate}[1 -, nosep] 
\item\label{sumtwodom} If $m_1,m_2, \dots, m_k \in \Lambda_M^{\flat}$, then there exists $m_{\circ}\in \Lambda_M^{-\flat}$ such that $m_{\circ} + m_{{I^{\flat}}} \in \Lambda_M^{-,\flat}$\footnote{We will denote the operation on $\Lambda_M$ additively.} for all $1\le i\le k$. 
\item If $m_1,m_2\in \Lambda_M^{-,\flat}$, then $\ell(m_1 + m_2) = \ell(m_1)+\ell(m_2)$.

For any $w\in  {W}$, we have:
\item $\ell(mw) = \ell(m)+\ell(w)$ if $m\in \Lambda_M^{-,\flat}$ and $\ell(wm) = \ell(w)+\ell(m)$ if $m\in \Lambda_M^{+,\flat}$.
\item $\ell(w(m)) =\ell(wmw^{-1})= \ell(m)$ for all $m\in \Lambda_M^{\flat}$.
\end{enumerate}
\end{lemma}
\begin{proof}
These are well-known facts for extended affine Weyl groups, so in particular for $\cH_{{I}^1}(\Z[q^{-1}])$. 
Accordingly, the lemma follows immediately because $\square^{\beta}\colon \Lambda_M^{\flat}\to \Lambda_M^1$ is $W$-equivariant and preserves dominance and the length $\ell$. See also \cite[Lemma 5.2.1]{R15}.
{
However, for the reader's convenience, let us show them using the ingredients we prepared so far:

1. Once we show it for $k=2$ (which is obvious), take $m_\circ= \sum m_{\circ,i}$ such that $m_{\circ,i}+ m_{i} \in \Lambda_M^-$ and $m_{\circ,i} \in\Lambda_M^-$.

2. If $m,m'$ are both antidominant then by Lemma \ref{normalizingI} $I^{\flat}mm'I^{\flat}=I^{\flat}mI^{\flat}m'I^{\flat}$, hence $\ell(mm')=\ell(m)+ \ell(m')$ by Lemma \ref{equiviIellq}.

3. If $w\in W$, using this time both Lemma \ref{normalizingI} and Lemma \ref{actionwonI}, we get $I^{\flat}mwI^{\flat}=I^{\flat}mI^{\flat}wI^{\flat}$, hence $\ell(mw)=\ell(m)+ \ell(w)$ by Lemma \ref{equiviIellq}.

4. Thanks to Remark \ref{commuwaffs}, this follows from the extended affine case.
 }\end{proof}
 
\subsection{The Weyl dot-action}\label{weyldotsect}
Set $\mathcal{R}_{\aff}^{\flat}:=\Z[\Lambda_{\aff}^{\flat}] \subset \mathcal{R}^{\flat}:=\Z[\Lambda_M^{\flat}] $.

The map $mM^{\flat} \mapsto {\mathbf{1}}_{mM^{\flat}}$ yields a natural identification between $\mathcal{R}^{\flat}$ (resp. $\mathcal{R}_{\aff}^{\flat}$) and the Hecke algebra $\cH(M\sslash M^{\flat},\Z)$ (resp. $\cH(M^{\aff,\flat}\sslash M^{\flat},\Z)$). 
The action of the Weyl group on $\Lambda_M^{\flat}$ extends linearly to an action on $\mathcal{R}^{\flat}$. 
For integrality reasons that will become clearer in the coming sections, we are interested in a variant of this action:
\begin{definition}[Dot-action]\label{dotaction}We define a twisted action of $W$ on $\cH(M\sslash M^{\flat},\Z[q^{-1}])$, by
$$\dot{w} (r) :=m \mapsto \delta_B(m)^{1/2} \delta_B(w^{-1}(m))^{-1/2} r(w^{-1}(m)),$$
for any $w\in W$ and any $r \in  \cH(M\sslash M^{\flat},\Z[q^{-1}])$. 
In particular, for any $m\in M$,
$$\dot{w} ({\bf 1}_{mM^{\flat}})= \delta_B(w(m))^{1/2} \delta_B(m)^{-1/2} {\bf 1}_{w(m)M^{\flat}}.$$
Thus, upon identifying $\cH(M\sslash M^{\flat},\Z[q^{-1}])$ with ${\Z[q^{-1}] \otimes_\Z \mathcal{R}^{\flat}}$, the above dot-action is given on basis elements by 
$\dot{w} (mM^{\flat}) = \left(\frac{\delta_B(w(m))}{\delta_B(m)}\right)^{1/2} w(m)M^{\flat}\in \Lambda_M^{\flat}$, 
and extended $\Z[q^{-1}]$-linearly to $\Z[q^{-1}] \otimes_\Z \mathcal{R}^{\flat}$.
\end{definition}
\begin{remark}\label{flatiswequi}
(i) From now on, we write $\dot{W}$ to specify that we are dealing with the dot-action of $W$ and not the standard one. 
(ii) The natural projection map ${\square}^{\flat}\colon \Lambda_M \to \Lambda_M^{\flat}$ is $\dot{W}$-equivariant, i.e. 
$(\dot{w}(m))^{\flat}= \dot{w}(m^{\flat})$ for any $w\in W$ and any $m\in \Lambda_M$.
\end{remark}
Define $c(m,n):= \delta_B(nmn^{-1})^{1/2} \delta_B(m)^{-1/2}$\nomenclature[C]{$c(m,n)$}{$\delta_B(nmn^{-1})^{1/2} \delta_B(m)^{-1/2}$ for $m\in M$ and $n\in N$.} for any $m\in M$ and any $n\in N$. Note that since the $\delta_B$ is trivial on the compact $M^1$ it factors through $\Lambda_M^1$, and so $c(m,n)$ factors then through the image of $(m,n)$ in $\Lambda_M^1 \times W$. The following lemma will be used in the sequel.
\begin{lemma}\label{crintegral}
For any $m\in M$ and $n \in N$ with image $w$ in $W$, we have
\begin{enumerate}[topsep=0pt]
\item[(i)] $c(m,n)=\prod_{\alpha \in \Phi_{\red}^- \cap w^{-1}(\Phi_{\red}^+)} \delta_\alpha(m) \in q^\Z$.
\item[(ii)] If $m \in M^-$ then $c(m,n)$ is positive power of $q$.
\end{enumerate}
\end{lemma}
We could have extracted (ii) from a similar result proved in \cite[Lemma 3.2.8 \& Proposition 3.2.4]{Mac71} for simply connected groups.
\begin{proof}
(i) Recall that the modulus character $\delta_B\colon B=MU^+ \to q^\Z, mu \mapsto |\det \text{Ad}_{\text{Lie}(U^+)}(m)|$ (it is trivial on $U^+$)\footnote{Being an absolute value of a determinant on an $F$-vector space, the modulus function has clearly its values in $q^{\Z}$.}. 
Given that  
$\text{Lie}(\U^+)= \oplus_{\alpha \in \Phi_{\red}^+} \text{Lie}(\U_\alpha)$, we have 
$$\delta_B(m)=\prod_{\alpha   \in \Phi_{\red}^+} \delta_{\alpha}(m), \text{ where }\delta_{\alpha}(m) := |\det \text{Ad}_{\text{Lie}(\U_\alpha)}(m)|.$$
Since $n \U_\alpha n^{-1}= \U_{w(\alpha)}$, we have $n{\text{Lie}(\U_\alpha)}n^{-1}= {\text{Lie}(\U_{w(\alpha)})}$ and so
\begin{align*}\delta_{\alpha}(w(m))=|\det \text{Ad}_{\text{Lie}(\U_\alpha)}(w(m))| 
=|\det \text{Ad}_{\text{Lie}(\U_{w^{-1}(\alpha)})}(m)|
= \delta_{w^{-1}(\alpha)}(m).
\end{align*}
Accordingly,
\begin{align*}
c(m,n)=
\frac{\prod_{\alpha \in \Phi_{\red}^+}  \delta_{w^{-1}(\alpha)}(m)^{1/2}}{\prod_{\alpha  \in \Phi_{\red}^+}  \delta_\alpha(m)^{1/2} }
&=\frac{\prod_{\alpha   \in w^{-1}(\Phi_{\red}^+) \cap \Phi_{\red}^- }  \delta_{\alpha}(m)^{1/2}}{\prod_{\alpha   \in \Phi_{\red}^+ \cap w^{-1}(\Phi_{\red}^-)}  \delta_{\alpha}(m)^{1/2} }\\
&={\prod_{\alpha   \in \Phi_{\red}^- \cap w^{-1}(\Phi_{\red}^+)}  \delta_{\alpha}(m)} \quad\quad\quad \quad(\delta_{-\alpha}=\delta_{\alpha}^{-1}).
\end{align*}

(ii) For any $m\in M^-$ and any $\alpha \in \Phi_{\red}^-$, one has $\omega(\delta_{\alpha}(m))\le 0$. Hence, $c(m,n)\in q^{\N}$ for any $n \in N$.
\end{proof}

\section{Twisted integral Bernstein morphism}\label{SBM}
\subsection{Freeness of the $\mathcal{R}^{\flat}$-module $\cM_{K_\cF^{\flat}}(\Z)$} 
Let $\cF \subset \mathbb{A}(\G,\Sbf)_{\red}$ be any fixed facet. 
{Following the ideas of Bernstein \cite{Deligne84}, as exposed by the approach of \cite{HKP10}}, we define the universal unramified principal series right $\cH_{K_{\cF}^{\flat}}(\Z)$-module $\cM_{K_{\cF}^{\flat}}(\Z) = \cC_c (M^{\flat}U^+\backslash G/K_{\cF},\Z)$, this is the set of $\Z$-valued functions supported on finitely many double cosets. 
The $\cH_{K_{\cF}^{\flat}}(\Z)$-module structure of $\cM_{K_{\cF}^{\flat}}(\Z)$ comes from the natural right action by convolution with respect to the normalized measure $\mu_{K_\cF^{\flat}}$ giving $K_\cF^{\flat}$ volume $1$. 

The natural identification between $\mathcal{R}^{\flat}\simeq \cH(M\sslash M^{\flat},\Z)$
allows us to endow the $\cH_{K_{\cF}^{\flat}}(\Z)$-module $\cM_{K_{\cF}^{\flat}}(\Z)$ with a left $\mathcal{R}^{\flat}$-action as follows: define for any $\psi\in \cM_{K_{\cF}^{\flat}}(\Z)$ and $r\in \mathcal{R}^{\flat}$:
$$r\cdot \psi(g):=\int_Mr(a)\psi(a^{-1}g)d\mu_{M^{\flat}}(a) \quad (\forall g\in G)$$
here, $d\mu_{M^{\flat}}(a)$ is the Haar measure on $M$ giving $M^{\flat}$ volume $1$. 
We will see in Lemma \ref{actionlambdaMI} a more concrete description of this action, it will confirm that the resulting $r\cdot \psi$ is indeed $\Z$-valued. 
It is clear that the actions of $\cH_{K_{\cF}^{\flat}}(\Z)$ and $\mathcal{R}$ on $\cM_{K_{\cF}^{\flat}}(\Z)$ commute: $\cM_{K_{\cF}^{\flat}}(\Z)$ is an $(\mathcal{R}^{\flat},\cH_{{I^{\flat}}}{K_{\cF}^{\flat}}(\Z))$-bimodule.
\begin{remark}
Here, we have defined an twisted action without using the modulus character. 
The standard action is defined as follows 
\begin{align*}
r\cdot_{\text{twist}} \psi(g)&:=\int_M \delta_B^{1/2}(a)r(a)\psi(a^{-1}g)d\mu_{M_1}(a).\qedhere
\end{align*}
\end{remark}
In Theorem \ref{structureMI}, we will generalize \cite[Lemma 1.6.1]{HKP10}. For this purpose, we begin by the following two lemmas. For any $n\in N$ set $v_{n,\cF}^{\flat}:={\bf 1}_{U^+ n K_\cF^{\flat}}$.
\begin{lemma}\label{basisMI} 
The family $$\{v_{w,\cF}^{\flat}\colon w \in \widetilde{W}^{\flat}/W_\cF^{\flat}\},$$ forms a $\Z$-basis for the $\Z$-module $\cM_{K_\cF^{\flat}}(\Z)$.
\end{lemma}
\begin{proof}
This is an immediate consequence of $U^\pm \backslash G/K_{\cF}^{\flat} \cong \widetilde{W}^{\flat}/W_\cF^{\flat}$ (Proposition \ref{bruhatdecomposition}).
\end{proof}
\begin{lemma}\label{actionlambdaMI}
For any $w \in N$ and $m \in \Lambda_M^{\flat}$ we have
$$m\cdot v_{n,\cF}^{\flat} = v_{mn,\cF}^{\flat}.$$
So the action of $\mathcal{R}^{\flat}$ on $\cM_{K_\cF^{\flat}}(\Z)$ is induced by the left action of $\Lambda_M^{\flat}$ on $M^{\flat} U^+ \backslash G/{K_\cF^{\flat}}$.
\end{lemma}
\begin{proof}
Let $n \in N$ and $r =  mM^{\flat} \in \mathcal{R}$ for some $m\in M$. Recall that $\mathcal{R}^{\flat}$ is identified with the Iwahori--Hecke algebra for $M$; $r=mM^{\flat}  \leftrightarrow r={\bf 1}_{m M^{\flat}}$. 
The integral defining $r\cdot v_{n,\cF}^{\flat}(b)$ is (by definition) non-zero only if $b\in mM^{\flat}U^+n{K_\cF^{\flat}}=M^{\flat}U^+mn{K_\cF^{\flat}}$, hence $r\cdot  v_{n,\cF}^{\flat}= \text{s} \, v_{mn,\cF}^{\flat}$ where $s$ is the scalar $r\cdot  v_{n,\cF}^{\flat}(mw) =|mM^{\flat}|=1$.
\end{proof}
The following proposition gives a precise description of the $\mathcal{R}^{\flat}$-module $\cM_{K_\cF^{\flat}}(\Z)$. 
\begin{proposition}\label{HIRbasis}\label{freerank1MK}
The $\mathcal{R}^{\flat}$-module $\cM_{K_\cF^{\flat}}(\Z)$ is free of rank $r_\cF:=|W/\jmath_W(W_\cF)|$, with basis $\{v_{w,\cF}^{\flat},\, w\in D_\cF^{\flat}\}$ where $D_\cF^{\flat}\subset W_{a_\circ}$ is any fixed set of representatives for $W/ \jmath_W(W_{\cF})$. In particular, $\{v_{w,\a}^{\flat},\, w\in W_{a_\circ}\}$ is a canonical basis for the $\mathcal{R}^{\flat}$-module $\cM_{I^{\flat}}(\Z)$.
\end{proposition}
\begin{proof}
Given that we have a bijection of $\Lambda_M^{\flat}$-sets $\widetilde{W}^{\flat}/W_\cF^{\flat} \cong \Lambda_M^{\flat} \times (W/\jmath_W(W_\cF))$ (Lemma \ref{decompwwf}), the proposition follows readily from Lemma \ref{basisMI} and Lemma \ref{actionlambdaMI}.
\end{proof}
\begin{remark}\label{lanskygeneral}
Proposition \ref{HIRbasis} yields an alternative proof of the main theorem of \cite{lansky02} 
giving the dimension of the space of $K_{\cF}$-fixed vectors of an unramified principal series representation of $G$.
\end{remark}
\begin{remark}\label{transition313}
By Lemma \ref{flatKcenter} and Corollary \ref{ontoaveraging}, the map $\psi^{\flat}\colon \cM_{K_\cF}(\Z)\to \cM_{K_\cF^{\flat}}(\Z)$, $v\mapsto v*_{K_\cF}\mathbf{1}_{K_\cF^{\flat}}$ is an homomorphism of bi-$(\Lambda_M-\cH_{K_\cF^{\flat}}(\Z))$-modules. 
Here, the left action $\Lambda_M$ on $\cM_{K_\cF^{\flat}}(\Z)= \cM_{K_\cF}(\Z) *_{K_\cF}\mathbf{1}_{K_\cF^{\flat}}$ is well defined and factors through $\Lambda_M^{\flat}$. 
\end{remark}
\subsection{Properties of the $\cH_{I^{\flat}}(\Z)$-module $\cM_{I^{\flat}}(\Z)$}
In this subsection we will focus on the case $\cF=\a$, for this write $v_{w,\a}^{\flat}, w \in \widetilde{W}^{\flat}$ simply as $v_w^{\flat}$. 
We would like to describe the structure of $\cM_{I^{\flat}}(\Z)$, this time as a $\cH_{I^{\flat}}(\Z)$-module. However, a satisfactory result can only be obtained after enlarging the coefficients ring $\Z$ to $R:=\Z[q^{\pm 1}]$.  
\begin{proposition}\label{relationsiwahori} 
The action of $\cH_{I^{\flat}}(\Z)$ on $\cM_{I^{\flat}}(\Z)$ has the following rules: for any $w \in W_{a_\circ}^{\flat}$ and $m \in \Lambda_M^{\flat}$, we have
\begin{enumerate}[1 -, nosep]
\item\label{1actionMI} $v_1^{\flat} *_{I^{\flat}}i_w^{\flat}=v_w^{\flat}$,
\item\label{2actionMI} $v_m^{\flat} *_{I^{\flat}}i_w^{\flat}=v_{m w}^{\flat}$,
\item\label{3actionMI} $v_1^{\flat} *_{I^{\flat}}i_m^{\flat}=v_m^{\flat}$ if $m\in \Lambda_M^{-,\flat}$.
\item\label{4actionMI} $v_1^{\flat} *_{I^{\flat}}(i_m^{\flat})^{-1}=v_{-m}^{\flat}\in \cM_{I^{\flat}}(\Z[q^{-1}])$ if $m\in \Lambda_M^{-,\flat}$.
\end{enumerate}
\end{proposition}
\begin{proof}
1 - \& 3 - Let $n\in N$. If the quantity $v_1^{\flat}*_{I^{\flat}} i_n^{\flat}(b)= \int_{U^+{I}^{\flat}}i_n^{\flat}(a^{-1}b)d\mu_{I^{\flat}}(a)$ is non-zero for some $b$, then 
$b\in U^+I^{\flat}nI^{\flat}=U^+I^{-,\flat}nI^{\flat}$ (the equality follows using the Iwahori factorization of $I^{\flat}$ and $M^{\flat} \triangleleft N$). 

If $n\in N_{1,a_\circ}\cup M^-$, then we must have $b\in U^+nI^{\flat}$ by Corollary \ref{WInInI}. 
Accordingly, $v_1^{\flat} *_{I^{\flat}} i_w^{\flat}= \text{s} \, v_w^{\flat}$, with $s=v_1^{\flat}*_{I ^{\flat}} i_n^{\flat}(n) =U^+I^{\flat} \cap n I^{\flat} n^{-1}I^{\flat} \in \Z$, which is $1$ by Corollary \ref{WInInI}. 

2 - Using the previous rule and Lemma \ref{actionlambdaMI}: $v_m^{\flat}*_{I^{\flat}}i_w^{\flat}=m\cdot v_1^{\flat}*_{I^{\flat}} i_w^{\flat}
=m\cdot v_w^{\flat} =v_{mw}^{\flat}$. 

4 - Using fact 3, we have $ v_{-m}^{\flat}={-m} \cdot v_1^{\flat} ={-m} \cdot v_m^{\flat}*_{I^{\flat}}(i_m^{\flat})^{-1}=v_1^{\flat}*_{I^{\flat}}(i_m^{\flat})^{-1}$. 
\end{proof}
\begin{corollary}\label{surjectiveh}
For any $m\in  \Lambda_M^{\flat}$, any $w\in W$ and any $m_\circ\in \Lambda_M^{-,\flat}$ such that $m+m_\circ \in \Lambda_M^{-,\flat}$ (which always exists by Proposition \ref{propertiesdominance}), we have in $\cM_{I^{\flat}}(\Z[q^{-1}])$
$$v_{mw}^{\flat}= v_1^{\flat}*_{I^{\flat}}i_{m+m_\circ}^{\flat}*_{I^{\flat}} (i_{m_\circ}^{\flat})^{-1} *_{I^{\flat}} i_w^{\flat}= v_1^{\flat}*_{I^{\flat}} (i_{m_\circ}^{\flat})^{-1}*_{I^{\flat}}i_{m+m_\circ}^{\flat} *_{I^{\flat}} i_w^{\flat}.$$
\end{corollary}
\begin{proof}
This follows readily from Proposition \ref{relationsiwahori}.
\end{proof}
Accordingly, any $v \in \cM_{I^{\flat}}(\Z[q^{-1}])$ can be written as $v_1^{\flat}*_{I^{\flat}} h_v$ for some $h_v \in \cH_{I^{\flat}}(\Z[q^{-1}])$, i.e. 
the following homomorphism of right $\cH_{I^{\flat}}(\Z[q^{-1}])$-modules
\begin{align*}\hbar_{I^{\flat}}\colon \cH_{I^{\flat}}({\Z[q^{-1}]}) \longrightarrow \cM_{I^{\flat}}({\Z[q^{-1}]}), \quad h\longmapsto  v_1^{\flat}*_{I^{\flat}}h.\end{align*}
is surjective. 
Actually, more is true:
\begin{theorem}\label{structureMI}
The homomorphism $\hbar_{I^{\flat}}$ is an isomorphism. 
\end{theorem}
\begin{proof} 
Given Corollary \ref{surjectiveh}, we just need to show the injectivity of $\hbar_{{I^{\flat}}}$. 
In doing so, we will actually show that this map is "upper triangular with respect to the Chevalley--Bruhat order, with invertible diagonals".

Recall that $\cH_{I^{\flat}}(\Z)$ (resp. $\cM_{I^{\flat}}(\Z)$) admits the $\Z$-basis $i_y^{\flat}$ (resp. $v_x^{\flat}$) for $x,y \in \widetilde{W}^{\flat}$.  
For any $y\in \widetilde{W}^{\flat}$, write $v_1^{\flat} *_{I^{\flat}} i_y^{\flat}= \sum_{x \in \widetilde{W}^{\flat}} c_{x,y} v_x^{\flat}$ with $c_{x,y} = v_1^{\flat} *_{I^{\flat}} i_y^{\flat} (x) \in \Z$. 
The function $v_1^{\flat} *_{I^{\flat}} i_y^{\flat} $ can be nonzero only on the set $U^+I^{\flat}yI^{\flat}$, so by Corollary \ref{uppertriangulariy} one gets
$$c_{x,y} \neq 0 \Leftrightarrow U^+x \cap I^{\flat}yI^{\flat}\neq \emptyset \Rightarrow x \le y.$$ 
Therefore,
$$v_1^{\flat} *_{I^{\flat}} i_y^{\flat}=\sum_{x \in \widetilde{W}^{\flat} \colon x\le y} c_{x,y} v_x^{\flat}.$$
Now, write any $h\in \ker \hbar_{I^{\flat}}$ as $\sum_{y\in \widetilde{W}^{\flat}} r_y i_y^{\flat}$, where the coefficients $\{r_y\}$ are nonzero precisely on some finite set $I_h \subset \widetilde{W}^{\flat}$. 
Assume that $h\neq 0$ and let $z\in I_h$ be any element maximal with respect to the Chevalley--Bruhat order in $I_h$. 
The coefficient $c_{z,z}=|U^+I^{\flat} \cap z I^{\flat}z^{-1}I^{\flat}|_{I^{\flat}}\ge 1$ and it is actually a power of $q$ by Lemma \ref{powerq}. 
By maximality of $z$ in $I_h$, the term $v_z^{\flat}$ can not be killed in the sum $v_1^{\flat} *_{I^{\flat}} h=\sum_{y\in I_h} \sum_{x \in \widetilde{W}^{\flat} \colon x\le y} c_{x,y} r_y v_x^{\flat}$, hence $v_1^{\flat} *_{I^{\flat}} h$ can not equal zero. 
Therefore,  $\ker \hbar_{I^{\flat}}$ is trivial and consequently $ \hbar_{I^{\flat}}$ is an isomorphism.
\end{proof}
Theorem \ref{structureMI}, implies immediately the following 
\begin{corollary}\label{cor6}
The right $\cH_{I^{\flat}}({\Z[q^{-1}]})$-module $\cM_{I^{\flat}}({\Z[q^{-1}]})$ is free and of rank $1$, with canonical generator $v_1^{\flat}$. 
This yields a canonical isomorphism
$$\begin{tikzcd}
\cH_{I^{\flat}}({\Z[q^{-1}]}) \arrow{r}{\simeq}& \text{End}_{\cH_{I^{\flat}}({\Z[q^{-1}]})}(\cM_{I^{\flat}}({\Z[q^{-1}]})), \quad h \arrow[mapsto]{r}& (v_1^{\flat}*_{I^{\flat}}h'\mapsto v_1*_{{I^{\flat}}}h*_{{I^{\flat}}}h').
\end{tikzcd}$$
\end{corollary}
\begin{corollary}\label{injectivebarh}
Let $e\in \cH_{I^{\flat}}(\Z[q^{-1}])$ be an idempotent and set 
$$v_e:=v_1^{\flat}*_{I^{\flat}} e, \cH_e({\Z[q^{-1}]}):=e*_{I^{\flat}}\cH_{I^{\flat}}({\Z[q^{-1}]}) *_{I^{\flat}}e \text{ and }\cM_e({\Z[q^{-1}]}):=\cM_{I^{\flat}}({\Z[q^{-1}]})*_{I^{\flat}}e$$
The following homomorphism of right $\cH_e({\Z[q^{-1}]})$-modules is injective
\begin{align*}\hbar_e\colon \cH_e({\Z[q^{-1}]})\longrightarrow \cM_e({\Z[q^{-1}]}), \quad h\longmapsto  v_e *_{I^{\flat}}h.\end{align*} 
\end{corollary}
\begin{proof}
It is clear that $\hbar_e$ is a well defined homomorphism of right $\cH_e({\Z[q^{-1}]})$-modules. 
Let $h\in \ker \hbar_e$, then $0=v_e*_{I^{\flat}}h= v_1^{\flat}*_{I^{\flat}} h$. Accordingly, $\ker \hbar_e \subset \ker \hbar_{I^{\flat}}= \{0\}$ by Theorem \ref{structureMI}.
\end{proof}
\begin{remark}\label{idempI2K}
Let $R$ be a ring in which ${|{K^{\flat}}:I^{\flat}|}$ is invertible.
We can naturally identify $\cH_{K^{\flat}}(R)$ with the two-sided ideal $e_{K^{\flat}}*_{{I^{\flat}}}\cH_{{I^{\flat}}}(R)*_{{I^{\flat}}}e_{K^{\flat}} \subset \cH_{{I^{\flat}}}(R)$, 
where  $e_{K^{\flat}}$ is the idempotent ${|{K^{\flat}}:I^{\flat}|}^{-1}{\bf 1}_{K^{\flat}}$ (see Lemma \ref{identificationsubsuphecke}). In a similar way, $\cM_{K^{\flat}}(R)$ can also be naturally identified with $\cM_{{I^{\flat}}}(R)*_{{I^{\flat}}}e_{K^{\flat}} \subset \cM_{{I^{\flat}}}(R)$ and the right action of $\cH_{K^{\flat}}(R)$ will correspond to the right action of $e_{K^{\flat}}*_{{I^{\flat}}}\cH_{{I^{\flat}}}(R)*_{{I^{\flat}}}e_{K^{\flat}}$.\end{remark} 
\subsection{Decomposition of \texorpdfstring{$\cH_{{I^{\flat}}}({\Z[q^{-1}]})$}{HI[Zq-1]} and the twisted Bernstein map}\label{decomptwistbern}
We begin this section by considering two sub-algebras of the Iwahori--Hecke algebra $ \cH_{{I^{\flat}}}(\Z)$:
\begin{itemize}
\item The finite Hecke algebra: $$\cH_{{I^{\flat}}}(\Z)^0:=\cC_c( K^{\flat} \sslash I^{\flat},\Z).$$ 
A $\Z$-basis for $\cH_{{I^{\flat}}}(\Z)^0$ is given by $\{i_w^{\flat}\}_{w\in W_{a_\circ}^{\flat}}$ (Proposition \ref{Iwahori1}).
\item The $(\mathcal{R}^{\flat},\cH_{{I^{\flat}}}(\Z))$-bimodule structure on $\M_{{I^{\flat}}}(\Z)$ induces a homomorphism of algebras
$$\mathcal{R}^{\flat} \hra  \text{End}_{\cH_{{I^{\flat}}}(\Z)}(\cM_{{I^{\flat}}}(\Z)),$$
which is actually an embedding since $\cM_{{I^{\flat}}}(\Z)$ is free over $\mathcal{R}^{\flat}$. 
Accordingly, thanks to Theorem \ref{structureMI}, one gets another embedding of algebras:
\begin{align*}\dot{\Theta}_{\text{\text{Bern}}}^{\flat}\colon \mathcal{R}^{\flat}\hookrightarrow \cH_{{I^{\flat}}}({\Z[q^{-1}]}),\quad m \mapsto \dot{\Theta}_m^{\flat}, \end{align*}
characterized by the property: $m\cdot v_1^{\flat}= v_1^{\flat} *_{{I^{\flat}}} \dot{\Theta}_m^{\flat}$, for any $m\in \mathcal{R}^{\flat}$. 
\end{itemize}
\begin{lemma}\label{explicittheta}
For any $m\in  \Lambda_M^{\flat}$ and any $m_\circ\in \Lambda_M^{-,\flat}$ such that $m+m_\circ \in \Lambda_M^{-,\flat}$,  
we have
$$\dot{\Theta}_m^{\flat}= i_{m+m_\circ}^{\flat}*_{{I^{\flat}}} (i_{m_\circ}^{\flat})^{-1}= (i_{m_\circ}^{\flat})^{-1}  *_{{I^{\flat}}}i_{m+m_\circ}^{\flat}.$$
In particular, $\dot{\Theta}_m^{\flat}=i_m^{\flat}$ if $m\in \Lambda_M^{-,\flat}$ and $\dot{\Theta}_m^{\flat}=(i_{-m}^{\flat})^{-1}$ if $m\in \Lambda_M^{+,\flat}$.
\end{lemma}
\begin{proof} 
These equalities follow readily from Corollary \ref{surjectiveh}.
\end{proof}
\begin{remark}\label{Tran331}
(i) By Remark \ref{transition313}, since $\mathbf{1}_{I^{\flat}}$ is central in $\cH_{I^{\flat}}(\Z[q^{-1}])$, the following diagram is commutative
$$\begin{tikzcd}[column sep=large]
 \mathcal{R}\arrow[hook]{r}{\dot{\Theta}_{\text{\text{Bern}}}} \arrow[swap,two heads]{d}{\square^{\flat}}& \cH_{I}({\Z[q^{-1}]}) \arrow[two heads]{d}{\iota^{\flat}} \\
\mathcal{R}^{\flat}\arrow[hook]{r}{\dot{\Theta}_{\text{\text{Bern}}}^{\flat}}& \cH_{{I^{\flat}}}({\Z[q^{-1}]})
\end{tikzcd}$$
In other words, $\iota^{\flat}(\dot\Theta_m)= \dot\Theta_{{ m}^{\flat}}^{\flat}, m\in \mathcal{R} $.
 
(ii) The Iwahori--Hecke operators $\dot{\Theta}_m^{\flat}$ generalizes the element denoted by $\widetilde{T}$ in Lusztig's \cite[\S 7]{Lu83} (up to the factor $\delta_B(m)^{1/2}$), 
which does not appear here. 
The untwisted version of the map $\dot{\Theta}_{\text{Bern}}$ was attributed by Lusztig in {\em loc. cit.} (for split reductive groups) to Bernstein, this is the reason we use the subscript $\text{Bern}$.
\end{remark}
\begin{corollary}\label{cortheostructureMI}
For any $w \in \widetilde{W}^{\flat}$, we have
$$i_w^{\flat}=\sum_{x=m_x w_x \le w} \dot\Theta_{m_x}^{\flat}*_{I^{\flat}} i_{w_x}^{\flat}.$$
\end{corollary}
\begin{proof}
The lemma follows readily from the proof of Theorem \ref{structureMI} and Corollary \ref{surjectiveh}.
\end{proof}

\begin{lemma}\label{veetheta}
For any $m\in \Lambda_M^{\flat}$, on has
$$(\dot{\Theta}_m^{\flat})^\vee= (i_{w_\circ}^{\flat})^{-1} *_{I^{\flat}}\dot{\Theta}_{-w_\circ(m)}^{\flat}*_{I^{\flat}} i_{w_\circ}^{\flat}$$
where $w_\circ \in W$ is the longest element.
\end{lemma}
\begin{proof}
It suffices to show it for elements in $ \Lambda_M^{-,\flat}$, so let $m\in \Lambda_M^{-,\flat}$ we also have $-w_\circ(m)\in \Lambda_M^{-,\flat}$ and accordingly by Lemma \ref{propertiesdominance} $\ell( w_\circ  (-m) \cdot w_\circ )= \ell( w_\circ m^{-1})=\ell(w_\circ) + \ell(-m) =\ell(w_\circ) + \ell(w_\circ(-m)) $ 
hence
\begin{align*} i_{-w_\circ(m)}^{\flat}*_{I^{\flat}} i_{w_\circ}^{\flat}&=i_{ w_\circ  (-m) \cdot w_\circ}^{\flat} = i_{w_\circ m^{-1}}^{\flat} = i_{w_\circ}^{\flat}*_{I^{\flat}}i_{-m}^{\flat}.\qedhere\end{align*}
\end{proof}
In our setting, it is easy to obtain a statement similar to \cite[Lemma 1.7.1]{HKP10}:
\begin{lemma}\label{RH0isomorphism}
The homomorphism of $\mathcal{R}^{\flat}\otimes_\Z{\Z[q^{-1}]}$-modules 
$$
(\mathcal{R}^{\flat}\otimes_\Z {\Z[q^{-1}]})\otimes_\Z \cH_{{I^{\flat}}}^0(\Z)\longrightarrow \cH_{{I^{\flat}}}({\Z[q^{-1}]}), \quad m\otimes h^0 \longmapsto \dot{\Theta}_m^{\flat} *_{{I^{\flat}}}h^0,
$$
is an isomorphism. Composing this homomorphism with $h\mapsto v_1^{\flat}*_{{I^{\flat}}} h$ yields the isomorphism of $\mathcal{R}^{\flat}\otimes_\Z{\Z[q^{-1}]}$-modules $\mathcal{R}^{\flat}\otimes_\Z \cH_{{I^{\flat}}}^0({\Z[q^{-1}]}) \longrightarrow \cM_{{I^{\flat}}}({\Z[q^{-1}]})$, 
given by $m\otimes i_w^{\flat} \longmapsto v_{m w}^{\flat}$, for $w\in W_{a_\circ}^{\flat}$ and $m\in \Lambda_M^{\flat}$.
\end{lemma}
\begin{proof}
Given Theorem \ref{structureMI}, this lemma follows immediately from Corollary \ref{surjectiveh} which shows that the composition 
$$\mathcal{R}^{\flat}\otimes_\Z \cH_{{I^{\flat}}}^0({\Z[q^{-1}]}) \longrightarrow \cH_{{I^{\flat}}}({\Z[q^{-1}]}) \longrightarrow \cM_{{I^{\flat}}}({\Z[q^{-1}]}),$$ 
yields an isomorphism of $\mathcal{R}^{\flat}\otimes_\Z{\Z[q^{-1}]}$-modules.  
\end{proof}
A direct application of Lemma \ref{RH0isomorphism} gives an explicit $\mathcal{R}^{\flat}\otimes_\Z{\Z[q^{-1}]}$-basis for the Iwahori--Hecke algebra:
\begin{corollary}\label{HIbasis} 
The algebra $\cH_{{I^{\flat}}}({\Z[q^{-1}]})$ is a free left (and right) $\mathcal{R}^{\flat}\otimes_\Z{\Z[q^{-1}]}$-module, with canonical basis $\{i_w^{\flat} \colon w \in W_{a_\circ}^{\flat}\}$.
The sets $\{\dot{\Theta}_m^{\flat} *_{{I^{\flat}}} i_w^{\flat}\colon m \in \Lambda_M^{\flat}, w \in W_{a_\circ}^{\flat}\}$ 
and $\{ i_w^{\flat} *_{{I^{\flat}}}\dot{\Theta}_m^{\flat} \colon m \in \Lambda_M^{\flat}, w \in W_{a_\circ}^{\flat}\}$ are both ${\Z[q^{-1}]}$-basis for $\cH_{{I^{\flat}}}({\Z[q^{-1}]})$.
\end{corollary}
\begin{proof}
The fact that $\cH_{{I^{\flat}}}({\Z[q^{-1}]})$ is a free left $\mathcal{R}^{\flat}\otimes_\Z{\Z[q^{-1}]}$-module, with canonical basis $\{i_w^{\flat} \colon w \in W_{a_\circ}^{\flat}\}$ and the set $\{\dot{\Theta}_m^{\flat} *_{{I^{\flat}}} i_w^{\flat}\colon m \in \Lambda_M^{\flat}, w \in W_{a_\circ}^{\flat}\}$ is ${\Z[q^{-1}]}$-basis for $\cH_{{I^{\flat}}}({\Z[q^{-1}]})$ is clear from Corollary \ref{RH0isomorphism}. Let us show the remaining statements. We may argue in two different ways: 

(i) Consider $\mathcal{M}^-_{{I^{\flat}}}(\Z)=\Z[ I^{\flat} \backslash G/U^-]$ instead, this is a $(\cH_{I^{\flat}}(\Z),\mathcal{R}^{\flat})$-bi-module. 
Rewriting the whole story of \S \ref{SBM} in this context shows that $\cH_{{I^{\flat}}}({\Z[q^{-1}]})$ is a free right $\mathcal{R}^{\flat}\otimes_\Z{\Z[q^{-1}]}$-module and so $\{ i_w^{\flat} *_{{I^{\flat}}}\dot{\Theta}_m^{\flat} \colon m \in \Lambda_M^{\flat}, w \in W_{a_\circ}^{\flat}\}$ is ${\Z[q^{-1}]}$-basis.

(ii) The map $\cH_{{I^{\flat}}}^0(\Z) \to \cH_{{I^{\flat}}}^0(\Z), i_w^{\flat} \to  i_{w_\circ}^{\flat} *_{{I^{\flat}}}(i_w^{\flat})^\vee$ is an isomorphism of $\Z$-modules, the map $\Lambda_M^{\flat} \to \Lambda_M^{\flat}, m \mapsto -w_\circ( m)$ is an automorphism of rings and $\cH_{I^{\flat}}(\Z[q^{-1}]) \to \cH_{I^{\flat}}(\Z[q^{-1}]), h \mapsto  ((i_{w_{\circ}}^{\flat})^{-1} *_{I^{\flat}} h)^\vee$ is an automorphism of $\Z[q^{-1}]$-modules. 
Therefore, the homomorphism of ${\Z[q^{-1}]}$-modules 
$$
(\mathcal{R}^{\flat}\otimes_\Z {\Z[q^{-1}]})\otimes_\Z \cH_{{I^{\flat}}}^0(\Z)\longrightarrow \cH_{{I^{\flat}}}({\Z[q^{-1}]}), \quad m\otimes i_w^{\flat} \longmapsto \left(( i_{w_\circ}^{\flat})^{-1} *_{{I^{\flat}}}  \dot{\Theta}_{-w_\circ( m)}^{\flat} *_{{I^{\flat}}} i_{w_\circ}^{\flat} *_{{I^{\flat}}} (i_w^{\flat})^\vee\right)^\vee,
$$
is an isomorphism. Given Lemma \ref{veetheta}, this concludes the proof of the corollary.
\end{proof}
\begin{remark}\label{decomhiaff}
If we replace $G$ by $G_1^{\flat}$ then $\Lambda_{\aff}^{\flat}$ plays exactly the role of $\Lambda_M^{\flat}$. For example, as in Corollary \ref{RH0isomorphism}, we have two decomposition: 
$$\cH_{I^{\flat}}^{\aff}(\Z[q^{-1}]) = \bigoplus_{w\in W}i_w^{\flat}*_{I^{\flat}} \Theta_{\text{Bern}}^{\flat}(\Lambda_{\aff}^{\flat}\otimes_\Z  \Z[q^{-1}])= \bigoplus_{w\in W} \Theta_{\text{Bern}}^{\flat}(\Lambda_{\aff}^{\flat}\otimes_\Z  \Z[q^{-1}]) *_{I^{\flat}} i_w^{\flat}.$$
\end{remark}

\subsection{Further properties of $\dot\Theta^{\flat}$-elements}

By Corollary \ref{2rootsystems}, various results in \cite{Lu89} applies to $\cH_{I^1}({\Z[q^{-1}]})$ (i.e. $M^{\flat}=M^1$) and $\cH_{I}^{\aff}({\Z[q^{-1}]})$. In this subsection, we will extend some results proved by Lusztig for $\cH_{I^1}(\Z[q^{-1}])$ to $\cH_{{I}}(\Z[q^{-1}])$.

From now on, whenever needed, we identify $\mathcal{R}^{\flat}$ with its image $\dot{\Theta}_{\text{\text{Bern}}}^{\flat}(\mathcal{R}^{\flat})$ and $\mathcal{R}_{\aff}^{\flat}$ with $\dot{\Theta}_{\text{\text{Bern}}}^{\flat}(\mathcal{R}_{\aff}^{\flat})$. 
Let $\mathcal{L}_{\aff}^{\flat}$ be the fraction field of $ \mathcal{R}_{\aff}^{\flat}$.
\begin{lemma}\label{341injLaff}
We have a commutative diagram
$$
\begin{tikzcd}
\cH_{I^{\flat}}^{\aff}({\Z[q^{-1}]}) \arrow[hook]{r} \arrow[hook]{d}& \cH_{I^{\flat}}({\Z[q^{-1}]}) \arrow[hook]{d}\\
\mathcal{L}_{\aff}^{\flat} \otimes_{\mathcal{R}_{\aff}^{\flat}} \cH_{I^{\flat}}^{\aff}({\Z[q^{-1}]}) \arrow[hook]{r} \arrow[swap]{d}{\simeq}& \mathcal{L}_{\aff}^{\flat} \otimes_{\mathcal{R}_{\aff}^{\flat}} \cH_{I}({\Z[q^{-1}]}) \arrow[two heads]{d}{\iota^{\flat}}\\
\mathcal{L}_{\aff}^1 \otimes_{\mathcal{R}_{\aff}^1}\cH_{I^1}^{\aff}({\Z[q^{-1}]}) \arrow[hook]{r} &\mathcal{L}_{\aff}^1 \otimes_{\mathcal{R}_{\aff}^1}\cH_{I^1}({\Z[q^{-1}]})
\end{tikzcd}
$$
\end{lemma}
\begin{proof}
By Remark \ref{9111}, we have a commutative diagram
$$
\begin{tikzcd}
\cH_{I^{\flat}}^{\aff}({\Z[q^{-1}]}) \arrow[hook]{r} \arrow[swap]{d}{\simeq}& \cH_{I^{\flat}}({\Z[q^{-1}]})\arrow[two heads]{d}{\iota^1}\\
\cH_{{I}^1}^{\aff}({\Z[q^{-1}]}) \arrow[hook]{r} &\cH_{{I}^1}({\Z[q^{-1}]})
\end{tikzcd}
$$
Therefore, It remains to show why the horizontal ones and the upper vertical ones are injective. 
For this, we only need to show that no non-zero element in $\mathcal{R}_{\aff}^{\flat}$ is a zero left/right divizor in $\cH_I({\Z[q^{-1}]})$:
Let $a \in \mathcal{R}_{\aff}^{\flat}$ and $b\in \cH_{I^{\flat}}({\Z[q^{-1}]}) $ such that $a *_{I^{\flat}}b =0$. 
By the decomposition of Remark \ref{generalizedIWdecomp}, $b= \sum_{\sigma\in \widetilde\Upomega^{\flat}}\sigma\otimes b_\sigma$ for some $b_\sigma \in \cH_{I^{\flat}}^{\aff}({\Z[q^{-1}]})$. Hence $\sigma(a) *_{I^{\flat}}b_\sigma=0$ for all $\sigma\in \widetilde\Upomega^{\flat}$. 
Thanks to the decomposition $\cH_{I^{\flat}}^{\aff}({\Z[q^{-1}]}) = \oplus_{w\in W_{a_\circ}^{\flat}} (\mathcal{R}_{\aff}^{\flat}\otimes_\Z  \Z[q^{-1}]) *_{I^{\flat}}i_w^{\flat}$ (Remark \ref{decomhiaff}) we see that $a$ must be a zero-divisor in $\mathcal{R}_{\aff}^{\flat}$ which forces it to be zero since this latter is a domain. 
If we start with a right zero, i.e. $b *_{I^{\flat}}a =0$, then using decomposition $\cH_{I^{\flat}}^{\aff}({\Z[q^{-1}]}) = \oplus_{w\in W_{a_\circ}^{\flat}} i_w^{\flat}*_{I^{\flat}} (\mathcal{R}_{\aff}^{\flat}\otimes_\Z  \Z[q^{-1}])  $ one shows similarly that $a$ must be zero.
\end{proof}
\begin{lemma}\label{conjuginaff}
For any $m\in \Lambda_M^{\flat}, w \in W_{a_\circ}^{\flat}$ and $s\in \mathcal{T}_{a_\circ}^{\flat} = \cS^{\flat} \cap W_{a_\circ}^{\flat} $, we have $\dot\Theta_{-m}^{\flat}*_{I^{\flat}} i_w^{\flat}*_{I^{\flat}} \dot\Theta_{s(m)}^{\flat} \in \cH_{I^{\flat}}^{\aff}({\Z[q^{-1}]})$.
\end{lemma}
\begin{proof}
{For any $m\in \Lambda_M^{\flat}$ and $s \in  \mathcal{T}_{a_\circ}^{\flat}$, by \ref{sumtwodom} Proposition \ref{propertiesdominance}, one can find an element $m_\circ \in \Lambda_M^{-,\flat}$ such that $m_\circ-m $ and $m_\circ + s(m) $ both lie in $\Lambda_M^{-,\flat}$. Hence, for any $w\in W_{a_\circ}^{\flat}$, we have 
$$ \dot\Theta_{-m}^{\flat}*_{I^{\flat}} i_w^{\flat} *_{I^{\flat}}\dot\Theta_{s(m)} ^{\flat}
=(i_{m_\circ}^{\flat})^{-1} *_{I^{\flat}} i_{(m_\circ - m) w}^{\flat}*_{I^{\flat}}i_{m_\circ + s(m)}^{\flat} *_{I^{\flat}} (i_{m_\circ}^{\flat})^{-1}  .$$
Let $  z_{m_\circ}$ be the projection of $m_\circ$ in $ \widetilde{\Upomega}^{\flat}$. 
By \ref{v} Corollary \ref{additionalIMpres}, we have 
$$ i_{m_\circ}^{\flat} \in \cH_{I^{\flat}}^{\aff}({\Z[q^{-1}]})^\times *_{I^{\flat}} i_{z_{m_\circ}}^{\flat}=i_{z_{m_\circ}}^{\flat}*_{I^{\flat}} \cH_{I^{\flat}}^{\aff}({\Z[q^{-1}]})^\times.$$
Accordingly, the support 
$$\text{Supp}(\dot\Theta_{-m}^{\flat}*_{I^{\flat}} i_w^{\flat} *_{I^{\flat}}\dot\Theta_{s(m)} ^{\flat}) \subset I^{\flat}z_{m_\circ}^{-1}m_\circ m^{-1} wI^{\flat}m_\circ  s(m)z_{m_\circ}^{-1}I^{\flat} \subset G_1^{\flat}.$$
Therefore, $\dot\Theta_{-m}^{\flat}*_{I^{\flat}} i_w^{\flat}*_{I^{\flat}} \dot\Theta_{s(m)} \in  \cH_{I^{\flat}}^{\aff}({\Z[q^{-1}]})$.}
\end{proof}
\begin{proposition}\label{propoTheta}
For any $s \in  \mathcal{T}_{a_\circ}^{\flat}$ and any $m\in \Lambda_M^{\flat}$, we have
$$ \dot\Theta_m^{\flat} *_{I^{\flat}}  (i_s^{\flat} +1)-(i_s^{\flat}+1)*_{I^{\flat}}  \dot\Theta_{\dot{s}( m)}^{\flat}= ( \dot\Theta_{m}^{\flat}-\dot\Theta_{\dot{s}( m)}^{\flat})\mathcal{G}(\alpha) \in \mathcal{R}^{\flat}$$
where, $\mathcal{G}(\alpha) \in \mathcal{L}_{\aff}^{\flat}$ is the element defined in \cite[\S 3.8 \& \S 3.13]{Lu89} for the algebra $ \cH(\cD_{\aff},{W}_{\aff},v=q, L)$.
\end{proposition}
\begin{proof} 
Let $m\in M$. We want to apply \cite[Proposition 3.9]{Lu89}. Since in {\em loc. cit.} the notation ${\theta}_m$ does not equal our $\dot{\Theta}_m$, we first compare the two to make things clear. Let $m_1, m_2\in M^-$ verifying $m=m_1-m_2$, hence by Corollary \ref{2rootsystems} and \cite[Definition \S 3.3 (a)]{Lu89}
$${\theta}_m = \frac{q_{m_2}^{\frac{1}{2}}}{q_{m_1}^{\frac{1}{2}}} i_{m_1}*_Ii_{m_2}^{-1}
\overset{\text{Lemma \ref{qmdeltam}}}{=}
\delta_B(m)^{\frac{1}{2}}\dot{\Theta}_m.
$$
Set $A_{s}(m):= (i_s +1)- \theta_{-m}(i_s+1) \theta_{s(m)}-(1-\theta_{s(m)-m})\mathcal{G}(\alpha)$. 
Clearly, the term $(1-\theta_{s(m)-m})\mathcal{G}(\alpha) \in \mathcal{L}_{\aff}$, but its image by $\iota^1$ lies in $\mathcal{R}_{\aff}^{1}$ \cite[\S 2.5 (b)]{Lu89}, hence $(1-\theta_{s(m)-m})\mathcal{G}(\alpha) \in \mathcal{R}_{\aff}$. 
Accordingly, by Lemma \ref{conjuginaff}, we have $A_s(m)\in \cH_I^{\aff}({\Z[q^{-1}]})$.

Thanks to \cite[Proposition 3.9]{Lu89}, we know that $\iota^{1} (A_s(m))=0$ in $\cH_{I^{1}}({\Z[q^{-1}]})$ since by Remark \ref{Tran331} we have $\iota^1(\dot\Theta_m)= \dot\Theta_{{m}^1}^1$ (recall that $m^1$ denotes the image of $m$ in $\Lambda_M^1$) . 
Therefore $A_s(m)=0$, given that $\iota^1 \colon \cH_I^{\aff}({\Z[q^{-1}]})\to \cH_{I^{1}}({\Z[q^{-1}]})$ is injective.

In conclusion, we proved that for any reflexion $s\in  \mathcal{T}_{a_\circ}$ and any $m\in M$
$$ \dot\Theta_m (i_s +1)-c(m,s) (i_s+1) \dot\Theta_{s(m)}= ( \dot\Theta_{m}-c(m,s)\dot\Theta_{s(m)})\mathcal{G}(\alpha)$$
which is proves the lemma (see Remark \ref{Tran331}).
\end{proof}

\subsection{$\dot{W}$-invariant elements are central}
For $m\in \Lambda_M^{-,\flat}$, set 
$$r_m^{\flat}:=\sum_{w\in W/W_m} \dot{w}(m)=\sum_{w\in W/W_m} c(m,w) w(m)$$ 
where, $W_m$ denotes the isotropy subgroup of $m$ in $W$ (for the standard Weyl action). 

Let ${\mathcal{R}^{\flat}}^{\dot{W}} \subset{\mathcal{R}^{\flat}} $ denotes the $\Z$-subalgebra of elements invariant under the dot-action, which again will be identified with its $\dot{\Theta}_{\text{\text{Bern}}}^{\flat}({\mathcal{R}^{\flat}}^{\dot{W}})$.
 So $r_m^{\flat}$ also denotes its image $\sum_{w\in W/W_m} \dot\Theta_{\dot{w}(m)}^{\flat}=\sum_{w\in W/W_m} c(m,w) \dot\Theta_{{w}(m)}^{\flat}$.

By Lemma \ref{crintegral}, all $r_m\in {\mathcal{R}^{\flat}} $ for all $m\in \Lambda_M^{-,\flat}$. 
In consequence, Lemma \ref{orbitLambda} implies that $\{r_m^{\flat}, m\in \Lambda_M^{-,\flat}\}$ is a $\Z$-basis for ${\mathcal{R}^{\flat}}^{\dot{W}}$.
\begin{corollary}\label{invincentIwahori}
For any $m\in \Lambda_M^{\flat}$ and any $s \in \mathcal{T}_{a_\circ}$, $\dot\Theta_m^{\flat} + \dot\Theta_{\dot{s}( m)}^{\flat}$ commutes with $i_{s}^{\flat}$. 
In consequence,
$${\mathcal{R}^{\flat}}^{\dot{W}} \subset Z(\cH_{I^{\flat}}({\Z[q^{-1}]})).$$
\end{corollary}
\begin{proof}
The first statement is an immediate consequence of Proposition \ref{propoTheta}. 
For the second statement, 
we observe that for any $s \in\mathcal{T}_{a_\circ}$ (Definition \ref{wallsgeneratingcox}) and any $m\in \Lambda_M^{\flat}$, we have
$$ |W/W^m| r_m^{\flat}= \sum_{w\in \dot{W}}\dot\Theta_{\dot{w}(m)}^{\flat}  = \sum_{w \in  \mathcal{P}_s^l \cap {W}}\dot\Theta_{\dot{w}(m)}^{\flat} + \dot\Theta_{ \dot{s}\dot{w}( m)}^{\flat} .$$
So $r_m^{\flat}$ commutes with $i_{s}^{\flat}$ for all $s\in \cS^{\flat}$. 
This completes the proof since $W$ is generated by $ \cS^{\flat} \cap W_{a_\circ}^{\flat}$.
\end{proof}

\section{Twisted integral Satake morphisms}\label{sectionsatake}
\subsection{Twisted Satake homomorphism: first basic properties}\label{definitionsatake}

In this section, we focus on the level $K^{\flat}$. 
By Proposition \ref{freerank1MK}, 
the $\mathcal{R}^{\flat}$-module $\cM_{K^{\flat}}(\Z)$ is free of rank $1$ (with canonical basis element the spherical vector $v_{1,K}^{\flat}:=v_{1,a_\circ}^{\flat}={\bf 1}_{U^+K^{\flat}}$).
Accordingly, we have a $\Z$-algebra homomorphism $\cH_{K^{\flat}}(\Z)\to \mathcal{R}^{\flat}$, which we denote by $\dot{\mathcal{S}}_M^{G^{\flat}}$. It is called the Satake transform and is characterized by
$$v_{1,K}^{\flat}*_{K^{\flat}} h = \dot{\mathcal{S}}_M^{G^{\flat}}(h)\cdot v_{1,K}^{\flat}, \text{ for all }h\in \cH_{K^{\flat}}(\Z).$$ 
This is actually an embedding of $\Z$-algebras (using the fact that $v_{1,K}^{\flat}=v_{1}^{\flat}*_{I^{\flat}}{\bf 1}_{K^{\flat}}$ and Corollary \ref{injectivebarh}). In particular, the algebra $\cH_{K^{\flat}}(\Z)$ is a commutative.
\begin{proposition}\label{incenter}
The image of $\dot{\mathcal{S}}_M^G$ is $\dot{W}$-invariant, i.e.
$\begin{tikzcd}\dot{\mathcal{S}}_M^{G^{\flat}} \colon \cH_{K^{\flat}}(\Z) \arrow[r,hook]& {\mathcal{R}^{\flat}}^{\dot{W}}.\end{tikzcd}$
\end{proposition}
\begin{proof}
We give two proofs in the following two subsubsections; the first is "algebraic" while the second is of "analytical" taste. The first is short while the second is explicit:

First, note that by Corollary \ref{ontoaveraging} and Remark \ref{flatiswequi}, it is sufficient to show$\begin{tikzcd}[column sep= small]\dot{\mathcal{S}}_M^{G} \colon \cH_{K^{\flat}}(\Z) \arrow[r,hook]& {\mathcal{R}}^{\dot{W}}.\end{tikzcd}$
\subsubsection{Algebraic argument}
\begin{proof}
Let $h\in \cH_K(\Z) \subset \cH_I(\Z[q^{-1}])$. 
We have $$v_{1}*_I h= v_{1,K}*_K h= \dot{\mathcal{S}}_M^G (h) \cdot v_{1,K} =  \sum_{w \in W_{a_\circ}} \dot{\mathcal{S}}_M^G (h) \cdot v_w.$$ 
Hence, $h= \dot{\mathcal{S}}_M^G (h) *_I \sum_{w \in W_{a_\circ}} i_w  =\dot{\mathcal{S}}_M^G (h) *_I {\bf 1}_K$. 
Now, because $h$ is $K$-bi-equivariant, we have for any $s\in \mathcal{T}_{a_\circ}$
$$ i_s *_I h=i_s *_I {\bf 1}_K*_Kh= h= h*_I i_s.$$
Therefore, for any $s\in \cS \cap W_{a_\circ}$:
$$ 0=i_s*_I \dot{\mathcal{S}}_M^G (h) *_I {\bf 1}_K-\dot{\mathcal{S}}_M^G (h)*_I i_s *_I {\bf 1}_K= \left(i_s*_I \dot{\mathcal{S}}_M^G (h) -\dot{\mathcal{S}}_M^G (h)*_I i_s\right) *_I   \sum_{w\in W_{a_\circ}} i_w.$$
Therefore, given that $i_s*_I \dot{\mathcal{S}}_M^G (h) -\dot{\mathcal{S}}_M^G (h)*_I i_s \in \mathcal{R}$ by Proposition \ref{propoTheta}, 
by Corollary \ref{HIbasis} we must have
$$i_s*_I \dot{\mathcal{S}}_M^G (h)-\dot{\mathcal{S}}_M^G (h)*_I i_s=0, \forall s \in\mathcal{T}_{a_\circ}.$$
This concludes the first proof of the proposition, since $W_{a_\circ}$ is generated by $\mathcal{T}_{a_\circ}$.
\end{proof}
\subsubsection{Analytical argument}
First, we need an explicit integral formula for the Satake transform:
\begin{lemma}\label{explicitsatake1}
Let $h$ be any function in $\cH_K(\Z)$, then its Satake twisted transform is explicitly given by
$$\dot{\mathcal{S}}_M^G( h)\colon m \mapsto \int_{U^+}  h(um)d\mu_{U^+}(u),
$$
where $d\mu_{U^+}$ is the Haar measure on $U^+$ giving $1$ on $U^+\cap K$.
\end{lemma}
\begin{proof}
We just need to evaluate $v_{1,K }*_Kh = \dot{\mathcal{S}}_M^G(h)\cdot v_{1,K }$ on both sides:
On the one hand, we have for $m\in M$
\begin{align*}(\dot{\mathcal{S}}_M^G(h) \cdot  v_{1,K })(m)
=\dot{\mathcal{S}}_M^G(h)(m).
\end{align*}
On the second hand we have\footnote{We refer to \cite[\S 4.1]{Cartier79} for all integration formulas used in the sequel} thanks to the Iwasawa decomposition $G=BK$ (Corollary \ref{Iwasawadeccor}) 
\begin{align*}(v_{1,K }*_Kh)(m)
&=\int_{B}\int_K v_{1,K }(bk)  h(k^{-1}b^{-1}m) d\mu_B(b) \,dk\\
&=\int_{B} v_{1,K }(b) h(b^{-1}m)d\mu_B(b) \quad\quad\quad(K \text{ invariance}) \\
&=\int_{M}\int_{U^+} v_{1,K }(au) h(u^{-1}a^{-1}m)d\mu_M(a)d\mu_{U^+}(u)  \\
&=\int_{U^+}  h(u^{-1}m)d\mu_{U^+}(u)\quad\quad\quad\quad(v_{1,K }(au) \neq 0 \Rightarrow a \in M_1)  \\
&=\int_{U^+}  h(um)d\mu_{U^+}(u)\quad\quad\quad\quad\quad (U^+ \text{ is unimodular})
\end{align*}
Where, $d\mu_B$ (respectively $d\mu_{U^+}$ and $dk$) is the Haar left invariant measure on $B$ (respectively  $U^+$ and $K$) giving volume $1$ to $ B\cap K$ (respectively $U^+ \cap K$ and $K$).
\end{proof}
\begin{remark}\label{theregulars}\label{regulardef} 
The function $ m \mapsto  \det(\text{Ad}_{\text{Lie} (U^+)}(m)-\text{Id}_{\text{Lie} (U^+)}) $ is polynomial and nonzero. 
The set of regular elements is, by definition, the dense set of elements of $M$ which do not annihilate the nonzero polynomial function $ m \mapsto  \det(\text{Ad}_{\text{Lie} (U^+)}(m)-\text{Id}_{\text{Lie} (U^+)}) $.
\end{remark} 
\begin{lemma}\label{explicitsatake2}
Let $h$ be any function in $\cH_K(\Z)$ and $m\in M$ any regular elements (Definition \ref{theregulars}). Then
$$\int_{U^+}  h(um)d\mu_{U^+}(u)=\Delta(m)\int_{G/S} h(gmg^{-1})\frac{d\mu_K(g)}{ds}$$
where, the Haar measure $ds$ on $S=\Sbf(F)$ is normalized by $\int_{M/S}\frac{d\mu_{M_1}}{ds}=1$.
\end{lemma}
\begin{proof}
The same reasoning of \cite[Lemma 4.1]{Cartier79} proves the lemma.
\end{proof}
\begin{remark}\label{independenceB}As opposed to \cite[Remark 10.1.1]{HR10}, the twisted Satake transform defined here, is dependent on the choice of the minimal parabolic $\mathbf{B}$ which contains $\Mbf$ as a Levi factor. 
\end{remark}
Given Lemma \ref{explicitsatake1}, we need to prove that $$\dot{\mathcal{S}}_M^G (h)(m)= \dot{w} \left(   \dot{\mathcal{S}}_M^G(h) (m)\right)$$
for all functions $h \in \cH_K(\Z)$, $m \in M$ and $w\in W$. 
By continuity, it suffices to show this equality for all $m$ in the dense set of elements of regular (Remark \ref{regulardef}) which are semi-simple as elements in $G$. 
For these elements, we know by Lemma \ref{explicitsatake2} that
$$\dot{\mathcal{S}}_M^G (h)(m)=\Delta(m)\int_{G/S} h(gmg^{-1})\frac{d\mu_K(g)}{ds}$$
On the one hand, by \cite[(23)]{Cartier79}, one easily shows that $\Delta(\dot{w}(m))=\Delta(m)$. 
On the other hand, since $h$ is $K$-bi-invariant, one has $h(n_w g n_w^{-1})= h(g)$ for all $g \in G$. 
The compact subgroup $N \cap K$ acts by inner automorphisms on $G$ and on $S$ and hence leaves invariant the $G$-measure on $G/S$. 
Accordingly, 
$$\dot{w}\left(  \dot{\mathcal{S}}_M^G(h) (m)\right) = \Delta(w\bullet m) \int_{G/S} h(g (n_wmn_w^{-1}) g^{-1})\frac{d\mu_K(g)}{ds}=\Delta(m)\int_{G/S} h(gmg^{-1})\frac{d\mu_K(g)}{ds} =  \dot{\mathcal{S}}_M^G(h) (m).$$
This ends the "analytical" argument for Proposition \ref{incenter}. 
\end{proof}
\begin{remark}\label{transitionhecke}
The following two diagrams are commutative
$$
\begin{tikzcd}
\cH_{K}(\Z) \arrow[hook]{r}{\dot{\cS}_M^{G}}\arrow[two heads, swap]{d}{\iota^{\flat}=-*_{K_\cF}\mathbf{1}_{K_\cF^{\flat}}} & \mathcal{R}^{\dot{W}}\arrow[two heads]{d}{ \square^{\flat}}& \cH_{K}(\Z) \arrow[hook]{r}{\dot{\cS}_M^{G}} & \mathcal{R}^{\dot{W}} \\
\cH_{K^{\flat}}(\Z) \arrow[hook]{r}{\dot{\cS}_M^{G^{\flat}}} & {\mathcal{R}^{\flat}}^{\dot{W}}&\cH_{K^{\flat}}(\Z) \arrow[hook]{r}{\dot{\cS}_M^{G^{\flat}}}\arrow[hook]{u}& {\mathcal{R}^{\flat}}^{\dot{W}}\arrow[hook]{u}
\end{tikzcd}$$
Here, the last right vertical map is the natural inclusion $\Lambda_M^{\flat} \to \Lambda_M$ given by $mM^{\flat} \mapsto \sum_{m'\in M^{\flat}/M_1} mm'M_1$.
\end{remark}

\subsection{Twisted Satake isomorphism}\label{untwistsatakeK}
Motivated by arithmetic problems for Shimura varieties, Haines and Rostami already established in \cite{HR10}, $\cS_M^G\colon \cH_K(\C)\iso \mathcal{R}^W\otimes_\Z \C$. In this section, we extend their result to $\Z$-coefficients provided by the twisted Satake homomorphism $\dot{\cS}_M^G$. 
\begin{theorem}\label{Ztwistedsatakeisomorphism}\label{twistedsatakeisomorphism}The twisted Satake homomorphism $\dot{\cS}_M^{G^{\flat}}\colon \cH_{K^{\flat}}(\Z) \to {\mathcal{R}^{\flat}}^{\dot{W}}$ is a canonical isomorphism of ${\Z}$-algebras.
\end{theorem}
As for Proposition \ref{incenter}, we give in the following two subsubsections two proofs; the first is "algebraic" while the second is of "combinatorial" taste. 
We will see in \S \ref{compressiblesection}, yet a third method which will give simultaneously this type of isomorphisms for all parahoric subgroup $K_\cF$ with $\cF \subset \a$.

\subsubsection{Algebraic argument}

\begin{proof}
\begin{lemma}[Integral compatibility]\label{Zcompatibility}
The Satake and Bernstein twisted $\Z$-homorphisms are compatible, i.e.,
the following diagram (of $\Z$-modules) is commutative:
$$\begin{tikzcd}  \cH_{K^{\flat}}(\Z) \otimes_\Z\Z[q^{-1}]\arrow[hook]{rr}{\dot{\cS}_M^{G^{\flat}} \otimes \text{Id}}&& {\mathcal{R}^{\flat}}^{\dot{W}}\otimes_\Z\Z[q^{-1}]\arrow[hook]{dl}{\dot\Theta_{Bern}^{\flat}}\\
&Z(\cH_{{I^{\flat}}}(\Z[q^{-1}]))\arrow{lu}{-*_{I^{\flat}}{\bf 1}_{K^{\flat}}}&
\end{tikzcd}$$
\end{lemma}
The injectivity of the horizontal and right diagonal maps is Corollary \ref{invincentIwahori} and Proposition \ref{incenter}. 
Let $h\in \cH_{K^{\flat}}({\Z})$, we have
\begin{align*} v_{1} *_{I^{\flat}} h&= v_{1}*_{I^{\flat}} e_{K^{\flat}} *_{{I^{\flat}}} h   &\quad(e_{K^{\flat}}*_{I^{\flat}} h =h)\\\hspace*{\fill}
&={[K^{\flat}:I^{\flat}]} (v_1^{\flat}*_{I^{\flat}} e_{K^{\flat}} )*_{K^{\flat}} h   &\quad(\text{passing from $*_{I^{\flat}}$ to $*_{K^{\flat}}$})\\\hspace*{\fill}
&= {[K^{\flat}:I^{\flat}]}\dot{\mathcal{S}}_M^{G^{\flat}}(h)\cdot( v_1^{\flat}*_{I^{\flat}}e_{K^{\flat}})  &(\text{Definition of $\dot{\mathcal{S}}_M^{G^{\flat}}$})\\
&= {[K^{\flat}:I^{\flat}]}(\dot{\mathcal{S}}_M^{G^{\flat}}(h)\cdot v_1^{\flat})*_{I^{\flat}}e_{K^{\flat}}  \\
&= {[K^{\flat}:I^{\flat}]}  (v_1^{\flat}*_{I^{\flat}}\dot{\Theta}_{\text{\emph{Bern}}}^{\flat}({\dot{\mathcal{S}}_M^{G^{\flat}}(h)}))*_{I^{\flat}}e_{K^{\flat}} & (\text{Definition of $\dot{\Theta}_{\text{\text{Bern}}}^{\flat}$})\\
&=v_1^{\flat}*_{I^{\flat}}\dot{\Theta}_{\text{\emph{Bern}}}\circ {\dot{\mathcal{S}}_M^{G^{\flat}}(h)}*_{I^{\flat}}\mathbf{1}_{K^{\flat}}
\end{align*}
Hence, $h=\dot{\Theta}_{\text{\emph{Bern}}}^{\flat}\circ {\dot{\mathcal{S}}_M^{G^{\flat}}(h)}*_{I^{\flat}}\mathbf{1}_{K^{\flat}}=\mathbf{1}_{K^{\flat}}*_{I^{\flat}}\dot{\Theta}_{\text{\emph{Bern}}}^{\flat}\circ {\dot{\mathcal{S}}_M^{G^{\flat}}(h)}$.\qed

Accordingly, Theorem \ref{twistedsatakeisomorphism}, follows readily from:
\begin{lemma}\label{explicithsatake1k}
For any $h\in \cH_{K^{\flat}}(\Z)$, and any $f \in {\mathcal{R}^{\flat}}^{\dot{W}}$ we have $$h=\dot\Theta_{\text{Bern}}^{\flat}(\dot{\cS}_M^{G^{\flat}}(h))*_{I^{\flat}}{\bf 1}_{K^{\flat}} \text{ and } \dot{\cS}_M^{G^{\flat}}(\dot\Theta_{\text{Bern}}^{\flat}(f)*_{I^{\flat}}{\bf 1}_{K^{\flat}} )=f.$$ 
In other words, $(-*_{I^{\flat}} {\bf}_{I^{\flat}})\circ \dot{\Theta}_{\text{Bern}}^{\flat}$ is the inverse of $\dot{\cS}_M^{G^{\flat}}$. 
\end{lemma}
Let $f \in {\mathcal{R}^{\flat}}^{\dot{W}}$ so $\dot\Theta_{\text{Bern}}^{\flat}(f)\in Z(\cH_{I^{\flat}}(\Z[q^{-1}]))$. 
Then, clearly $h_f=\dot\Theta_{\text{Bern}}^{\flat}(f)*_{I^{\flat}} {\bf 1}_{K^{\flat}}= {\bf 1}_{K^{\flat}}*_{I^{\flat}}\dot\Theta_{\text{Bern}}^{\flat}(f) \in \cH_{K^{\flat}}(\Z)$. 
Now, given that $v_{1,K}^{\flat}*_{K^{\flat}} h_f= v_1^{\flat}*_{I^{\flat}} h_f= f\cdot v_1^{\flat}*_{I^{\flat}} {\bf 1}_{K^{\flat}}$, one concludes that $\dot{\cS}_M^{G^{\flat}}(h_f)=f.$
 
If we start with any element $h\in \cH_{K^{\flat}}(\Z)$, the same reasoning shows that
$$\dot{\cS}_M^{G^{\flat}}(h)=\dot{\cS}_M^{G^{\flat}}( \dot\Theta_{\text{Bern}}^{\flat} \circ\dot{\cS}_M^{G^{\flat}}(h)*_{I^{\flat}} {\bf 1}_{K^{\flat}})$$ and Proposition \ref{incenter} then implies $h= \dot\Theta_{\text{Bern}}^{\flat} \circ\dot{\cS}_M^{G^{\flat}}(h)*_{I^{\flat}} {\bf 1}_{K^{\flat}}.$
\end{proof}

\subsubsection{Combinatorial argument}\label{combarg}
Thanks to Remark \ref{transitionhecke}, it is sufficient to show that $\dot{\mathcal{S}}_M^{G} \colon \cH_{K}(\Z) \iso {\mathcal{R}}^{\dot{W}}$ ($M^{\flat}=M_1$) to prove Theorem \ref{Ztwistedsatakeisomorphism}. So, we focus on this case in this subsection.
\begin{definition}\label{presimorder}
Let $\precsim^{\flat}$ be the partial order on $\Lambda_M^{\flat}$ defined as follows
$$m\succsim^{\flat}m' \text{ if and only if }m-m' \in \Lambda_{\aff}^{\flat}  \text{ and } \nu(m-m')= \sum_{\Delta^\vee} n_\alpha \alpha^\vee, \text{ with } n_\alpha \in \N.$$
\end{definition}
\begin{remark}\label{presimtrans}
(i) When $M^{\flat}=M^1$, this is the usual order on the lattice $\Lambda_M^1$. 
(ii) Note that if $m,m'\in \Lambda_M$ then $m'\precsim m$ if and only if $m'^{\flat}\precsim^{\flat} m^{\flat}$, because (i) $\nu$ factors through $M/\ker \nu$ and a fortiori through $\Lambda_M^1$, and (ii) because $\square^{\flat}\colon \Lambda_M \twoheadrightarrow \Lambda_M^{\flat}$ restricts to an isomorphism $\Lambda_{\aff} \to \Lambda_{\aff}^{\flat}$.
\end{remark}!
\begin{lemma}\label{Proposition444ii}
If $m\in \Lambda_M^-$, then $U^+m^{-1}K \cap Km^{-1}K=m^{-1}K$.
\end{lemma}
\begin{proof}
We intend to use \cite[4.3.6 Corollaire]{BT72} for $G_1$ by taking (with {\em loc. cit.} notations)  $\Upomega=\{a_\circ\}$ and $\Upomega'=\{a_\circ+\nu(m)\}$ and $n=1$. Since $\Upomega \subset \Lambda_M^-$, we have $a_\circ \subset \Upomega' + \overline{\cC}^+$, so the condition of \cite[4.3.6 Corollaire]{BT72} is fulfilled and we have $U^+K \cap K_{a_\circ + \nu(m)}K=K$. This shows the lemma.
\end{proof}
\begin{lemma}\label{Proposition444}
Let $m,m'\in \Lambda_M$, if $U^+m' \cap KmK\neq \emptyset$ then $m' \precsim m$.
\end{lemma}

\begin{proof}
If the above intersection is nonempty, i.e. $u \in U^+ \cap {m'}^{-1}KmK$. 
There exists then $m_0 \in M$ such that $m_0um_0^{-1} \in I$, hence, the alcove $\a'=\nu(m_0)(\a)$ verifies 
$$I_{\a'}m'm^{-1} \cap KK_{a_\circ+ \nu(m)}\neq \emptyset$$
So in particular, $m-m' \in \Lambda_{\aff}$. We may then apply \cite[4.3.16 Corollaire]{BT72} for $G_1$, by taking (with {\em loc. cit.} notations) $y=a_\circ+ \nu(m)$, $x=a_\circ, n=m'-m\in \Lambda_{\aff}$, $C=\a'$ and $D= \cC^-$ our fixed negative chamber. Then, the assumption $y\in x + \cC^-$ is fulfilled because $m \in \Lambda_M^-$. 
Therefore, we have
$$y+\nu(m'-m)= a_\circ +\nu(m') \le_D y = a_\circ +\nu(m).$$
The order $\le_D$ in {\em loc. cit.}\footnote{We refer to \cite[$n^\circ$1 \S 6 p. 155]{BourbakiLieV} for more details on the order used in \cite[4.3.16 Corollaire]{BT72}.} translates here to $\nu(m-m') = \sum_{\alpha \in \Delta} n_\alpha\alpha^\vee, n_\alpha \ge 0$. 
\end{proof}

\begin{proof}[Second proof of Theorem \ref{twistedsatakeisomorphism}]
This resembles the proof in \cite[\S 10.2]{HR10} for the untwisted Satake transform with coefficients in $\C$, which remains morally valid in our twisted context, although we deal with the "lexicographic" order on  $\Lambda_M^-$ in a different manner.  
We need to show that the following morphism of $\Z$-modules
$$\begin{tikzcd}[column sep=small, row sep=small]\cH_K({\Z}) \arrow[hook]{r} & \mathcal{R}^{\dot{W}}\cdot v_{1,K}, h\arrow[r, mapsto] &\dot{\mathcal{S}}_M^G(h) *_Kv_{1,K}\end{tikzcd}$$ is "upper triangular with respect to some total order, with invertible diagonals". 
The proof is similar to the proof of Theorem \ref{structureMI}. 

Recall that $\{h_{x}\colon x\in \Lambda_M^-\}$ forms a $\Z$-basis for $\cH_K(\Z)$, we also have a $\Z$-basis for $\mathcal{R}^{\dot{W}}\cdot v_{1,K}$ given by  $\{r_m \cdot v_{1,K} \colon m \in \Lambda_M^-\}$. 
Fix an element $x \in \Lambda_M^-$, we have
$$v_{1,K}*_Kh_x=\dot{\mathcal{S}}_M^G(h_x)\cdot v_{1,K}= \sum_{m\in  \Lambda_M^-} s_{x,m} r_m, \text{ where } s_{x,m}=|U^+m^{-1}K \cap Kx^{-1}K|_K \ge 0.$$ 
For $m \in \Lambda_M^-$, $r_m$ appears in the above summation, i.e. $s_{x,m}\neq 0$ if and only if $U^+m \cap KxK\neq \emptyset$. 
Accordingly, by Lemma \ref{Proposition444}, 
we obtain
$$v_{1,K}*_Kh_x= \sum_{m \precsim x} s_{x,m} r_m.$$ 
Therefore, since the monoid $\Lambda_M^-$ is countable and any element $x\in \Lambda_M^-$ has only finitely many predecessors with respect to the  partial order $\precsim$, there exists a lexicographic (total) ordering $x_1,x_2 , \cdots$ for the elements of $\Lambda_M^-$. 
The above discussion shows then, that the matrix of the transformation $h \mapsto \dot{\mathcal{S}}_M^G(h)$ with respect to the bases $\{h_{x_i}\}_1^\infty$ and $\{{x_i}\}_1^\infty$ is upper triangular.

Finally, note that the diagonals $s_{x,x}=|U^+x^{-1}K \cap Kx^{-1}K|_K$ are equal to $1$ by Lemma \ref{Proposition444ii}. 
This observation shows that the triangular matrix $\cS=\left(s_{x,m}\right)_{x,m \in \Lambda_M^-}$ is indeed invertible, which concludes the combinatorial proof of Theorem \ref{Ztwistedsatakeisomorphism}. 
\end{proof}
\begin{remark}\label{genkotwitz}
By the previous proof, for any $m\in \Lambda_M^-$, we have $\dot{\cS}_M^G(h_m)= \sum_{m'\precsim m} s_{m,m'} r_{m'}$. 
In particular, if $m$ is minuscule, hence minimal with respect to the partial order $\precsim$, then $\dot{\cS}_M^G(h_m)=r_m$. 
This yields an alternative proof (and generalizes) for Kottwitz's lemmas \cite[Lemma 2.3.3 \& Lemma 2.3.7 (b)]{Ko1}. 
\end{remark}

\begin{remark}[Comparison of the untwisted and twisted Satake isomorphisms]\label{untwistedtotwisted}
Consider the $\C$-linear map
$$\begin{tikzcd}\eta_B\colon \mathcal{R}\otimes_\Z \C \arrow{r}&  \mathcal{R}\otimes_\Z \C; m\otimes c \arrow[r,mapsto]& \delta_B(m)^{1/2} m\otimes c.\end{tikzcd}$$
Observe that for any $m\in \Lambda_M$, we have $\eta_B( \sum_{w\in W/W_m} w(m)={\delta(m)^{1/2}} r_m$. 
Hence, $\eta_B\colon \mathcal{R}\otimes_\Z \C \to  \mathcal{R}\otimes_\Z \C$ is actually an isomorphism of $\C$-algebras with $W$-action where the target is endowed with the dot-action (Definition \ref{dotaction}) and the source with the usual action. Let $\mathcal{S}_M^G\colon \cH_K(\C) \to  \mathcal{R}\otimes_\Z \C$ be the "standard" Satake homomorphism \cite[\S 9.2]{HR10}, then we have
$$ \eta_B \circ  \mathcal{S}_M^G= \dot{ \mathcal{S}}_M^G\otimes \text{id}_\C.\qedhere$$
\end{remark}
\subsection{Positivity properties of the Satake isomorphism}
In this subsection we show that the natural generalization of Rapoport's positivity result for unramified groups \cite[Theorem 1.1]{Ra} and Haines listed properties \cite[Theorem 1.2]{Haines00} remains true in general. 
\begin{theorem}\label{positivitythm}
For any $x\in \Lambda_M^-$, write $\dot{\cS}_M^G(h_x)= \sum_{m \in \Lambda_M^-} s_{x,m} r_{m}$. The following statements are equivalent
\begin{enumerate}
\item $s_{x,m}>0$,
\item $U^+m \cap KxK \neq \emptyset$,
\item $U^+m \cap Iwxw'I \neq \emptyset$, for some $w,w' \in W_{a_\circ}$,
\item ${x} \succsim {m}$.
\end{enumerate}
\end{theorem}
\begin{proof}
The equivalence between (1) and (2) was already established while proving Theorem \ref{twistedsatakeisomorphism}, in which it is also proved (1) $\Rightarrow$ (4).
The equivalence between (2) and (3) is clear from Lemma \ref{KmK}.

Now, we show (4) $\Rightarrow$ (1): C. Schwer proved in \cite{Schwer06} such an implication for the case of extended affine Weyl groups attached to root data. We shall apply his result to  $\cD:=({\Lambda}_M^1, \Sigma^{1,\vee},{\Lambda}_M^{1,\vee}, \Sigma^1,\Delta^{1,\vee})$. 
As we saw in Remark \ref{transitionhecke}, the following diagram is commutative $\left(\dot{\cS}_M^{G} ( h_x ) \right)^1=\dot{\cS}_M^{G^1} \iota^1(h_x )$, hence
$$\sum_{m\precsim x} s_{x,m} r_{m^1}= \sum_{y\precsim^1 x^1} s_{x^1,y} r_{y} \in {\mathcal{R}^1}^{\dot{W}}$$
By Remark \ref{presimtrans}, $m \precsim x$ if and only if ${m}^1 \precsim {x}^1$. 
Accordingly, $\sum_{m\precsim x} s_{x,m} r_{m^1}= \sum_{m\precsim x} s_{x^1,m^1} r_{m^1} \in {\mathcal{R}^1}^{\dot{W}}$ and so $ s_{x,m}=  s_{x^1,m^1}$, for all $m \precsim x$. 
We can actually see it directly; because $M^1$ normalizes $K$ and $x-m \in \Lambda_{\aff}$, we get
$$ s_{x^1,m^1}=|(U^+K \cap mKx^{-1}K^1) K^1|_{K^1}=|(U^+K \cap mKx^{-1}K) K^1|_{K^1}=s_{x,m}.$$
Thanks to \cite[Corollary 4.8]{Schwer06}, we have the implication $ x^1\precsim m^1 \Rightarrow s_{m^1,x^1} >0$, which ends the proof of the theorem.
\end{proof}

\section{A compressible algebra and centers of parahoric-Hecke algebras}\label{compressiblesection}

\subsection{Normalized Intertwiners}\label{normintertwin}

The dot-action of $W$ on ${\mathcal{R}^{\flat}}$ stabilizes ${\mathcal{R}_{\aff}^{\flat}}$, 
hence it extends to a dot-action of $W$ on $\mathcal{L}_{\aff}^{\flat} \otimes_{\mathcal{R}_{\aff}^{\flat}} \mathcal{R}^{\flat}$. 
Now, since $W$ action's is locally finite, we have by \cite[\S1 n$^{\circ}$ 9, Proposition 23]{BourbakiAlgComV} two canonical isomorphisms:
$$\mathcal{L}^{\flat}:=\mathcal{L}_{\aff}^{\flat} \otimes_{\mathcal{R}_{\aff}^{\flat}} \mathcal{R}^{\flat}  \simeq {\mathcal{L}_{\aff}^{\flat}}^{\dot{W}} \otimes_{{\mathcal{R}_{\aff}^{\flat}}^{\dot{W}}} \mathcal{R}^{\flat}
\text{ and }
{\mathcal{L}^{\flat}}^{\dot{W}}  \simeq {\mathcal{L}_{\aff}^{\flat}}^{\dot{W}} \otimes_{{\mathcal{R}_{\aff}^{\flat}}^{\dot{W}}} {\mathcal{R}^{\flat}}^{\dot{W}}  $$
Tensoring the decomposition provided by Corollary \ref{HIbasis} with ${\mathcal{L}^{\flat}}^{\dot{W}}$ over ${\mathcal{R}_{\aff}^{\flat}}^{\dot{W}}$, yields 
\begin{equation}\label{decompLaff}
{\mathcal{L}_{\aff}^{\flat}}^{\dot{W}} \otimes_{{\mathcal{R}_{\aff}^{\flat}}^{\dot{W}}}  \cH_{I^{\flat}}({\Z[q^{-1}]})= \bigoplus_{w\in W_{a_\circ}^{\flat}}i_w^{\flat}*_{I^{\flat}} \mathcal{L}^{\flat}= \bigoplus_{w\in W} \mathcal{L}^{\flat} *_{I^{\flat}} i_w^{\flat}.
\end{equation}
\begin{lemma}\label{fracpropoTheta}
For any $f \in \mathcal{L}^{\flat}$,  and $s \in \cS^{\flat}$ we have
$$f*_{I^{\flat}} (i_s^{\flat}+1)-(i_s^{\flat}+1) *_{I^{\flat}}  \dot{s}(f)= (f-\dot{s}(f))*_{I^{\flat}} \mathcal{G}(\alpha).$$
\end{lemma}
\begin{proof}
It suffices to show this statement for $f\in \mathcal{R}^{\flat}$, which follows readily from Proposition \ref{propoTheta}. 
\end{proof}
For any $s\in \mathcal{T}_{a_\circ}$, set
$$\mathscr{k}_s^{\flat}:=-1+({i}_{s}^{\flat}+1)\mathcal{G}(\alpha)^{-1} \in \mathcal{L}_{\aff}^{\flat}.$$
\begin{proposition}\label{propKw}
\begin{enumerate}[(a)]
\item There exists a unique group homomorphism 
$$\mathcal{K}^{\flat}\colon \dot{W}_{a_\circ} \hra \left(\mathcal{L}_{\aff}^{\flat} \otimes_{\mathcal{R}_{\aff}^{\flat}} \cH_{I^{\flat}}({\Z[q^{-1}]})\right)^\times $$
such that $\mathcal{K}^{\flat}(s)=\mathscr{k}_{s}^{\flat}$ for all $s \in \mathcal{T}_{a_\circ}$. 
\item For any $f \in \mathcal{L}^{\flat}$, we have $f *_{{I}^{\flat}}\mathscr{k}_w^{\flat}= \mathscr{k}_w^{\flat} *_{{I}^{\flat}} \dot{w}^{-1} (f)$.
\item $\mathcal{L}_{\aff}^{\flat} \otimes_{\mathcal{R}_{\aff}^{\flat}}  \mathcal{H}_{{I}^{\flat}}({\Z[q^{-1}]})= \bigoplus_{w\in W}\mathcal{L}^{\flat} *_{I^{\flat}} \mathscr{k}_w^{\flat}=\bigoplus_{w\in W} \mathscr{k}_w^{\flat}*_{I^{\flat}}\mathcal{L}^{\flat} $.
\end{enumerate}
\end{proposition}
\begin{proof}
(a) The condition forces the image of $\mathcal{K}$ to be in $\mathcal{L}_{\aff}^{\flat} \otimes_{\Lambda_{\aff}} \cH_{^{\flat}}^{\aff}({\Z[q^{-1}]})=\bigoplus_{w\in W} \mathcal{L}_{\aff}^{\flat} *_{I^{\flat}} i_w^{\flat}$. 
The existence and unicity follows then from \cite[Proposition 5.2 (a)]{Lu89} applied to $\cH_{{I}^{\flat}}^{\aff}({\Z[q^{-1}]}) \simeq \cH({W}_{\aff},\cD_{\aff},v=q, L)$.

(b) Because $\mathcal{K}$ is a morphism of groups, it suffices to check the statement (b) for $w=s\in \mathcal{T}_{a_\circ}^{\flat}$, which follows from Lemma \ref{fracpropoTheta} and the definition of $\mathscr{k}_s^{\flat}$.

(c) We proceed as in the proof of \cite[Proposition 5.5 (a)]{Lu89}. 
Let $X=\sum_{w\in W} \mathcal{L}^{\flat} *_{I^{\flat}} \mathscr{k}_w$. 
Since $W_{a_\circ}$ is generated by $ \mathcal{T}_{a_\circ}$, it suffices to show that $i_s^{\flat}\in X$, for all $s\in \mathcal{T}_{a_\circ}$. 
But, by definition, we have ${i}_{s}^{\flat}\in \mathcal{L}^{\flat} (1+\mathscr{k}_s^{\flat})$. 
Therefore, $\mathcal{L}_{\aff}^{\flat} \otimes_{\Lambda_{\aff}} \cH_{I^{\flat}}^{\flat}({\Z[q^{-1}]}) = X$. 
We saw in (\ref{decompLaff})  p.~\pageref{decompLaff}, that $\mathcal{L}_{\aff}^{\flat} \otimes_{\mathcal{R}_{\aff}^{\flat}} \cH_{I^{\flat}}({\Z[q^{-1}]})$ is free over $\mathcal{L}$ of rank $|W|$. 
Therefore, $\mathcal{L}_{\aff}^{\flat} \otimes_{\mathcal{R}_{\aff}^{\flat}} \cH_{I^{\flat}}({\Z[q^{-1}]}) = \oplus_{w\in W} \mathcal{L}^{\flat} *_{I^{\flat}} \mathscr{k}_w^{\flat}$. 
The other equality is proved in a similar way.
\end{proof}

\subsection{A skew group ring}
\begin{proposition}\label{skewringLaff}
Let $\mathcal{L}^{\flat} *\dot{W}$ be the skew group ring of $\dot{W}$ over $\mathcal{L}^{\flat}$ (See Definition \ref{sgring}). 
We have an isomorphism of $\mathcal{L}^{\flat} $-algebras: 
$$  \mathcal{L}^{\flat} *\dot{W}\to \mathcal{L}_{\aff}^{\flat} \otimes_{\mathcal{R}_{\aff}^{\flat}}  \mathcal{H}_{{I}^{\flat}}({\Z[q^{-1}]}), r \cdot w\mapsto r *_{I^{\flat}} \mathscr{k}_w^{\flat}.$$
\end{proposition}
\begin{proof}
This follows readily from (b) and (c) of Proposition \ref{propKw}.
\end{proof}
\subsection{An Azumaya algebra}

\begin{lemma}\label{sgrmatrix}
\begin{enumerate}[(i)]
\item The skew group ring $\mathcal{L}_{\aff} *\dot{W} $ is the matrix algebra $\mathbb{M}_{|W|}(\mathcal{L}_{\aff}^{\dot{W}})$.
\item The two ring extensions $\mathcal{L}_{\aff}/\mathcal{L}_{\aff}^{\dot{W}}$ and $\mathcal{L}^{\flat}/{\mathcal{L}^{\flat}}^{\dot{W}}$ are ${\dot{W}}$-Galois. 
\end{enumerate}
\end{lemma}
\begin{proof}
(i) There is an obvious map of algebras 
$
\mathcal{L}_{\aff}*{\dot{W}} \to \End_{\mathcal{L}_{\aff}^{\dot{W}}} (\mathcal{L}_{\aff})
$ 
where, an element $a \in \mathcal{L}_{\aff}$ is taken to multiplication by $a$ on $\mathcal{L}_{\aff}$, and $w \in {\dot{W}}$ is taken to the obvious map $w: \mathcal{L}_{\aff} \to \mathcal{L}_{\aff}$. 
By an argument similar to the proof of linear independence of characters, one can check that this map is injective. 
Now since each side has dimension $|W|^2$ as vector space over $\mathcal{L}_{\aff}^{\dot{W}}$, this map is an isomorphism. 

(ii) We deduce from (i) that $\mathcal{L}_{\aff}*{\dot{W}}$ is Azumaya over $\mathcal{L}_{\aff}^{\dot{W}}$. 
Accordingly, Theorem \ref{fundgaloiscom} shows that $\mathcal{L}_{\aff}/\mathcal{L}_{\aff}^{\dot{W}} $ 
is ${\dot{W}}$-Galois, and so is $\mathcal{L}^{\flat}/{\mathcal{L}^{\flat}}^{\dot{W}}$ using Corollary \ref{obviouscor}.
\end{proof}
\begin{proposition}\label{Laffgalois}
The ${\Z[q^{-1}]}$-algebra $ A=\mathcal{L}_{\aff}^{\flat} \otimes_{\mathcal{R}_{\aff}^{\flat}}  \mathcal{H}_{{I}^{\flat}}({\Z[q^{-1}]})$ is compressible, that is $$Z(e A e) = eZ(A)=e\mathcal{L}^{\dot{W}}$$
for any idempotent $e$ of $A$.
\end{proposition}
\begin{proof}
By Theorem \ref{fundgaloiscom} and Lemma \ref{sgrmatrix} (ii), the algebra $\mathcal{L}^{\flat}* \dot{W}$ is Azumaya over $\mathcal{L}^{{\flat}}$, so the proposition is a consequence of Theorem \ref{compress}.
\end{proof}
\subsection{The center of parahoric--Hecke algebras}

\begin{theorem}\label{center}\label{Winvariantcenter}
For any facet $\cF\subset \overline{\a}$, we have
$$Z\left(\cH_{K_{\cF}^{\flat}}({\Z[q^{-1}]})\right)=\dot{\Theta}_{\text{Bern}}^{\flat}({\mathcal{R}^{\flat}}^{\dot{W}}\otimes_\Z\Z[q^{-1}])*_{I^{\flat}}{\bf 1}_{K_\cF^{\flat}}.$$
In particular, the set $\{r_m^{\flat}*_{I^{\flat}}{\bf 1}_{K_\cF^{\flat}}: m \in \Lambda_M^{-,\flat}\}$ forms a basis for the $\Z[q^{-1}]$-module $Z\left(\cH_{K_{\cF}^{\flat}}({\Z[q^{-1}]})\right)$.
\end{theorem}
\begin{lemma}\label{ZeHe}
For any idempotent $e \in \mathcal{H}_{{I}^{\flat}}({\Q})$, we have 
$$Z(e({\mathcal{L}_{\aff}^{\flat}}^{\dot{W}} \otimes_{{\mathcal{R}_{\aff}^{\flat}}^{\dot{W}}}  \mathcal{H}_{{I}^{\flat}}({\Q}))e)= e{\mathcal{L}_{\aff}^{\flat}}^{\dot{W}} \otimes_{e{\mathcal{R}_{\aff}^{\flat}}^{\dot{W}}}  Z( e(\mathcal{H}_{{I}^{\flat}}({\Q}))e).$$
\end{lemma}
\proof Given Corollary \ref{invincentIwahori}, we have an inclusion 
$Z(e({\mathcal{L}_{\aff}^{\flat}}^{\dot{W}} \otimes_{{\mathcal{R}_{\aff}^{\flat}}^{\dot{W}}}  \mathcal{H}_{{I}^{\flat}}({\Q}))e)\supset e{\mathcal{L}_{\aff}^{\flat}}^{\dot{W}} \otimes_{e{\mathcal{R}_{\aff}^{\flat}}^{\dot{W}}}  Z( e(\mathcal{H}_{{I}^{\flat}}({\Q}))e).$
The opposite one is true because: 

(i) The morphism $e\mathcal{H}_{I^{\flat}}({\Z[q^{-1}]})e \to e{\mathcal{L}_{\aff}^{\flat}}^{\dot{W}}  \otimes_{e{\mathcal{R}_{\aff}^{\flat}}^{\dot{W}}}  e\mathcal{H}_{I^{\flat}}({\Z[q^{-1}]})e$, $h \mapsto e\otimes h$ is injective. 
This follows readily from the proof of Lemma \ref{341injLaff}, in which we proved that $\mathcal{H}_{I^{\flat}}(\Z[q^{-1}])$ is ${\mathcal{R}_{\aff}^{\flat}}^{\dot{W}}$-torsion free, so in particular $e\mathcal{H}_{I^{\flat}}({\Q})e$ is also $e{\mathcal{R}_{\aff}^{\flat}}^{\dot{W}}$-torsion free and ${\mathcal{R}_{\aff}^{\flat}}^{\dot{W}} \simeq e {\mathcal{R}_{\aff}^{\flat}}^{\dot{W}}$. 

(ii) Any element in the left hand side can be written as $\ell \otimes h$, for some $\ell \in e{\mathcal{R}_{\aff}^{\flat}}^{\dot{W}}$ and $h\in e\mathcal{H}_{I^{\flat}}({\Z[q^{-1}]})e$, hence we must have $h\in Z(e\mathcal{H}_{I^{\flat}}({\Z[q^{-1}]})e)$. This shows the lemma.
\qed

\begin{proof}[Proof of Theorem \ref{Winvariantcenter}]
Let $\cF\subset \overline{\a}$ be any facet and  $e_{\cF}$ the idempotent $[K_\cF^\flat : I^\flat]{\bf 1}_{K_\cF^\flat}$. 
By Proposition \ref{Laffgalois} and Lemma \ref{ZeHe}, we have
$$e_{\cF}^{\flat} \mathcal{L}^{\dot{W}}=e_{\cF}^{\flat} {\mathcal{L}_{\aff}^{\flat}}^{\dot{W}} \otimes_{e_{\cF}{\mathcal{R}_{\aff}^{\flat}}^{\dot{W}}} e_{\cF}{\mathcal{R}^{\flat}}^{\dot{W}} = e_{\cF}{\mathcal{L}_{\aff}^{\flat}}^{\dot{W}}  \otimes_{e_{\cF}{\mathcal{R}_{\aff}^{\flat}}^{\dot{W}}}  Z(e_{\cF}\mathcal{H}_{I^{\flat}}({\Q})e_{\cF}).$$
For any $h\in Z(e_{\cF}\mathcal{H}_{I^{\flat}}({\Q})e_{\cF})$, there are then $s,r \in  {\mathcal{R}^{\flat}}^{\dot{W}}$ such that $sh- re=0$. 
We have then an equality $s v_{1,\cF}^{\flat}*_{K_{\cF}^{\flat}}h=rv_{1,\cF}^{\flat}$. 
By Propoisition \ref{HIRbasis}, we deduce that $v_{1,\cF}^{\flat}*_{K_{\cF}^{\flat}}h= t v_{1,\cF}^{\flat}$ for some $t \in {\mathcal{R}^{\flat}}^{\dot{W}} \otimes \Q$, this shows $v_{1}^{\flat}*_{I^{\flat}}h=  v_{1}^{\flat} t*_{I^{\flat}} e_{\cF}$ and accordingly $h=t*_{I^{\flat}} e_{\cF}$ by Corollary \ref{cor6}.
In conclusion, 
$${\bf 1}_{K_{\cF}^{\flat}} *_{I}({\mathcal{R}^{\flat}}^{\dot{W}}\otimes_\Z \Q) =Z(e_{\cF}\mathcal{H}_{I^{\flat}}({\Q})e_{\cF})= Z(\mathcal{H}_{K_{\cF}^{\flat}}({\Q})).$$
Let $h\in Z(\mathcal{H}_{K_{\cF}^{\flat}}({\Z[q^{-1}]}))$, by the previous discussion, there exists $a/b \in \Q$ and $r \in {\mathcal{R}^{\flat}}^{\dot{W}}$ such that $h= \frac{a}{b}{\bf 1}_{K_{\cF}^{\flat}} *_{I^{\flat}} \dot{\Theta}_{Bern}^{\flat}(r)$, but by definition the term ${\bf 1}_{K_{\cF}^{\flat}} *_{I^{\flat}} \dot{\Theta}_{Bern}^{\flat}(r) \in \cH_{I^{\flat}}(\Z[q^{-1}])$, hence one must have $\frac{a}{b} \in \Z[q^{-1}]$. Consequently, we have 
\begin{align*} Z(\mathcal{H}_{K_{\cF}^{\flat}}({\Z[q^{-1}]}))&={\bf 1}_{K_{\cF}^{\flat}} *_{I^{\flat}}\dot{\Theta}_{Bern}^{\flat}({\mathcal{R}^{\flat}}^{\dot{W}}\otimes_\Z \Z[q^{-1}]).\qedhere\end{align*}
The linear independence of the elements $r_m^{\flat}*_{I^{\flat}}{\bf 1}_{K_\cF^{\flat}} (m \in \Lambda_M^{-,\flat})$ is clear from the injectivity of $-*_{I^{\flat}}{\bf 1}_{K_\cF^{\flat}}$ and Corollary \ref{HIbasis}.

\end{proof}
\subsection{Compatibility of the twisted  Satake and twisted Bernstein isomorphisms}
From the above, we see that $Z(\mathcal{H}_{K_{\cF}}({\Z[q^{-1}]}))$ is a free $\mathcal{R}^{\dot{W}}\otimes_\Z \Z[q^{-1}]$-module of rank one, with basis the canonical vector ${\bf 1}_{K_{\cF}} $. 
Define $\dot\cS_{\cF}\colon Z(\mathcal{H}_{K_{\cF}}({\Z[q^{-1}]})) \to \mathcal{R}^{\dot{W}}\otimes_\Z \Z[q^{-1}]$, characterized by $\dot\cS_{\cF}(h)*_I{\bf 1}_{K_{\cF}} = h$. 
By Lemma \ref{explicithsatake1k}, we have $\dot\cS_{\cF}=\dot\cS_M^G$ if $\cF=\{a_\circ\}$. 
The Satake and Bernstein twisted isomorphisms are compatible in the following sense:
\begin{theorem}\label{compatibility} 
Let $\cF \subset \overline{\cF}' \subset  \overline{\a}$, be two facets, i.e. $K_{\cF}\supset K_{\cF'}\supset I$. 
The following diagram of $\Z[q^{-1}]$-algebras is commutative:
$$\begin{tikzcd}  Z(\cH_{K_\cF}(\Z[q^{-1}])) \arrow[]{rr}{\dot{\cS}_{\cF} }[swap]{\simeq}&& {\mathcal{R}}^{\dot{W}}\otimes_\Z\Z[q^{-1}]\arrow[]{d}{\dot\Theta_{Bern}}[swap]{\simeq}\\
Z(\cH_{K_{\cF'}}(\Z[q^{-1}])) \arrow[]{u}{-*_{K_{\cF'}}{\bf 1}_{K_\cF}}[swap]{\simeq} &&Z(\cH_I(\Z[q^{-1}]))\arrow{ll}{-*_I{\bf 1}_{K_{\cF'}}}[swap]{\simeq}
\end{tikzcd}$$
\end{theorem}
\proof Similar to the proof of Lemma \ref{Zcompatibility}. \qed
\begin{remark}
Given Theorem \ref{twistedsatakeisomorphism}, 
we have, in particular, an inclusion $Z(\cH_{{I^{\flat}}}(\Z)) \subset {\dot{\Theta}_{\text{\emph{Bern}}}}({\mathcal{R}^{\dot{W}}})$.  
\end{remark}

\appendix
\addtocontents{toc}{\protect\setcounter{tocdepth}{1}} 

\section*{Appendix. Skew group rings}
This appendix is devoted to some properties of $G$-Galois extensions of commutative rings.
\subsection{$G$-Galois extensions of commutative rings}
\begin{definition}\label{sgring}
Let $S$ be a commutative ring and $G$ a finite group of ring automorphisms of $S$. Write $S*G$ for the free $S$-module with basis $\{g\in G\}$. 
We make $S*G$ into a ring as follows: for $s,s'\in S$ and $g,g'\in G$, define $(s \cdot g )(s' \cdot g' )= sg(s') \cdot gg'$.
The ring $S*G$ is usually called a skew group ring or a trivial crossed product.
\end{definition}
\begin{definition}
Let $R$ be a commutative ring. Two $R$-algebras homomorphisms $f,g \colon A \to B$ are said to be strongly distinct if for every nonzero idempotent $e \in B$ there exists $x \in A$ such that $f(x)e \neq g(x)e$.
\end{definition}
\begin{remark}
Strong distinctness is equivalent to Dedekind's lemma in the context of separable algebras. The elements of $G$ are all pairwise strongly distinct if and only if $G$ is free over $S$ in $\text{End}_R(S)$ \cite[Proposition 2.1]{FerPaq1997}. 
\end{remark}
\begin{theorem}\label{fundgaloiscom}
Let $S$ be a commutative ring extension and $G$ a finite group of ring automorphisms of $S$. Set $R:=S^G$. 
The ring $S$ is called a Galois extension of $R$ with Galois group $G$ or $S/R$ is a $G$-Galois ring extension, if one of the following equivalent statements is satisfied:
\begin{enumerate}
\item $S$ is separable over $R$ and the elements of $G$ are all pairwise strongly distinct.
\item There exist elements $x_i,y_i$ in $S$, $1\le i \le n$ for some integer $n$, such that $\sum_i x_ig(y_i)=\delta_{1,g}\mathbf{1}_S$, for all $g\in G$. The set $\{x_i,y_i\colon 1\le i \le n\}$ is called a $G$-Galois system for $S$. 
\item $S$ is a finitely generated projective $R$-module and the map $\varphi : S * G \to \text{End}_R(S)$, given by $\varphi(xg)(y) = xg(y)$, is an isomorphism of $S$-modules and $R$-algebras.
\item For any left $S * G$-module $M$ the map $\mu\colon S \otimes_R M^G\to M, (s\otimes m) \to sm$ is an isomorphism of $S$-modules.
\item The map $S\otimes_R S \to \prod_{g\in G} S, (x,y) \mapsto (xg(y))_{g\in G}$
is an isomorphism of $S$-algebras.
\item For each maximal ideal $\mathfrak{m}$ of $S$ and each $1\neq g \in G$ there exists $y \in S$ such that $g(y) - y \not\in \mathfrak{m}$.
\item $S \pi S= S* G$ where $\pi= \sum_{g\in G} g$.
\item $S*G$ is Azumaya over $R$.
\item $S*G$ is H-separable over $S$.
\end{enumerate}
When this is the case, $S$ is a maximal commutative $R$-subalgebra of $S*G$ with $(S*G)\otimes_R S\simeq \mathbb{M}_{|G|}(S) $ and $ S \otimes_R (S*G)\simeq \mathbb{M}_{|G|}(S)$. 
\end{theorem}
\begin{proof}
The equivalence of the first seven statement is \cite[Theorem 1.3]{CHR65}. 
For (8), first note that $S$ can be seen as an $(S;R)$-subbimodule of $S * G$ via $x \to x\pi$, hence the notation $S\pi S$ makes sense. Because $\pi g = \pi$ we see that $S \pi S$ is a two sided ideal of $S * G$, hence $S \pi S= S* G$ if and only if $\mathbf{1}_S \in S \pi S$, which is equivalent to the existence of $x_i,y_i$ in $S$, $1\le i \le n$ for some integer $n$, such that $\mathbf{1}_S=\sum_i x_i\pi y_i=\sum_{ g\in G}(\sum_{i} x_ig(y_i) )g$, i.e. $\sum_i x_ig(y_i)=\delta_{1,g}\mathbf{1}_S$, for all $g\in G$. 
For 8 and 9, this is \cite[Theorem 2]{Ikehata81} or \cite[Corollary 2]{AlfSze95}.
\end{proof}
\begin{remark}
If $S/R$ is $G$-Galois, then $\text{Rank}_RS=|G|$ \cite[Lemma 2.3]{FerPaq1997}.
\end{remark}
\subsection{Properties}
\begin{proposition}\label{Galprop}
For any $G$-Galois extension $S/R$:

(i) $S$ is faithfully flat over $R$.

(ii) $\text{Tr} \colon S \to R$ is surjective.

(iii) The $R$-submodule $R$ of $S$ is a direct summand of $S$.

(iv) For any $R$-algebra $T$, $T \otimes_R S/S$ is a
$G$-Galois extension iff $S/R$ is $G$-Galois.

(v) Let $S/R$ and $S'/ R$ be $G$-Galois. Then every $G$-equivariant $R$-algebra homomorphism $f: S \to S'$ of
$G$-Galois extensions is an isomorphism.
\end{proposition}
\begin{proof}
See \cite[Chapitre 0, \S 1]{Greither1992}.
\end{proof}
\subsection{An obvious corollary}
\begin{corollary}\label{obviouscor}
Let $T$ be a commutative ring with a unity $\mathbf{1}_T$ and $G$ be a finite group $G$ that acts faithfully on it by ring automorphisms. Let $S\subset T$ be a subring such that $\mathbf{1}_T=\mathbf{1}_S$. If $S/S^G$ is $G$-Galois then so is $T/T^G$.
\end{corollary}
\begin{proof}
Thanks to Theorem \ref{fundgaloiscom}, there exist elements $x_i,y_i$ in $S\subset T$, $1\le i \le n$, such that $\sum_i x_ig(y_i)=\delta_{1,g}\mathbf{1}_S$, for all $g\in G$. The corollary follows readily since $\mathbf{1}_T=\mathbf{1}_S$. 
\end{proof}
\subsection{Compressibility}
\begin{theorem}\label{compress}
If $A$ is an Azumaya algebra over a commutative ring $R$ then it is also compressible that is if $Z(e A e) = eZ(A)=eR$ for each idempotent $e$ of $A$.
\end{theorem}
\begin{proof}
This is \cite[Theorem 2]{ArmPar84}.
\end{proof}
\begin{example}
If $S/R$ is $G$-Galois then $S*G$ is compressible.
\end{example}
\bibliographystyle{amsalpha}
\bibliography{Reda_library2021}

\end{document}